\documentclass[11pt]{amsart}

\usepackage{amsgen, amsmath, amsfonts, amsthm, amssymb}

\usepackage{epsf}

\usepackage[all]{xy}

\usepackage[varg]{txfonts}
\usepackage[active]{srcltx}
\newtheorem{thm}{Theorem}[section]
\newtheorem{prop}[thm]{Proposition}

\newtheorem{lem}[thm]{Lemma}

\theoremstyle{definition}

\newtheorem{ass}[thm]{Assumptions}

\theoremstyle{remark}
\newtheorem{rmk}[thm]{Remark}

\numberwithin{equation}{section}

\newcommand\al{\alpha}
\newcommand\bt{\beta}
\newcommand\Gm{\Gamma}
\newcommand\Gf{\Gamma} % gamma function
\newcommand\gm{\gamma}
\newcommand\Dt{\Delta}
\newcommand\dt{\delta}
\newcommand\e{\varepsilon}
\newcommand\z{\zeta}

\renewcommand\th{\vartheta}
\renewcommand\k{\kappa}
\newcommand\Ld{\Lambda}
\newcommand\ld{\lambda}
\newcommand\x{\xi}

\newcommand\s{\sigma}

\newcommand\Ph{\Phi}
\newcommand\ph{\varphi}
\newcommand\ch{\chi}

\newcommand\ps{\psi}
\newcommand\om{\omega}
\newcommand\Om{\Omega}

\newcommand\CC{\mathbb{C}}
\newcommand\NN{\mathbb{N}}
\newcommand\QQ{\mathbb{Q}}
\newcommand\RR{\mathbb{R}}
\newcommand\ZZ{\mathbb{Z}}

\renewcommand\={\;=\;}
\newcommand\isdef{:=}
\renewcommand\setminus{\smallsetminus}

\newcommand\PSL{{\mathrm{PSL}}}
\newcommand\GL{{\mathrm{GL}}}
\newcommand\PGL{{\mathrm{PGL}}}
\newcommand\uhp{{\mathbb H}}
\newcommand\oh{{\mathrm{O}}}
\newcommand\re{\mathrm{Re}\,}
\newcommand\im{\mathrm{Im}\,}
\newcommand\sign{\,\mathrm{Sign}\,}
\newcommand\mf{{\mathsf{Maass}}}

\newcommand\proj[1]{\mathbb{P}^1_{#1}}
\newcommand\zssmb{\mathsf{Selb}}

\newcommand\zspec{\zssmb_{\mathrm{cu}}}

\newcommand\dscm{{\mathbf D}}
\newcommand\escm{{\mathbf E}}
\newcommand\scm{\mathbf{C}}
\newcommand\hypg[2]{{}_{#1}\!F_{\!#2}}
\newcommand\eis{\mathsf{Eis}}%{\mathbf{Q}}
\newcommand\poinc{\mathsf{Poinc}}%{\mathbf{P}}
\newcommand\four[1]{\mathbf{F}_{\!#1}}
\newcommand\R{\mathbf{R}}
\newcommand\mum{\mathbf{m}}
\newcommand\omm{\mathbf{w}}
\newcommand\lmm{\mathbf{n}}
\newcommand\Id{\mathbf{I}}
\newcommand\Jm{\mathbf{J}}
\newcommand\Pm{\mathbf{P}}
\newcommand\Um{\mathbf{U}}
\newcommand\Vm{\mathbf{V}}
\newcommand\Wm{\mathbf{W}}

\newcommand\txtfrac[2]{{\textstyle\frac{#1}{#2}}}

\newcommand\be{\begin{equation}}
\newcommand\ee{\end{equation}}
\newcommand\bad{\be\begin{aligned}}
\newcommand\badl[1]{\be\label{#1}\begin{aligned}}
\newcommand\ead{\end{aligned}\ee}
\newcommand\eadl{\ead}
\makeatletter
\newcommand\matc[4]{\left[ {#1\@@atop #3}{#2\@@atop #4}\right]}
\newcommand\matr[4]{\left[ {\hfill #1\@@atop\hfill #3}{\hfill
#2\@@atop\hfill #4}\right]}
\newcommand\rmatc[4]{\left( {#1\@@atop #3}{#2\@@atop #4}\right)}
\newcommand\rmatr[4]{\left( {\hfill #1\@@atop\hfill #3}{\hfill
#2\@@atop\hfill #4}\right)}
\makeatother

\newcommand\brdt{11.8}
\newcommand\grtt{\brdt\unitlength}
\newcounter{fact}
\renewcommand\thefact{\textbf{F\arabic{fact}}}
\long\def\fact#1\endfact{\smallskip\par\noindent\refstepcounter{fact}%
\textbf{\thefact. }{\it#1}\smallskip\par}

\begin{document}
\title[Zeros of the Selberg zeta-function]{Perturbation of zeros of
the Selberg zeta-function for $\Gm_0(4)$}

\author{Roelof Bruggeman}
\address{Mathematisch Instituut Universiteit Utrecht, Postbus 80010,
NL-3508 TA Utrecht, Nederland}
\email{r.w.bruggeman@uu.nl}

\author{Markus Fraczek}
\address{Institut f\"ur Theoretische Physik, Technische Universit\"at
Clausthal, Arnold-Sommerfeld-Stra\ss e~6, 38678 Clausthal-Zellerfeld,
Deutschland}
\email{Dieter.Mayer@tu-clausthal.de}

\author{Dieter Mayer}
\address{Institut f\"ur Theoretische Physik, Technische Universit\"at
Clausthal, Arnold-Sommerfeld-Stra\ss e~6, 38678 Clausthal-Zellerfeld,
Deutschland}
\email{markus.fraczek@tu-clausthal.de}

\begin{abstract}We study the asymptotic behavior of zeros of the
Selberg zeta-function for the congruence subgroup $\Gm_0(4)$ as a
function of a one-parameter family of characters tending to the
trivial character. The motivation for the study comes from
observations based on numerical computations. Some of the observed
phenomena lead to precise theorems that we prove and compare with the
original numerical results.
\end{abstract}


\maketitle

\section*{Introduction} This paper presents computational and
theoretical results concerning zeros of the Selberg zeta-function.
The second named author shows in~\cite{Fr} that it is possible to use
the transfer operator to compute in a precise way zeros of the
Selberg zeta-function, and carries out computations for $\Gm_0(4)$
for a one-parameter family of characters. The results show how zeros
of the Selberg zeta-function follow curves in the complex plane
parametrized by the character. In this paper we observe several
phenomena in the behavior of the zeros as the character approaches
the trivial character. Motivated by these observations we formulate a
number of asymptotic results for these zeros, and prove these results
with the spectral theory of automorphic forms. These asymptotic
formulas predict certain aspects of the behavior of the zeros more
precisely than we guessed from the data. We compare these predictions
with the original data. In this way our paper forms an example of
interaction between experimental and theoretical mathematics.

Selberg shows in~\cite{Se90} that for the group $\Gm_0(4)$ and a
specific one-parameter family of characters, the Selberg
zeta-function not only has countably many zeros on the central line
$\re\bt=\frac12$, but has also many zeros in the spectral plane
situated on the left of the central line, the so-called resonances.
Both type of zeros change when the character changes. As the
character approaches the trivial character the resonances tend to
points on the lines $\re \bt=\frac12$ or $\re\bt=0$, or to the
 non-trivial zeros of $\z(2\bt)=0$, so presumably to points on the
line $\re \bt=\frac14$. Many of these zeros have a real part tending
to $-\infty$ as the parameter of the character approaches other
specific values.

In this paper we focus on zeros on or near the central line
$\re \bt=\frac12$, and consider their behavior as the character
approaches the trivial character.

In Section~\ref{sect-res} we describe observations in the results of
the computations. We state the theoretical results, and compare
predictions with the observations in the computational results. The
approach of Fraczek is based on the use of a transfer operator, which
makes it possible to consider eigenvalues and resonances in the same
way. See \S7.4 in~\cite{Fr}.

In Section~\ref{sect-prfs} we give a short list of facts from the
spectral theory of automorphic forms, and give the proofs of the
statements in \S\ref{sect-res}.

In Section~\ref{sect-spth} we recall the required results from
spectral theory, applied to the group~$\Gm_0(4)$. Not all of the
facts needed in~\S\ref{sect-prfs} are readily available in the
literature, some facts need additional arguments in the present
situation. The spectral theory that we apply uses Maass forms with a
bit of exponential growth at the cusps. In this way it goes beyond
the classical spectral theory, which considers only Maass forms with
at most polynomial growth. We close \S\ref{sect-spth} with some
further remarks on the method and on the interpretation of the
results.
\medskip

The first named author thanks D.\,Mayer for several invitations to
visit Clausthal, and thanks the Volkswagenstiftung for the provided
funds.

\section{Discussion of results}\label{sect-res}
The congruence subgroup $\Gm_0(4) $ consists of the elements
$\matc abcd\in \PSL_2(\ZZ)$ with $c\equiv 0\bmod 4$. By $\matc abcd$
we denote the image in $\PGL_2(\RR)$ of $\rmatc abcd\in \GL_2(\RR)$.
The group $\Gm_0(4)$ is free on the generators $\matc1101$ and
$\matr10{-4}1$. A family $\al\mapsto \ch_\al$ of characters
parametrized by~$\al \in \CC \bmod \ZZ$ is determined by
\be \ch_\al \biggl( \matc1101\biggr) \= e^{2\pi i \al}\,,\quad
\ch_\al\biggl( \matr10{-4}1\biggr)\= 1\,. \ee
The character is unitary if $\al\in \RR\bmod \ZZ$.
This is the family
of characters of $\Gm_0(4)$ used in \cite{Fr}. See especially
\S8.1.3. Up to conjugation and differences in parametrization, this
is the family of characters considered by Selberg in \S3
of~\cite{Se90}, and by Phillips and Sarnak in \cite{PS92}
and~\cite{PS94}.

For a unitary character $\ch$ of a cofinite discrete group~$\Gm$ the
Selberg zeta-function $Z(\Gm,\ch;\cdot)$
is a meromorphic function on~$\CC$ with both geometric and spectral
relevance. As a reference we mention \cite{He83}, Chapter~X, \S2
and~\S5. One may also consult \cite{Fi}, or Chapter~7 of~\cite{Ve90}.

The geometric significance is clear from the product representation
\be Z(\Gm,\ch;\bt) \= \prod_{k\geq 0}\prod_{\{\gm\}} \bigl( 1-\ch(\gm)
\, e^{-(\bt+k)\ell(\gm)}\bigr) \qquad(\re \bt>1)\,, \ee
where $k$ runs over integers and $\gm$ over representatives of
primitive hyperbolic conjugacy classes. By $\ell(\gm)$ is denoted the
length of the associated closed geodesic. This geometric aspect is
used in the investigations in~\cite{Fr}. By means of a transfer
operator, Fraczek is able to compute zeros of the Selberg zeta
function for $\Gm_0(4)$ as a function of the character $\ch_\al$.

\begin{figure}
\begin{center}
\renewcommand\brdt{5.6}
\makebox[\grtt]{\epsfxsize=\grtt\epsffile{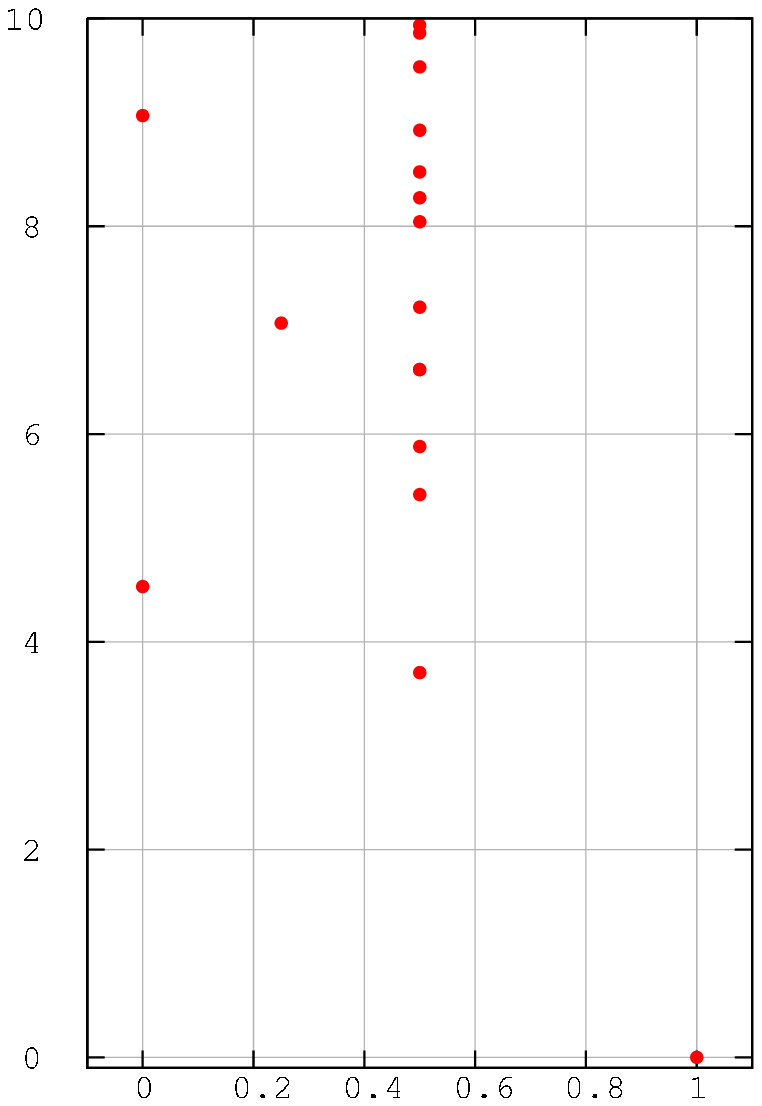}}
\hspace{\fill}
\makebox[\grtt]{\epsfxsize=\grtt\epsffile{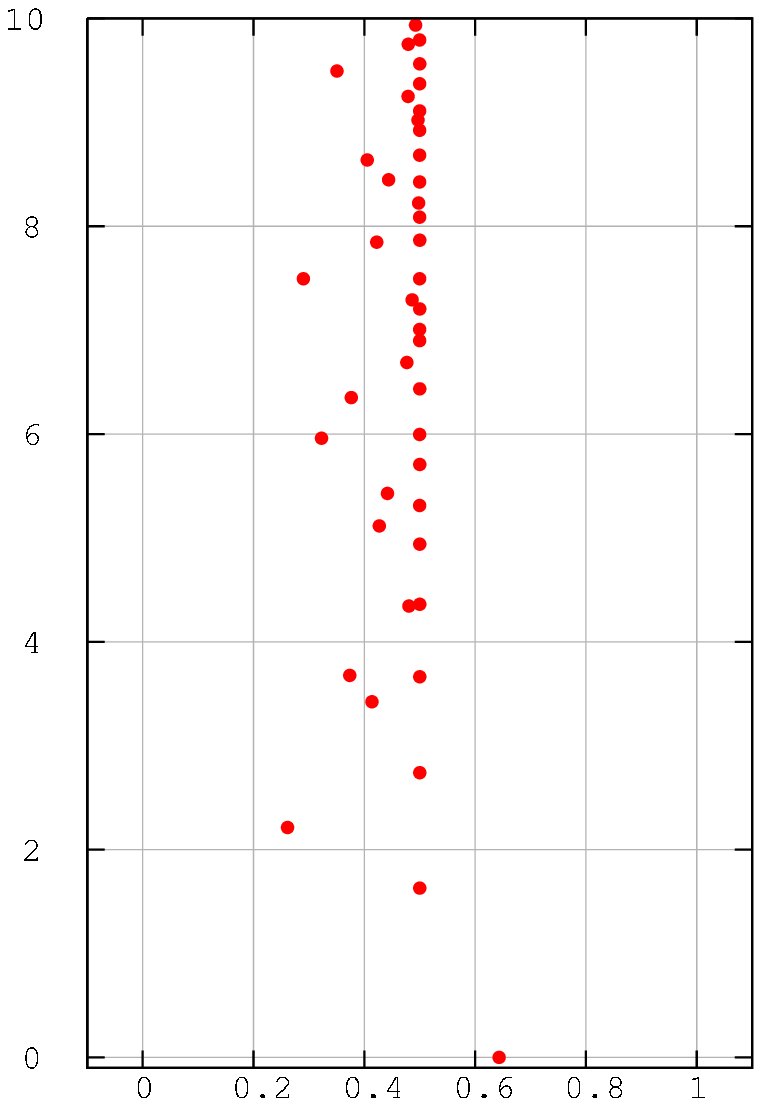}}
\end{center}
\caption{Zeros of the Selberg zeta-function
$Z(\al,\cdot) = Z(\ch_\al,\cdot)$
in the region $[0,1]\times i[0,10]$ of the spectral plane. On the left
for $\al=0$, on the right for $\al=\frac1{10}$.}\label{fig-zeros}
\end{figure}
Via the Selberg trace formula, the zeros of function
$Z(\Gm,\ch;\cdot)$
are related to automorphic forms. This is the relation that we use in
Sections \ref{sect-prfs} and~\ref{sect-spth} for our theoretical
approach.

We denote by $Z(\al,\bt)$ the Selberg zeta-function
$\bt\mapsto Z\bigl(\Gm_0(4),\ch_\al;\bt)$ for $\al\in \RR$. We
consider its zeros in the region $\im \bt>0$.\medskip

For each value of~$\al\in \RR$ the zeros of $Z(\al,\cdot)$ form a
discrete set. In Figure~\ref{fig-zeros} we give the non-trivial zeros
of $ Z(\al,\cdot)$ in the region $[0,1]\times i[0,10]$ in the
$\bt$-plane for the trivial character, $\al=0$ (Table~D.1
in~\cite{Fr}), and the nearby value $\al=\frac1{10}$
(interpolation of data discussed in \S8.2 of~\cite{Fr}). In the
unperturbed situation, $\al=0$, the zeros to the left of the central
line, the \emph{resonances}, are known to occur at the zeros of
$\z(2\bt)$, of which only one falls within the bounds in the figure.
There are also zeros at points $\frac{\pi  i \ell}{\log 2}$ with
$\ell\in \ZZ$.

We call zeros $\bt$ of $Z(\al,\cdot)=0$ with $\re \bt=\frac12$
\emph{eigenvalues}, although we will see in \S\ref{sect-mf} that
$\bt-\bt^2$ qualifies better for that name. The lowest unperturbed
eigenvalue is $.5 + 3.70331 \, i$. Perturbation to $\al=\frac1{10}$
gives a more complicated set of zeros, many of which are eigenvalues.

In~\cite{Fr}, \S8.2, it is explained how zeros are followed as a
function of the parameter. They follow curves that either stay on the
central line, or move to the left of the central line and touch the
central line only at some points.

\subsection{Curves of eigenvalues}\label{sect-cei}To exhibit curves of
zeros of the Selberg zeta-func\-tion on the central line
$\re\bt=\frac12$ we plot $\im \bt$ as a function of~$\al$.
\begin{figure}
\begin{center}
\renewcommand\brdt{9.2}
\makebox[\grtt]{\epsfxsize=\grtt\epsffile{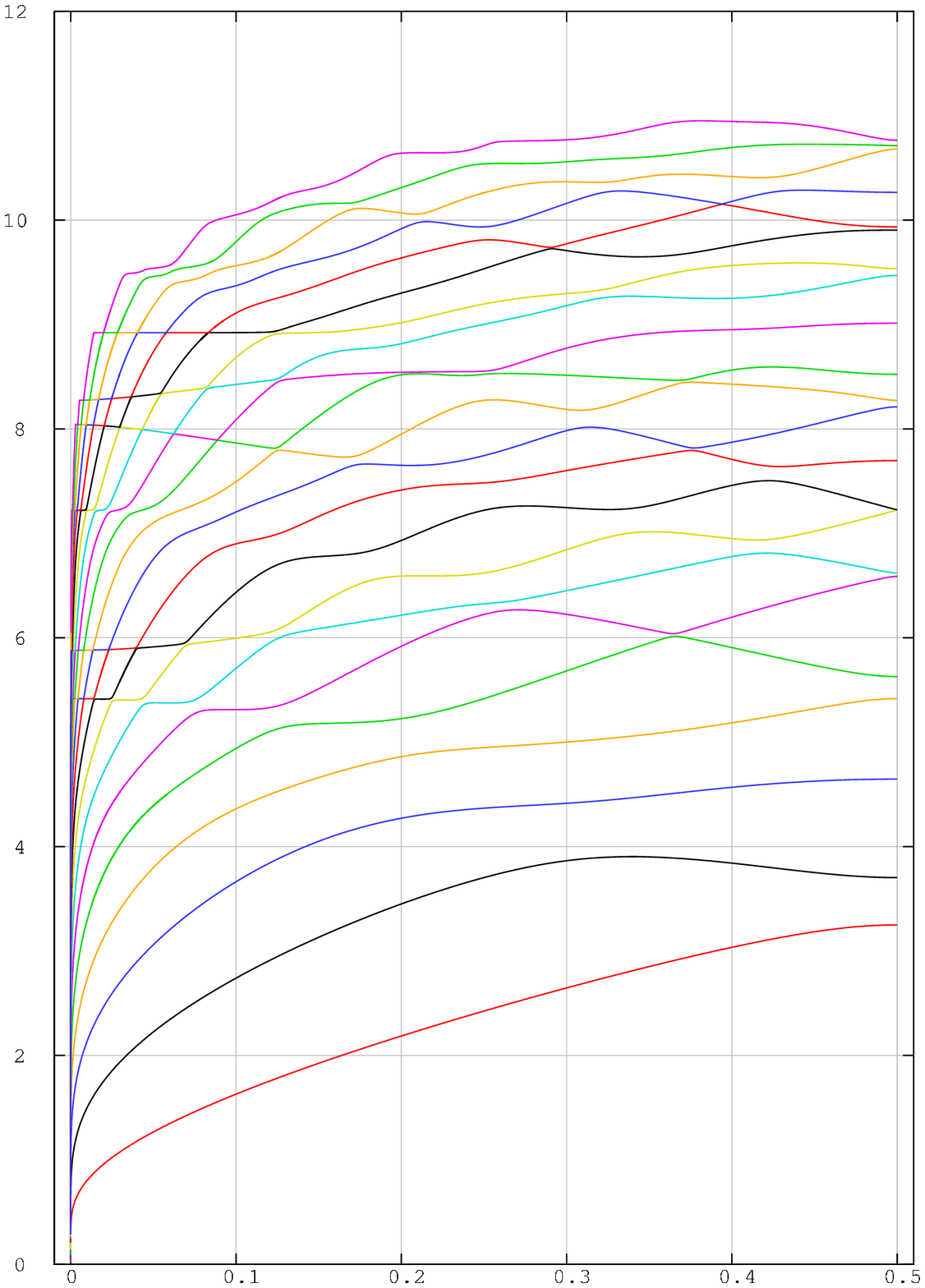}}
\end{center}
\caption{Zeros of $Z(\al,\bt)$ with $0<\al\leq \frac12$, $\bt\in
\frac12+i(0,10)$. Horizontal: $\al$; vertical:
$\im \bt$.}\label{fig-spectr}
\end{figure}
The curves in Figure~\ref{fig-spectr} were obtained in~\cite{Fr}, by
first determining for the arithmetical cases
$\al\in\bigl\{\frac18,\frac14,\frac38,\frac12\}$ all zeros in a
region of the form $\frac12+i[0,T]$. Here we display only those zeros
which stay on the central line $\re \beta =\frac12$. The computations
suggest that all these zeros go to $\bt=\frac12$ as
$\al\downarrow 0$, along curves that are almost vertical for small
values of~$\alpha$. Our first result confirms this impression, and
makes it more precise:
\begin{thm}\label{thm-ei-0}
For each integer $k\geq 1$ there are $\z_k\in (0,1]$ and a
real-analytic map $\tau_k:(0,\z_k)\rightarrow (0,\infty)$ such that
$Z\bigl( \al,\frac12+\nobreak i \tau_k(\al)\bigr)=0$ for all
$\al\in (0,\z_k)$.

For each $k\geq 1$
\be\label{as-0} \tau_k(\al) \= \frac{\pi k}{-\log(\pi^2\al/4)} +
\oh\Bigl( \frac{k^2}{(\log\al)^3}\Bigr)\qquad \text{as }\al\downarrow
0\,. \ee
\end{thm}
So there are infinitely many curves of zeros going down as
$\al\downarrow0$, and for each curve the quantity
$-\pi^{-1}\, \im \bt \, \log\frac{\pi^2\al}4$ tends to an integer.
Figure~\ref{fig-asm0} shows that the computational data confirm the
asymptotic behavior in~\eqref{as-0}.
\begin{figure}
\begin{center}
\renewcommand\brdt{7.4}
\makebox[\grtt]{\epsfxsize=\grtt\epsffile{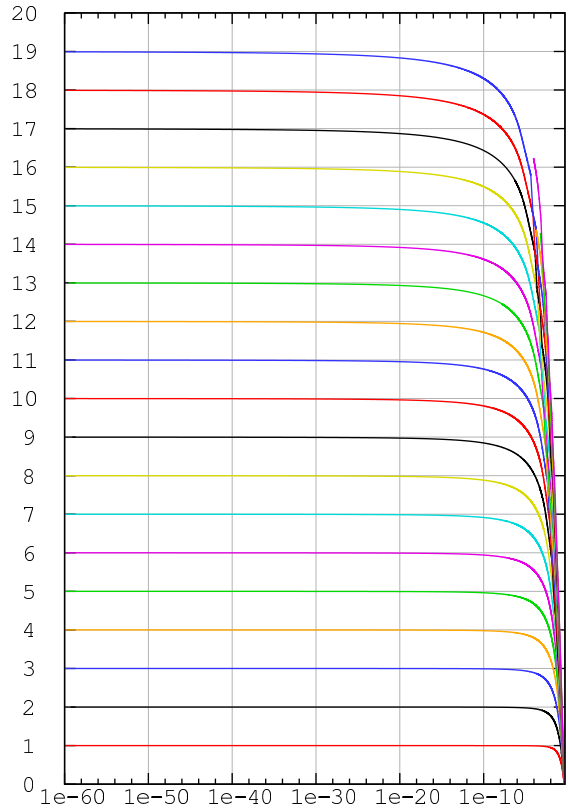}}
\end{center}
\caption{On the vertical axis the quantity
$-\pi^{-1} \, \im \bt \, \log\frac{\pi^2\al}4$ is given for those
curves in the data set of~\protect\cite{Fr} for which $\bt$ goes to~$\frac12$
along the central line as $\al\downarrow0$. The horizontal axis gives
$\al$ on a logarithmic scale. So $\al\downarrow0$ means going to the
left in the graph.}\label{fig-asm0}
\end{figure}
The theorem does not state that all zeros of the Selberg zeta-function
on the central line occur in these families. The spectral theory of
automorphic forms allows the possibility that there are other
families.
\smallskip

Figure~\ref{fig-spectr} shows also a regular behavior near many parts
on the central line. By theoretical means we obtain:
\begin{thm}\label{thm-ei-ln}Let $I\subset (0,\infty)$ be a bounded
closed interval such that the interval $\frac12+i\, I$ on the central
line does not contain zeros of the unperturbed Selberg zeta-function
$Z(0,\cdot)$.

Let for $t\geq 0$
\be\label{ph-def} \ph_-(t) \= \arg
\pi^{2it}\;\frac{2^{1+2it}-1}{2^{1-2it}-1}\;\frac{\z(-2it)}{\z(2it)}\,,
\ee
with $\z$ the Riemann zeta-function, be the continuous choice of the
argument that takes the value~$0$ for $t=0$.

For all sufficiently large $k\in \ZZ$ there is a function
$a_k:I\rightarrow(0,1)$ inverting on~$I$ the function $\tau_k$ of the
previous theorem: $\tau_k\bigl( a_k(t))=t$ for $t\in I$. Uniformly
for $t\in I$ we have
\be \label{alk-as}
a_k(t) \= \pi^{-1} e^{\ph_-(1/2+it)/2t+\pi k_I/2t}\, e^{-\pi k /t}
\,\Bigl( 1+ \oh\bigl(e^{-\pi k/t}\bigr)\Bigr)
\qquad(k\rightarrow\infty)\,, \ee
for some $k_I \in \ZZ$.
\end{thm}
The theorem gives an assertion concerning the behavior of the zeros on
the central line at given positive values $t$ of $\im \bt$, and
describes the asymptotic behavior as the parameter~$k$ from the
previous theorem tends to~$\infty$. To compare this prediction with
the data we determine by interpolation the value $a_k(t)$ for the
curves used in Figure~\ref{fig-asm0}. The theorem predicts that
\be\label{kI-q} \log\frac{ a_k(t)}\pi + \frac{\pi k}t -
\frac{\pi_-(1/2+it)}{2t} \,\approx\, \frac{\pi k_I}{2t} +
\oh\bigl(e^{-\pi k/t}\bigr)
\,.\ee
We used the data for the curves with $1\leq k \leq 19$ to compute an
approximation of the quantity on the left in~\eqref{kI-q}. 
We consider this as a vector in $\RR^{19}$, with coordinates
parametrized by~$k$, and project it orthogonally on the line spanned
by $(1,1,\ldots,1)$ with respect to the scalar product
$(x,y) = \sum_{k=1}^{19} k^{20}\, x_k\, y_k$, and thus obtain
approximations of $\pi k_I/2t$, which are given in
Table~\ref{tab-kI-}.
\begin{table}[ht]{\footnotesize
\[
\begin{array}{|c|ccccccccc|}\hline
t:&1&2&3&4&5&6&7&8&9\\
& -2.000& -2.000& -2.000& -2.000& -2.000& 2.000& 2.000& 4.015& 10.17\\
\hline
\end{array}
\]} \caption{Approximation of $k_I$ in~\protect\eqref{kI-q}.}\label{tab-kI-}
\end{table}

Figure~\ref{fig-kI-} illustrates the approximation of $k_I$ for more
values of $t$ between $0.05$ and $9.00$.
\begin{figure}[ht]
\begin{center}
\makebox[\grtt]{\epsfxsize=\grtt\epsffile{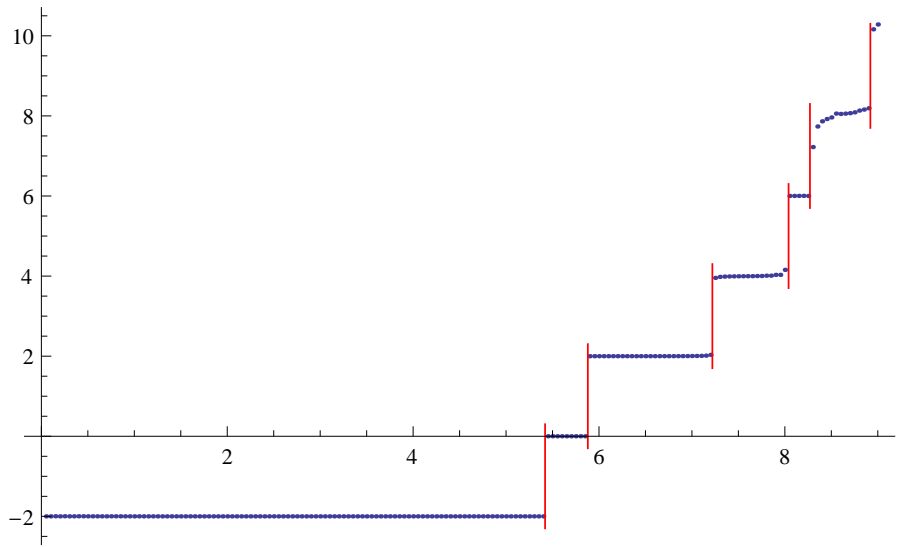}}
\end{center}
\caption{Approximation of $k_I$ for $t\in (0,9]\cap\frac1{20}\ZZ$. The
vertical lines indicate the position of the unperturbed odd
eigenvalues (from Table~D.1 in~\protect\cite{Fr}).}\label{fig-kI-}
\end{figure}
The intervals $I$ in the theorem should not contain zeros of the
unperturbed Selberg zeta-function. Actually, the proofs will tell us
that not all unperturbed zeros are not allowed to occur in~$I$, only
those associated to Maass cusp forms that are odd for the involution
induced by $z\mapsto \bar z/(2\bar z-\nobreak 1)$. We have indicated
the corresponding $t$-values by vertical lines in
Figure~\ref{fig-kI-}.\footnote{The comparison of the theoretically
obtained asymptotic formulas with the data from~\cite{Fr} has been
carried out mainly with \texttt{Pari/gp}, \cite{Pari}; for some of
the pictures we used Mathematica. }

\subsubsection{Avoided crossings. }\label{sect-ac}If one looks at the
graphs of the functions $\tau_k$ in Figure~\ref{fig-spectr} (ignoring
the coloring) it seems that the graphs intersect each other.\footnote{This 
may seem not to be 
completely true in the  the posting on arXiv, probably due to the lower
resolution that we had to use.}
In the
 enlargement in Figure~\ref{fig-ac} most of these intersections turn
out to be no intersections after all. This is the phenomenon of
\emph{avoided crossings} that is known to occur at other places as
well; for instance in the computations of Str\"omberg in~\cite{Str}.
\begin{figure}
\begin{center}
\renewcommand\brdt{11.2}
\makebox[\grtt]{\epsfxsize=\grtt\epsffile{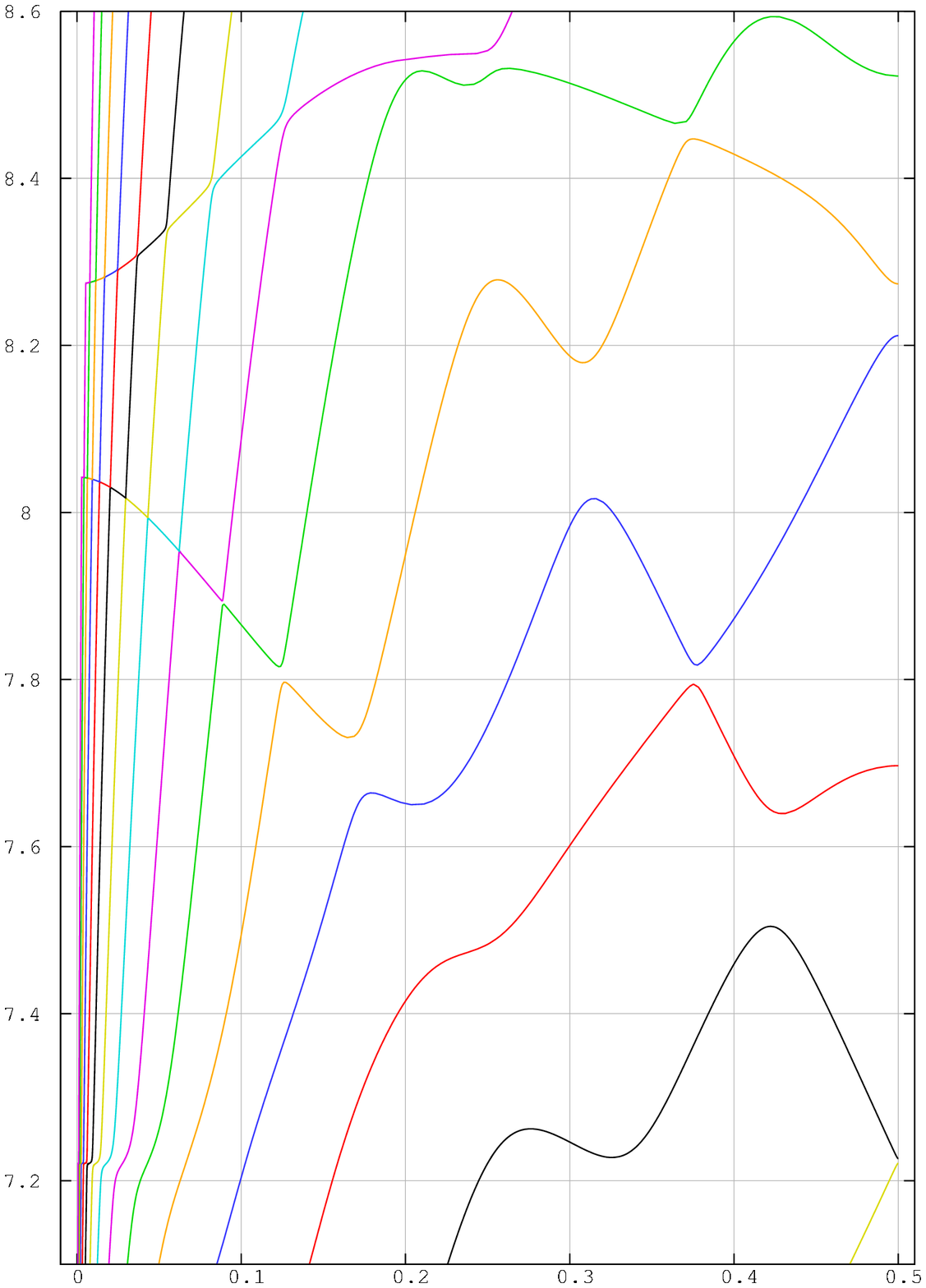}}
\end{center}
\caption{Enlargement of a subregion in Figure~\protect\ref{fig-spectr}. Zeros
of $Z(\al,\bt)$ with $0<\al\leq \frac12$, $\bt\in
\frac12+i(7.1,8.6)$.}\label{fig-ac}
\end{figure}
In the computations for~\cite{Fr} care was taken to decrease the step
length whenever curves of zeros approached each other. In all cases
this indicated that the curves of zeros do not intersect each other.
Theoretically, we know that no intersections occur for the zeros
moving along the central line in the region indicated in
Lemma~\ref{lem-S-}.

In Remark~\ref{rmk-av} we will discuss that for some of the $t_0>0$
for which $Z(0,\frac12+\nobreak it_0)=0$ there may be a curve through
$t_0$ in the $(\al,t)$-plane such that $\tau_k'(\al)$ is relatively
small for the value of $\al$ for which the graph of $\tau_k$
intersects the curve. We show this only under some simplifying
assumptions formulated in Proposition~\ref{prop-av}.

\subsection{Curves of resonances}\label{sect-cres} The zeros of
$Z(\al,\bt)$ with $\bt$ to the left of the central line are more
difficult to depict, since they form curves in the three-dimensional
set of $(\al,\bt)$ with $\al\in (0,1)$ and $\bt\in \CC$.
\begin{figure}
\begin{center}
\renewcommand\brdt{12.2}
\makebox[\grtt]{\epsfxsize=\grtt\epsffile{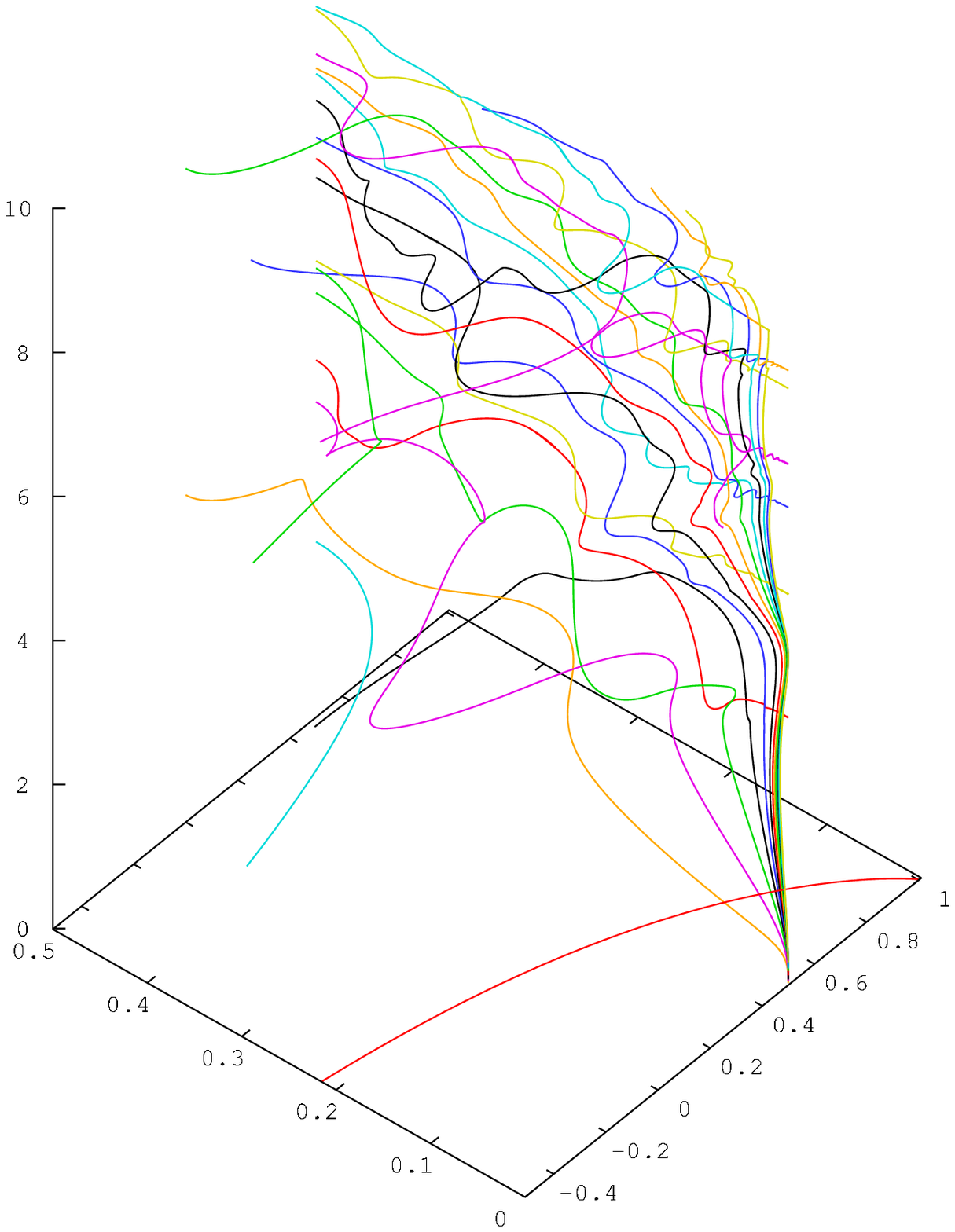}}
\end{center}
\caption{Zeros with $\re\bt<\frac12$, $\al>0$, in a $3$-dimensional
graph. On the vertical axis $\im \bt$ runs from $0$ to~$10$. The
`horizontal' axis running to the left gives the
coordinate~$\al\in \bigl(0,\frac12\bigr)$, and the `horizontal' axis
to the right gives
$\re\bt \in \bigl(-\frac12,1\bigr)$.}\label{fig-3dim}
\end{figure}

Figure~\ref{fig-3dim} gives a three-dimensional picture. We see one
curve in the horizontal plane, corresponding to $\im\bt=0$. In this
paper we do not consider real zeros of the Selberg zeta-function.
Many curves originate for $\al\approx 0$ from $\bt=\frac12$ and move
upwards in the direction of increasing values of $\im \bt$. On the
right we see also a few more curves that wriggle up starting from
higher values of $\im \bt$.

In Figure~\ref{fig-res-projC}
\begin{figure}
\begin{center}
\renewcommand\brdt{12.4}
\makebox[\grtt]{\epsfxsize=\grtt\epsffile{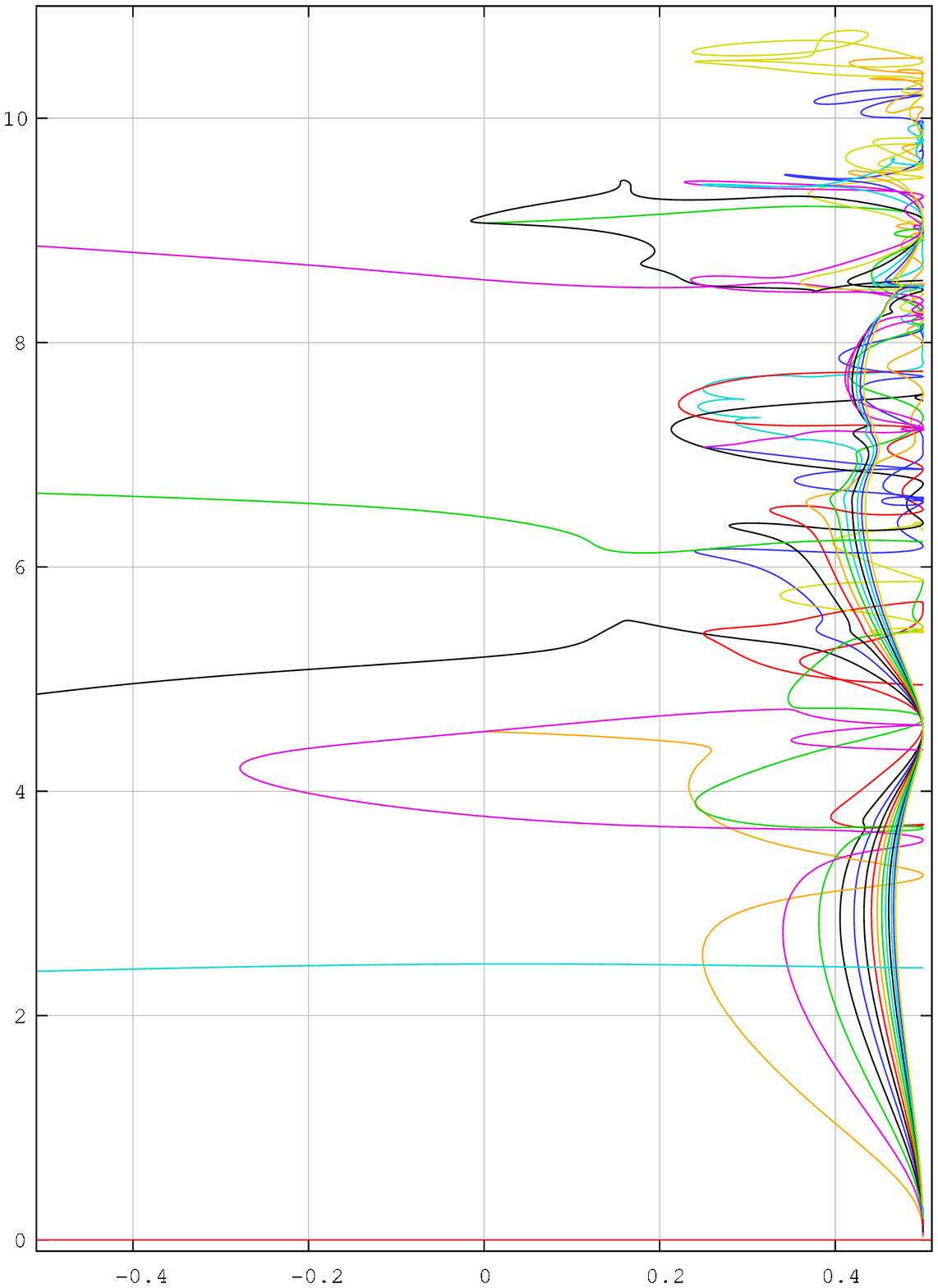}}
\end{center}
\caption{Curves of resonances in the complex plane. The vertical axis
carries $\im \bt$ and the horizontal one
$\re \bt$.}\label{fig-res-projC}
\end{figure}
we project the curves onto the $\bt$-plane. In this projection we
cannot see the $\al$-values along the curve. We see again the curves
starting at $\bt=\frac12$. Many of them seem to touch the central
line at higher values of~$\im\bt$. The curves that start higher up
are not well visible in this projection.
\smallskip

We can confirm certain aspects of these computational results by
theoretical results. We start with the behavior of the resonances
near $\bigl(0,\frac12\bigr)$.
\begin{thm}\label{thm-res0}There are $\e_1,\e_2,\e_3>0$ such that all
$(\al,\bt)$ that satisfy $Z(\al,\bt)=0$, $\al\in (0,\e_1]$,
$\frac12-\e_2 \leq \re \bt<\frac12$, and $0<\im\bt\leq \e_3$ occur on
countably many curves
\[ t\mapsto \bigl(\al_k(t),\s_k(t)+it\bigr)\qquad(0<t\leq \e_3)\,,\]
parametrized by integers $k\geq 1$. The functions $\al_k$ and $\s_k$
are real-analytic. The values of $\s_k$ are in
$\bigl[\frac12-\nobreak \e_2,\frac12\bigr)$. For each $k\geq 1$ the
map $\al_k$ is strictly increasing and has an inverse $t_k$ on some
interval $(0,\z_k]\subset(0,\e_1]$. As $\al\downarrow0$ we have
\begin{align}
\label{as+0t}
t_k(\al) &\= \frac{\pi k}{|\log\pi^2\al|}
+\oh\Bigl(\frac1{|\log\pi^2\al|^4}\Bigr)\,,\\
\label{as+0sg}
\s_k\bigl(t_k(\al)\bigr)&\= \frac12-\frac{2(\pi k \log
2)^2}{|\log\pi^2\al|^3}
+ \oh\Bigl(\frac1{|\log\pi^2\al|^4}\Bigr)\,.
\end{align}
\end{thm}
The theorem confirms that there are many curves of resonances that
approach the point $(0,\frac12)$ almost vertically as
$\al\downarrow0$. To check this graphically, one may consider the
three quantities
\bad\label{k1-3}
k_1(\al,\bt)&\= \sqrt{\bigl(\frac12-\re \bt\bigr) \,
|\log\pi\al|^3/2\pi^2(\log 2)^2}\,,\\
k_2(\al,\bt)&\= \im \bt \Bigl(
\frac\pi{|\log\pi\al|}+\frac{\pi\,\log\pi}{|\log\pi\al|^2}+
\frac{\pi(\log\pi)^2}{|\log\pi\al|^3}\Bigr)^{-1}\,,\\
k_3(\al,\bt)&\= \frac{2(\log 2)^2\, (\im\bt)^3}{\pi\bigl( \frac12-\re
\bt)}\,, \ead
which should each approximate the ``real'' $k$ as $\al\downarrow0$.
Figure~\ref{fig-k3} illustrates $k_3$. Note that the horizontal scale
gives $\al $ logarithmically, so the curves should tend to integers
when going to the left. Although the data go down to $\al=10^{-60}$
the limit behavior is not clear in the picture.
\begin{figure}
\begin{center}
\renewcommand\brdt{7.4}
\makebox[\grtt]{\epsfxsize=\grtt\epsffile{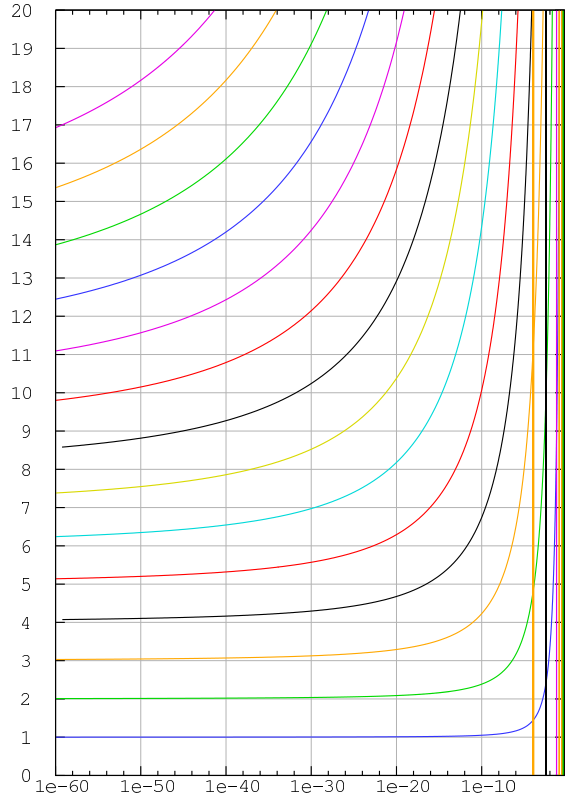}}
\end{center}
\caption{The quantity $k_3\bigl(\al,\bt\bigr)$ in~\protect\eqref{k1-3}
(vertical)
as a function of $\al$ (horizontal) on a logarithmic scale, for
several curves of resonances in~\protect\cite{Fr}.}\label{fig-k3}
\end{figure}

In a non-graphical approach we approximate the limit value by finding
the coefficients in a least square approximation
\be k_j(\al,\bt) \;\approx\; \sum_{\ell=1}^{5}
\frac{c_\ell}{|\log\al|^\ell} \ee
over the $500$ data points with lowest values of~$\al$ on each of the
curves of resonances going to $\bigl(0,\frac12\bigr)$. The
coefficient $c_0$ should be an approximation of the limit. The data
are in Table~\ref{tab-kvalues}.
\begin{table}
{\small \begin{tabular}{|l|l|l|l|}
\hline
file name & $ k_{R,1}$ & $ k_{R,2}$ & $ k_{R,3}$ \\
\hline
S6 & \hbox{\ \ }1.00000077911 & \hbox{\ \ }1.00000012847 & \hbox{\ \
}0.999999581376 \\
S8 & \hbox{\ \ }1.99993359598 &\hbox{\ \ }2.00000692659 & \hbox{\ \
}2.00002268168 \\
S7 & \hbox{\ \ }2.99969624243 & \hbox{\ \ }2.99999718822 & \hbox{\ \
}3.00014897712 \\
S1-1 & \hbox{\ \ }3.98453704088 & \hbox{\ \ }3.99972992678 & \hbox{\ \
}4.00794474117 \\
S3 & \hbox{\ \ }4.97680447599 & \hbox{\ \ }4.99941105921 & \hbox{\ \
}5.01294455984 \\
S16 & \hbox{\ \ }5.9998899911 & \hbox{\ \ }5.99996250451 & \hbox{\ \
}6.00002349805 \\
S19 & \hbox{\ \ }6.96419831285 & \hbox{\ \ }6.99744783611 & \hbox{\ \
}7.03093500105 \\
S23 & \hbox{\ \ }8.09149912519 & \hbox{\ \ }7.9904003854 & \hbox{\ \
}8.02661700448 \\
S26 & \hbox{\ \ }9.03588236648 & \hbox{\ \ }8.99230468926 & \hbox{\ \
}9.04177665474 \\
S11 & 10.1247216267 & \hbox{\ \ }9.99015681093 & 10.0254416847 \\
S34 & 11.4775951755 & 10.9860223595 & 10.9219590835 \\
S43 & 11.93647544 & 11.9959050481 & 12.0661799032
\\
S38 & 13.3059450769 & 12.9885322374 & 12.9785198908 \\
S46 & 14.0371581712 & 13.9931918442 & 14.0160494173 \\
\hline
\end{tabular}}\medskip
\caption{Least square approximation of the limits $k_{R,j}$ of the
quantities $k_j$ in~\protect\eqref{k1-3}. (The first column refers to the
naming in~\protect\cite{Fr} of the curves of zeros.)}\label{tab-kvalues}
\end{table}
This gives a reasonable confirmation that \eqref{as+0t} and
\eqref{as+0sg} describe the asymptotic behavior of the data. We also
experimented with direct least square approximation of the
coefficients of the expansion of $t_k(\al)$ and
$\s_k\bigl(t_k(\al)\bigr)$ as a function of $\frac1{|\log\pi^2\al|}$.
The results from the approximation of $\s_k\bigl(t_k(\al)\bigr)$ were
less convincing than those in Table~\ref{tab-kvalues}.
\medskip

The next result concerns curves higher up in the $\bt$-plane.
\begin{thm}\label{thm+I}Let $I$ be a bounded interval in $(0,\infty)$
such that $\{0\}\times\left( \frac12+i I\right)$ does not contain
zeros of the unperturbed Selberg zeta-function $Z(0,\cdot)$.

There are countably many real-analytic curves of resonances of the
form
\[ t\mapsto \bigl( \al_k(t), \s_k(t)+it\bigr)\qquad \text{ with }t\in
I\,,\]
parametrized by integers $k\geq k_1$ for some integer~$k_1$. Uniformly
for $t\in I$ we have the relations
\begin{align}
\label{asal+t}
\al_k(t) &\= \frac1\pi\, e^{A(1/2+it)/2t}\, e^{-\pi k/t}\,\Bigl(1+
\oh\bigl(\frac1k\bigr)\Bigr)\,,\\
\label{assg+t}
\frac12-\s_k(t) &\= \frac t{2\pi k}\, M\bigl(\frac12+it)
+ \oh \Bigl( \frac 1{k^2}\Bigr)\,,
\end{align}
as $k\rightarrow\infty$, where $M(t)$ and $A(t)$ are the real and
imaginary part of a continuous choice of
\[ \log\Bigl( \pi^{2it} \,\bigl( 2^{1+2it}-1\bigr)\,
\frac{\z(-2it)}{\z(2it)}\Bigr)\,.\]
\end{thm}
If we would use a standard choice of the argument the function $A$
would have discontinuities. The parameter $k$ is determined by the
choice of the branch of the logarithm. The theory does not provide
us, as far as we see, a way to relate the numbering of the branches
for different intervals~$I$.

We compare relation~\eqref{asal+t} with the data files in the same way
as we used for Theorem~\ref{thm+I}. The theorem says that the
relation holds for some choice $A$ of the argument. We picked a
continuous choice. Then we expect a factor $e^{\pi k_I/t}$
in~\eqref{asal+t} with $k_I$ constant on intervals as indicated in
the theorem. This leads to Figure~\ref{fig-evclk1}.
\begin{figure}
\begin{center}
\renewcommand\brdt{7.4}
\makebox[\grtt]{\epsfxsize=\grtt\epsffile{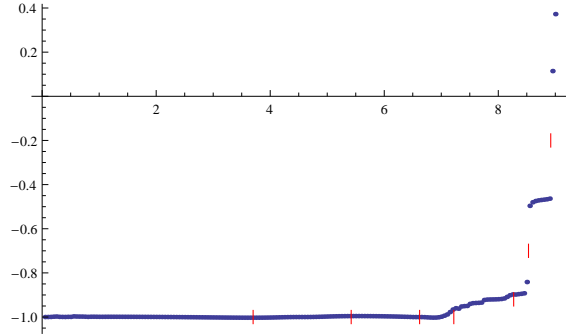}}
\end{center}
\caption{Approximation of $k_I$ for $t\in (0,9]\cap\frac1{20}\ZZ$. The
vertical lines indicate the position of unperturbed even eigenvalues
(from Table~D.1 in~\protect\cite{Fr}).}\label{fig-evclk1}
\end{figure}
The function~$M$ in~\eqref{assg+t} has the simple form
$M(t) = \log\bigl|2^{1+2it}-\nobreak 1\bigr|$. Figure~\ref{fig-evclM}
gives this function and the approximation of it based
on~\eqref{assg+t}.
\begin{figure}
\begin{center}
\renewcommand\brdt{7.4}
\makebox[\grtt]{\epsfxsize=\grtt\epsffile{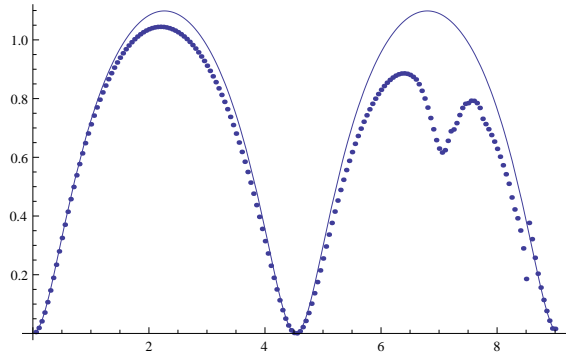}}
\end{center}
\caption{The function $t\mapsto \log\bigl|2^{1+2it}-\nobreak 1\bigr|$
and its approximation based on~\protect\eqref{assg+t}.}\label{fig-evclM}
\end{figure}
Figures \ref{fig-evclk1} and~\ref{fig-evclM} show differences that we
do not understand well.
\medskip

In Figure~\ref{fig-res-projC} it seems that at
$\bt \approx \frac12+4.5\,i$ many curves touch the central line.
Moreover, relation~\eqref{assg+t} suggests that there are infinitely
many curves that are tangent to the central line at the points
$\frac12+\frac{\pi i}{\log 2}\ell$ with $\ell\in \ZZ$. In
Figure~\ref{fig-res-projC} there seems to be a common touching to the
central line at $\bt\approx \frac12+ 9.0 \,i$ as well.
Figure~\ref{fig-touch} gives a closer few at the resonances near
$\bt=\frac12+\frac{\pi i}{\log 2}$ for curves computed in~\cite{Fr}.
\begin{figure}
\begin{center}
\renewcommand\brdt{7.4}
\makebox[\grtt]{\epsfxsize=\grtt\epsffile{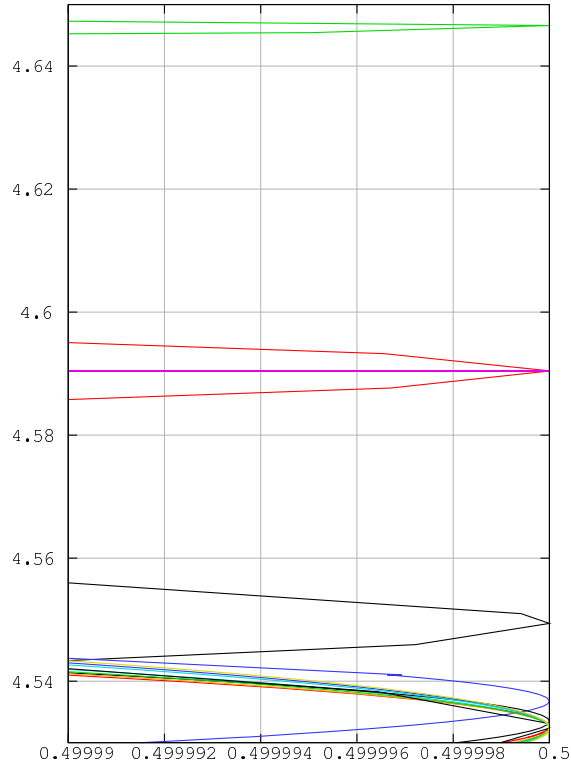}}
\end{center}
\caption{Enlargement of part of Figure~\protect\ref{fig-res-projC} near
$\bt=\frac12+\frac{\pi i}{\log 2}$.}\label{fig-touch}
\end{figure}
There is no common touching point, but a sequence of tangent points
approaching $\frac12+\frac{\pi i}{\log 2}$. Conclusions 8.2.32
and~8.2.33 in \cite{Fr} give a further discussion. Concerning this
phenomenon we have the following result:
\begin{thm}\label{thm+tch}Suppose that an interval~$I$ as in
Theorem~\ref{thm+I} contains in its interior a point
$t_\ell\isdef\frac{\pi\ell}{\log 2}$ with an integer $\ell\geq 1$.
Then there is $k_2\geq k_1$ such that for each $k\geq k_2$ the curve
$t\mapsto \s_k(t)+\nobreak it$ in Theorem~\ref{thm+I} is tangent to
the central line in a point $\frac12+it_\ell+i\dt_k\in \frac12+iI$,
and the $\dt_k$ satisfy
\be\label{dtk-as}
\dt_k \= \frac{\eta_2}{\pi^2} e^{A(1/2+it_\ell)/t_\ell}\, e^{-2\pi
k/t_\ell}\,\Bigl(1+\oh\bigl(k^{-1}\bigr)\Bigr)\,, \ee
for some $\eta_2\in \RR$. The function~$A$ is as in
Theorem~\ref{thm+I}.
\end{thm}
We do not get information concerning $\eta_2$ from the theory.
Table~8.8 in \cite{Fr} gives approximated tangent points near
$\frac12+\frac{\pi i}{\log 2}$. In Figure~\ref{fig-tp1} we give the
corresponding approximations of $\log\eta_2$.
\begin{figure}
\begin{center}
\renewcommand\brdt{7.4}
\makebox[\grtt]{\epsfxsize=\grtt\epsffile{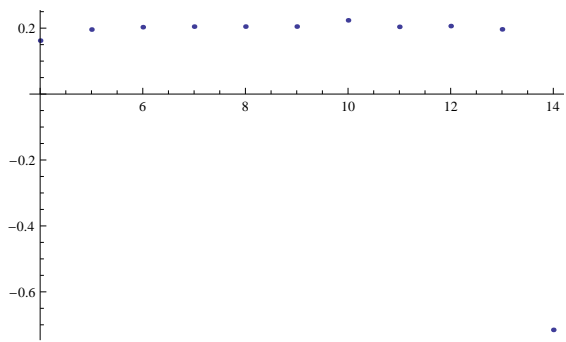}}
\end{center}
\caption{Approximation of $\eta_2$ in Theorem~\protect\ref{thm+tch} for
$\ell=1$, based on Table~8.8 in~\protect\cite{Fr}.}\label{fig-tp1}
\end{figure}
\medskip

\rmk Theorems \ref{thm-ei-0}--\ref{thm+tch} have been motivated by
part of the observations of Fraczek. In the next sections we present
proofs that do not depend on the computations. The comparisons of the
theoretically obtained asymptotic results with the computational data
is in some cases convincing, and show in other cases discrepancies
that we do not understand fully.

\rmk Figure~\ref{fig-3dim} shows curves of resonances that do not
approach $\bt=\frac12$ as $\al\downarrow0$.
\begin{figure}
\begin{center}
\renewcommand\brdt{8.4}
\makebox[\grtt]{\epsfxsize=\grtt\epsffile{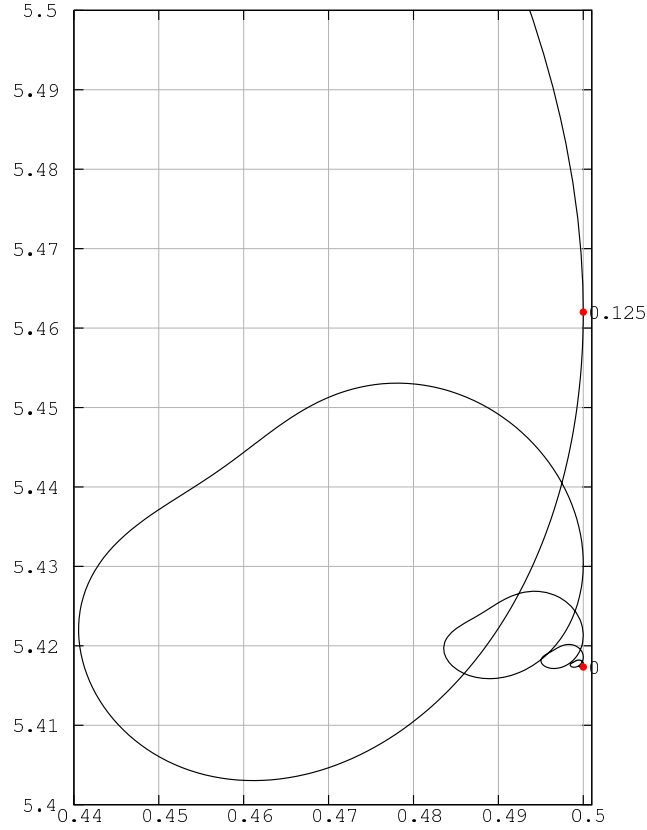}}
\end{center}
\caption{Curve of resonances starting at
$\bt\approx \frac12+i 5.4173$. Values of $\al$ are given in red.}
\label{fig-loop2}
\end{figure}
One of these curves is depicted in Figure~\ref{fig-loop2}, with an
enlargement of the part with small values of~$\al$ in
Figure~\ref{fig-loop2-detail}. The suggestion is that this curve
forms loops that repeatedly touch the central line.
\begin{figure}
\begin{center}
\renewcommand\brdt{8.4}
\makebox[\grtt]{\epsfxsize=\grtt\epsffile{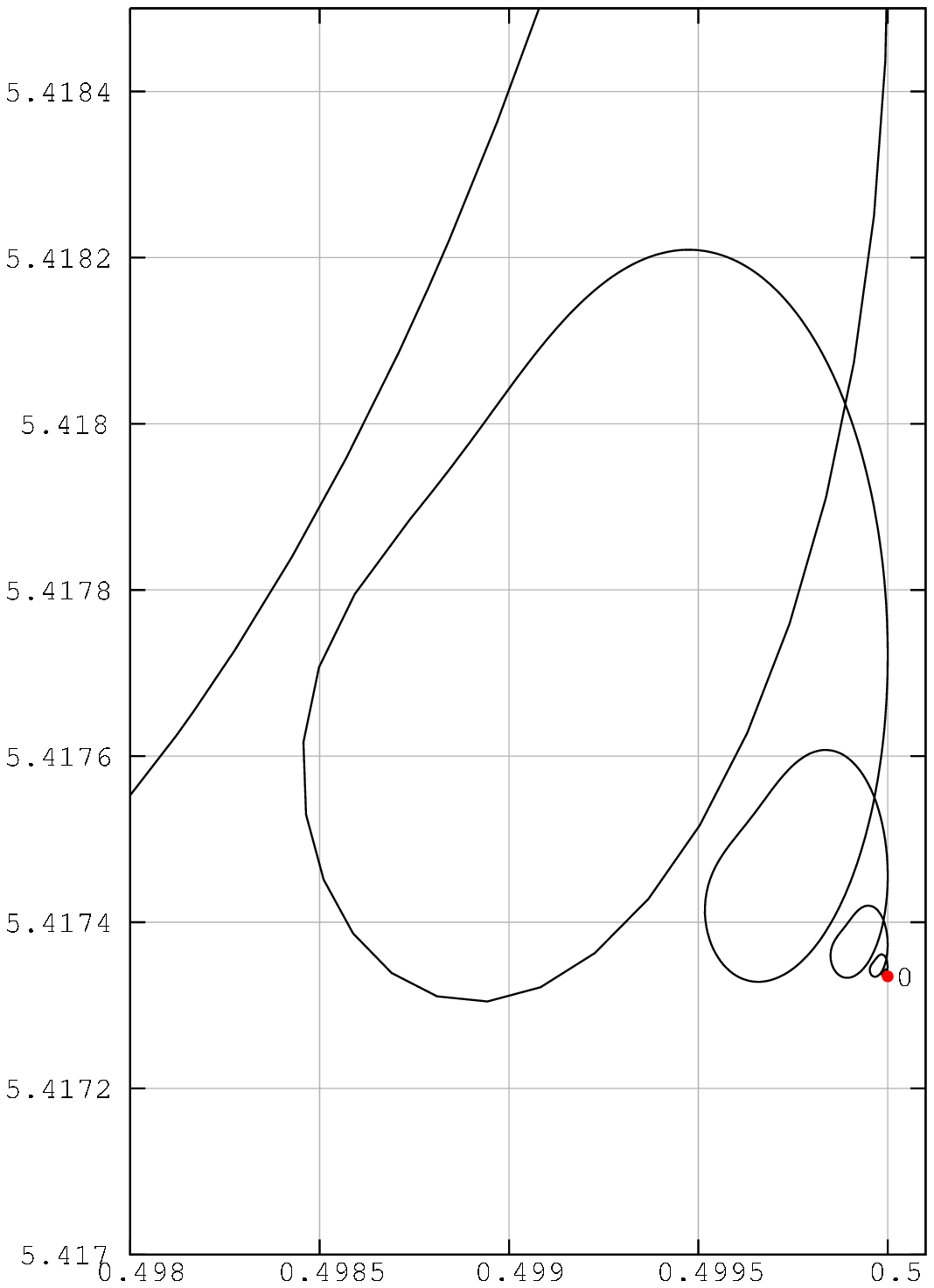}}
\end{center}
\caption{Enlargement of the initial part of the curve in
Figure~\protect\ref{fig-loop2}.} \label{fig-loop2-detail}
\end{figure}

We cannot prove that this type of behavior is bound to happen. In
Proposition~\ref{prop-loops} we work under a number of assumptions,
and then can prove some of the properties that can be seen in the
data.

\section{Proofs}\label{sect-prfs}
In this section we prove the theoretical results stated
in~\S\ref{sect-res}. The ingredients from the spectral theory of
automorphic forms that we use are the scattering matrix and a
generalization of it.

In \S\ref{sect-facts}--\ref{sect-esms} we summarize the facts we need.
In \S\ref{sect-cuei}--\ref{sect-loops} we prove the results
in~\S\ref{sect-res}.
\subsection{Facts from spectral theory}\label{sect-facts}
We will need a restricted list of facts from the spectral theory of
automorphic forms. Table~\ref{tab-refs} gives a reference to a
further discussion.

In the spectral theory of automorphic forms the \emph{scattering
matrix} plays an important role. For the unperturbed situation
$\al=0$ it is explicitly known:
\be \label{sm0-expl}
\scm_0(\bt) \=\frac1{2^{2\bt}-1}\, \frac{\Ld(2\bt-1)}{\Ld(2\bt)} \,
\begin{pmatrix} 2^{1-2\bt}& 1-2^{1-2\bt}&1-2^{1-2\bt}\\
1-2^{1-2\bt}&2^{1-2\bt}&1-2^{1-2\bt}\\
1-2^{1-2\bt}&1-2^{1-2\bt}&2^{1-2\bt}
\end{pmatrix}\,, \ee
where $\Ld(s) = \pi^{-s/2}\, \Gf(s/2)\, \z(s)$ is the completed
Riemann zeta-function. \fact\label{fact-unpert} The zeros of
$Z(0,\cdot)$ in the region $\re \bt<\frac12$ and $\im\bt>0$ are the
values at which one of the matrix elements of the scattering matrix
$\scm_0$ has a singularity.
\endfact The occurrence of $\z(2\bt)$ in the denominator of the matrix
elements explains the zeros of $Z(s,\cdot)$ with $\re \bt = \frac14$.
The factor $(2^{2\bt}-\nobreak1)^{-1}$ produces zeros on the line
$\re \bt=0$.

The scattering matrix in~\eqref{sm0-expl} can be embedded in a
meromorphic family of matrices
\be\label{escm} \scm(\al,\bt) \= \begin{pmatrix}
C_{0,0}(\al,\bt)&C_{0,\infty}(\al,\bt)&C_{0,\infty}(\al,\bt)\\
C_{0,\infty}(\al,\bt)& C_{\infty,\infty}(\al,\bt) &
C_{\infty,-1/2}(\al,\bt)\\
C_{0,\infty}(\al,\bt)& C_{\infty,-1/2}(\al,\bt)&
C_{\infty,\infty}(\al,\bt)\end{pmatrix}
\ee
 of $3\times 3$-matrices on $U\times\CC$, where $U$ is a neighborhood
 of $(-1,1)$ in~$\CC$. We call it the \emph{extended scattering
 matrix}. Its construction depends on functional analysis, and is far
from explicit. We will use its properties in
\ref{fact-escm0}--\ref{fact-pert-scm}.

The columns and rows are indexed by $0$, $\infty$ and $-\frac12$
(representatives of the cuspidal orbits of~$\Gm_0(4)$). If all
symmetries visible in the matrix in~\eqref{sm0-expl} would disappear
under perturbation, there would be nine different matrix elements.
However, some of the symmetries survive perturbation:
\fact\label{fact-sym-escm}The extended scattering matrix satisfies
the symmetries indicated by the coinciding matrix entries
in~\eqref{escm}.
\endfact

\fact\label{fact-escm0}The restriction $\bt \mapsto \scm(0,\bt)$
exists and is the scattering matrix $\bt\mapsto \scm_0(\bt)$
in~\eqref{sm0-expl}.\endfact We note that a meromorphic function
$(\al,\bt)\mapsto f(\al,\bt)$ on an open set of $\CC^2$ may have a
singularity at $(\al_0,\bt_0)$ that is not visible as a singularity
of $\bt \mapsto f(\al_0,\bt)$. (Consider for instance
$f(\al,\bt) = \frac{\al-\bt}{\al+\bt}$ at $(0,0)$.) Such
singularities are said to be of \emph{indeterminate type}.
\fact\label{fact-us-escm}Let $\bt_0\in \frac12+i[0,\infty)$. If the
extended scattering matrix has a singularity at $(0,\bt_0)$, then
$Z(0,\bt_0)=0$ and $\bt_0\neq\frac12$.
\endfact In \eqref{sm0-expl} we see that such a singularity has
necessarily indeterminate type.

There are functional equations: \fact\label{fact-fe} We have
\bad \scm(-\al,\bt) &\= \scm(\al,\bt)^t\,,&\quad \scm(\al,1-\bt)&\=
\scm(\al,\bt)^{-1}\,,\\
\overline{\scm(\bar \al,1-\bar\bt)} &\=
\bigl(\scm(\al,\bt)^t\bigr)^{-1}\,. \ead
as identities of meromorphic families of matrices on $U\times\CC$.
\endfact
\medskip

For the perturbed situation there is also a scattering ``matrix'',
with size $1\times 1$. Unlike the scattering matrix for $\al=0$, we
have no explicit formula for it. However it can be expressed in the
matrix elements of the extended scattering matrix.
\fact\label{fact-pert-scm}
Let $X$, $C_+$ and $C_-$ be the meromorphic functions on
$U_+\times \CC$, with $U_+=\{\al\in U\;:\; \re\al>0\}$, given by
\begin{align}\label{X-def}
 X(\al,\bt) &\= (\pi\al)^{2\bt-1} \, \Gf\bigl( \frac12-\bt\bigr)\,
 \Gf(\bt-\frac12\bigr)^{-1} \,,\\
\label{Cpm-def}
C_\pm(\al,\bt)&\=C_{\infty,\infty}(\al,\bt) \pm
C_{\infty,-1/2}(\al,\bt)\,.
\end{align}
The meromorphic function
\be\label{D-def}
 D_{0,0}(\al,\bt)\= \frac{ C_{0,0}(\al,\bt) - X(\al,\bt)\,\bigl(
C_{0,0}(\al,\bt)\, C_+(\al,\bt) - 2 \,
C_{0,\infty}(\al,\bt)^2\bigr)}{1-X(\al,\bt)\, C_+(\al,\bt)} \ee
on~$U_+\times\CC$, has a meromorphic restriction $D_\al$ to the
complex line $\{\al\}\times\CC$ for each $\al\in (0,1)$.

For $\al\in (0,1)$ the zeros of the Selberg zeta-function $Z(\al,\bt)$
with $\im \bt>0$ satisfy $\re\bt\leq \frac12$. Those of these zeros
that satisfy $\re\bt<\frac12$ are the values of $\bt$ at which
$D_\al(\bt)$ has a singularity.
\endfact

We note that the existence of the restriction to $\{\al\}\times\CC$ is
a non-trivial assertion. It says that the meromorphic
function~$D_{0,0}$ has no singularity along this complex line.

\fact\label{fact-odd-eq} If $C_-$ is holomorphic at
$(\al,\bt) \in (0,1)\times \left(\frac12+i\RR\right)$ and
\[ X(\al,\bt) \, C_-(\al,\bt) \= 1\,,\]
then $Z(\al,\bt)=0$.
\endfact

\def\ftn{\eqref{sm0-expl}& \cite{Hu84}}
\def\fta{\ref{fact-unpert}& \S\ref{sect-zSz}}
\def\ftb{\ref{fact-sym-escm}&\eqref{escm-2} and \S\ref{sect-ks}}
\def\ftc{\ref{fact-escm0}&\S\ref{sect-fam}}
\def\ftd{\ref{fact-us-escm}&\S\ref{sect-ecm-upcf}}
\def\fte{\ref{fact-fe}& \eqref{fea},\; \eqref{fec},\; \eqref{feMS}}
\def\ftf{\ref{fact-pert-scm}&\S\ref{sect-zSz} and \S\ref{sect-fam},
especially \eqref{UDUi}}
\def\ftg{\ref{fact-odd-eq}& \S\ref{sect-cpPs}}
\begin{table}
\renewcommand\arraystretch{1.2}
\begin{tabular}{|cl|cl|}\hline
\ftn & \ftd \\
\fta & \fte \\
\ftb & \ftf \\
\ftc & \ftg
\\ \hline
\end{tabular}\smallskip
\caption{The places in~\S\protect\ref{sect-spth} where a reference or a proof
is given for the facts in~\S\protect\ref{sect-facts}.}\label{tab-refs}
\end{table}

\subsection{The extended scattering matrix}\label{sect-esms} The
functional equations in \ref{fact-fe} imply that $\scm(\al,\bt)$ is a
unitary matrix if $\al\in (-1,1)$ and $\bt\in \frac12+i\RR$. This
implies that if $\al_0\in (-1,1)$ then the matrix elements of
$\scm(\al_0,\bt)$ are bounded. So the restriction
$\bt\mapsto \scm(\al_0,\bt)$ cannot have singularities on the line
$\{\al_0\}\times \left(\frac12+i\RR\right)$ as a function of the
variable~$\bt$. Nevertheless, the matrix elements can have
singularities at $(\al_0,\bt_0)$ with $\re\bt_0=\frac12$ as functions
of the two complex variables $(\al,\bt)$.

The extended scattering matrix can be partly diagonalized:
\badl{escm-split} \Um \, \scm \, \Um^{-1} &\= \begin{pmatrix} \scm^+ &
0\\ 0 & C_-\end{pmatrix}\,,
\qquad \Um \=
\begin{pmatrix}1&0&0\\0&\frac1{\sqrt2}&\frac1{\sqrt2}\\
0&-\frac1{\sqrt2}&\frac1{\sqrt2}\end{pmatrix}\,,\\
\scm^+ &\= \begin{pmatrix} C_{0,0} & \sqrt 2 \, C_{0,\infty} \\
\sqrt 2\, C_{0,\infty} & C_+ \end{pmatrix}\,. \eadl
For $(\al,\bt) \in (-1,1)\times\left(\frac12+i\RR\right)$ the matrix
$\scm^+(\al,\bt)$ is unitary, and $\bigl| C_-(\al,\bt)\bigr|=1$. The
functional equations in \ref{fact-fe} imply similar relations for
$\scm^+$. For $\al=0$ we have
\be\label{ucu0} \Um \,\scm_0(0,\bt) \,\Um^{-1} \=
\frac{\Ld(2\bt-1)}{\Ld{2\bt}}\,\begin{pmatrix}
\frac{ 2^{1-2\bt} }{2^{2\bt}-1} & \frac{ \sqrt 2(1-2^{1-2\bt})
}{2^{2\bt}-1} & 0 \\
\frac{ \sqrt 2(1-2^{1-2\bt}) }{2^{2\bt}-1} & \frac{ 1 }{2^{2\bt}-1} &
0 \\
0 & 0 & \frac{ 2^{2(1-\bt)}-1 }{2^{2\bt}-1}
\end{pmatrix}\,. \ee
See~\eqref{UCUi}.

\subsection{Curves of eigenvalues}\label{sect-cuei} For curves of
eigenvalue, {\sl i.e.}, zeros of the Selberg zeta-function that stay
on the central line, the fact \ref{fact-odd-eq} is important.

\begin{lem}\label{lem-S-}
There is a simply connected set
$S_{\!-} \subset (0,1) \times \bigl( \frac 12+\nobreak i(0,\infty)\bigr)$
in which $C_- $ has no singularities. For each bounded interval
$I \subset (0,\infty)$ there is $\e_I \subset (0,1)$ such that
$(0,\e_I)\times\bigl( \frac12+\nobreak iI\bigr)\subset S_{\!-}$.
\end{lem}
\begin{proof}If $f$ and $g$ are non-zero holomorphic functions on an
open subset $U\subset\CC^2$ without common factor in the ring of
germs of holomorphic functions at $p\in U$ then their null sets
intersect each other in an analytic subset of~$U$ that has
dimension~$0$ near~$p$. (See \cite{GR}. Chap.~5, \S2.4 implies that
the null sets have dimension~$1$ at~$p$. If their intersection would
also have dimension~$1$ at~$p$, then any prime component of this
intersection in the Lasker-Noether decomposition (\cite{GR}, p.~79)
would be given by a common factor of $f$ and~$g$ in the ring of germs
at~$p$.) So the quotient $ f/g$ can have singularities of
indeterminate type only at a discrete set of points in~$U$. Applying
this to $C_-$ the set of points to avoid is discrete in
$\bigl[0,\frac12\bigr]\times\bigl( \frac12+\nobreak i[0,\infty)\bigr)$.
Let $t_I=\max I$. We take $\e_I$ equal to the minimum of the $\al$
for which $(\al,\bt)$ is one of the finitely many points of
indeterminacy of $C_-$ in $\bigl(0,\frac12\bigr]\times\bigl(
\frac12+\nobreak i[0,t_I]\bigr)$.
\end{proof}

We consider the equation $X\, C_-=1$ in \ref{fact-odd-eq} in the set
$S_-$. The function
\be\label{Y-def}
Y_-(\al,\bt) \= \frac{\Gf(\bt-\frac12)}{\Gf(\frac12-\bt\bigr)}\, \frac
1{C_-(\al,\bt)} \ee
is holomorphic at all points of $S_-$, and has absolute value~$1$ at
the points of~$S_-$. The equation $X \, C_-=1$ on $S_-$ is equivalent
to
\[ (\pi \al)^{2\bt-1} \= Y_-(\al,\bt)\,.\]
We can choose a continuous argument $A_-(\al,\bt)$ of $Y_-(\al,\bt)$
on $S_-$, since this set is simply connected. The function $1/C_-$
may have singularities (of indeterminate type) at points of
$\{0\}\times\left(\frac12+i\RR\right)$. If that occurs then the
continuous extension of $A_-$ to
$\{0\}\times\left(\frac12+i\RR\right)$ minus the points where $1/C_-$
is singular does not have a constant difference with a continuous
argument of
\[Y_-\bigl(0,\frac12+it\bigr) \= \pi^{2it}\,\frac{2^{1+2it}-1}
 {2^{1-2it}-1}\, \frac{\z(-2it)}{\z(2it)}\,.\]
See \eqref{sm0-expl}. {}From \ref{fact-us-escm} we see that $Y_-$ is
holomorphic at~$(0,\frac12)$. It has value~$1$ at $(0,\frac12)$. We
normalize $A_-$ such that its continuous extension has value $0$
at~$(0,\frac12)$. We find the Taylor approximation
\be\label{phi-as}
A_-\bigl(\al,\frac12+it\bigr) \= 2t\log\frac 4\pi +
\oh(t^2)+\oh(\al)\qquad\text{ as }(\al,t)\rightarrow(0,0)\,. \ee

With this preparation, we can reformulate the equation $X\, C_-=1$
in~$S_-$ as
\be\label{-eq} 2t\,\log\pi\al \= A_-\bigl(\al,\frac12+it\bigr) - 2\pi
k\,, \qquad(k\in \ZZ)\,. \ee
We have written $\bt=\frac12+it$.

\begin{proof}[Proof of Theorem~\ref{thm-ei-0}.] The formulation
in~\eqref{-eq} shows that the solution set of $X\, C_-=1$ in~$S_-$ is
the disjoint union of components $V_k$, parametrized by the
integer~$k$ in~\eqref{-eq}.

We take $\e_1>0$, $\e_2>0$ such that
$\Om=(0,\e_1) \times\left( \frac12+i(0,\e_2) \right) \subset S_-$ and
such that $|A_-|<\pi$ on~$\Om$. In the course of the proof we will
impose a finite number of additional conditions on~$\e_1$ and~$\e_2$.

Equation~\eqref{-eq} implies that for
$(\al,\frac12+\nobreak it) \in V_k\cap \Om$ we have
\[ e^{-(C+\pi k)/t} \,\leq\, \pi\al\,\leq\, e^{(C-\pi k)/t}\]
for some $C\leq \frac\pi 2$. On the basis of this first estimate we
proceed more precisely, and obtain
\be\label{a-1} 2t\log\frac{\pi^2\al}4 \= -2\pi k + \oh(t^2)+\oh(\al)
\= -2\pi k+\oh(t^2)\,, \ee
and conclude
\be \al \= \frac 4{\pi^2}\, e^{-\pi k/t}\, \bigl( 1+\oh(t)\bigr)\,.
\ee
If $k\leq 0$ this does not allow small values of $\al$ for
$t\in (0,\e_2)$. Hence $k\geq 1$.

To show that $V_k \cap\Om$ is the graph of a function, we apply the
implicit function theorem. The set $V_k$ is the level set
$F(\al,t)=-2\pi k$ of the function $F(\al,t) =  2 t
\log\pi\al-A_-\bigl(\al,\frac12+\nobreak it\bigr)$, with derivatives
\[ \frac{\partial F}{\partial\al} \= \frac{2t}\al+\oh(1)\,,\qquad
\frac{\partial F}{\partial t} \= 2\log\frac{\pi^2\al}4 +
\oh(t)+\oh(\al)
\,.\]
So $\frac{\partial F}{\partial\al}>0$ and
$\frac{\partial F}{\partial t}<0$ if we take $\e_1$ and $\e_2$
sufficiently small. So $V_k \cap \Om$ is the graph of an injective
function $\al\mapsto \frac12+i\tau_k(\al)$ on $(0,\e_1)$. Since $F$
is a real-analytic function, the analytic implicit function theorem
shows that $\tau_k$ is real-analytic. (See, {\sl e.g.}\, \cite{KrPa},
Theorem~6.1.2.)

{}From~\eqref{a-1} it follows that for $\al\in (0,\e_1)$, with $\e_1$
sufficiently small,
\[\tau_k(\al)\= \oh\bigl( k/\log(\pi^2\al/4) \bigr)\,.\]
and then
\be \tau_k(\al) \= \frac{-\pi k+\oh\bigl(
\tau_k(\al)^2\bigr)}{\log\frac{\pi^2\al}4} \= \frac{\pi
k}{-\log\frac{\pi^2\al}4}\,\biggl(1 + \oh\Bigl(
k^2/\log(\pi^2\al/4)^2\Bigr)\biggr)\,. \ee
This gives~\eqref{as-0}.
\end{proof}

\begin{proof}[Proof of Theorem~\ref{thm-ei-ln}.] Let
$I\subset (0,\infty)$ be an interval as in the theorem.
Lemma~\ref{lem-S-} provides us with $\e_I$ such that
$(0,\e_I)\times \left(\frac12+iI\right)\subset S_-$. Since
$\{0\}\times\left(\frac12+iI\right)$ does not contain singularities
of~$C_-$ the function $A_-$ is continuous and hence bounded on
$[0,\e_1]\times \left(\frac12+iI\right)$ for $0<\e_1< \e_I$.

The graph $\tau_k$ is the level curve $F=-2\pi k$ of the function
\be\label{Fd}
 F(\al,t) \= 2t\,\log\pi\al - A_-\bigl(\al,\frac12+it\bigr)\,, \ee
and hence the graphs of $\tau_k$ for different values of~$k$ do not
 intersect each other in $S_-$. The asymptotic relation~\eqref{as-0}
shows that $\tau_k(\al)<\tau_{k_1}(\al)$ if $k<k_1$ for sufficiently
small~$\al$. Since the graphs have no intersections this relation is
preserved throughout the intersection of the domains of $\tau_k$
and~$\tau_{k_1}$. The Selberg zeta-function is not the zero function,
the eigenvalues $\tau_k(\e_1)$ form a discrete set, with only
finitely elements under~$\max I$. We take $k(\e_1)$ such that
$\tau_k(\e_1)>\max I$ for all $k\geq k(\e_1)$.

Take $t\in I$ and $k\geq k(\e_1)$. The function $F$ is equal to
$-2\pi k$ along the graph of $\tau_k$. We have
$\lim_{\al\downarrow 0} F(\al,t)=-\infty$, and $F(\al,t)$ is larger
than $-2\pi k$ under the graph of $\tau_k$. In particular
$F_k(\e_1,t)>-2\pi k$ since $\tau_k(\e_1)>\max I$. Differentiation
gives
\[ \frac d{d\al} F_k(\al,t) \= \frac{2t}\al + \oh(1)\,.\]
So the derivative of $\al\mapsto F_k(\al,t)$ is positive for
$\al\in(0,\e_1]$ if we take $\e_1$ sufficiently small. So there is a
unique $a_k(t) \in (0,\e_1]$ such that
$\tau_k\bigl( a_k(t)\bigr) = t$. This function $a_k$ inverts $\tau_k$
on~$I$.
\[\setlength\unitlength{1cm}
\begin{picture}(4,4)
\put(0,0){\line(1,0){4}}
\put(0,0){\line(0,1){4}}
\put(3,0){\circle*{.1}}
\put(0,2){\circle*{.1}}
\put(0,3.5){\circle*{.1}}
\put(3.1,.1){$\e_1$}
\put(4.1,0){$\al$}
\put(-.3,2.7){$I$}
\put(3.2,3.8){$F=-2\pi  k$}
\put(.1,3.0){$t$}
\put(0,2.9){\line(1,0){3}}
\thicklines
\put(0,2){\line(0,1){1.5}}
\qbezier(0,0)(0,2)(.8,3)
\qbezier(.8,3)(1.3,3.6)(3,4)
\end{picture}
\]

The estimate
\[ 2 t \, \log\pi a_k(t) + 2\pi k \= A_-\bigl(0,\frac12+it\bigr) +
\oh\bigl( a_k(t) \bigr)\]
is uniform for $t\in I$. It implies
$\log a_k(t) = -\frac{\pi k}t + \oh(1)$ uniform for $t\in I$ and
$k\geq k(\e_1)$, and hence
\[ \pi a_k(t) \= \exp\bigl(-\pi k/t+ A_-(0,1/2+it)/2t+ \oh(e^{-\pi
k/t})\bigr)
\,.\]
The function $\ph_-$ in the theorem is continuous on $[0,\infty)$. The
argument $A_-\bigl(0,\frac12+\nobreak it\bigr)$ differs from it by
$2\pi k_I$ for some integer depending on the interval~$I$. (More
precisely, depending on the component of $I$ in $[0,\infty)$ minus
the singularities of $Y_-$.)
This gives~\eqref{alk-as}.
\end{proof}
\smallskip\par
\noindent\emph{Avoided crossing.} For
$(\al,\bt)=\bigl(\al,\frac12+it\bigr)$ in the set $S_-$ in
Lemma~\ref{lem-S-} the gradient of $F$ is
\[ \nabla F (\al,t)\= \begin{pmatrix} 2t /\al \\ 2 \log\pi \al
\end{pmatrix} - \nabla A_-(\al,1/2+it)\,. \]
On the sets considered in Theorem~\ref{thm-ei-ln} the gradient of
$A_-\bigl(\al,\frac12+it\bigr)$ is $\oh(1)$. If $t = \tau_k(\al)$,
then $\nabla F(\al,\bt)$ is proportional to
$\bigl(\tau_k'(\al),-1\bigr)$, and hence
\[ \tau_k'(\al) \=\frac{2t\al^{-1}+\oh(1)}{- 2\,\log\pi\al+\oh(1)} \=
\frac t{\al\, |\log\pi \al|} \,\bigl(1+\oh(1/\log\al|) \bigr)\,. \]
This confirms that the graphs of the $\tau_k$ are steep for
small~$\al$.

If we are near a singularity of $C_-$ at
$(0,\bt_0)\in \{0\}\times \left(\frac12+i(0,\infty\right)$ this
reasoning is no longer valid. In \ref{fact-us-escm} we see that this
can only occur if $\bt_0$ is an unperturbed eigenvalue. For
$(\al,\bt)\in S_-$ we have $|C_-(\al,\bt)|=1$. So if $C_-$ has a zero
or a pole at $(\al,\bt)$ near $(0,\bt_0)$ with $\al$ real, then
$\re\bt\neq \frac12$.

It seems hard to analyze this precisely for a complicated singularity.
Hence we work under simplifying assumptions.
\begin{prop}\label{prop-av}Let $\bt_0=\frac12+it_0$ with $t_0>0$. We
assume that on a neighborhood $\Om$ of $(0,\bt_0)$ in $\CC^2$ the
matrix element $C_-$ of the extended scattering matrix has the form
\be C_-(\al,\bt) \= \ld(\al,\bt)\, \frac{\bt-\bt_0-n(\al)}
{\bt-\bt_0-p(\al)}\,, \ee
where $n $ and $p$ are holomorphic functions on a neighborhood of $0$
in~$\CC$ and $\ld$ is holomorphic on~$\Om$ without any zeros. We
suppose furthermore that $n(0)=p(0)=0$, and $\re p''(0)\neq 0$.

Then there are $\e_0>0$ and $k_0\geq 1$ such that for all $k\geq k_0$
there exists $\al_k\in (0,\e_1]$ such that
$\tau_k(\al_k) = t_0+\im  p(\al_k)$, and for these $\al_k$ we have
\be\tau_k'(\al_k) \= \im p'(\al_k) - \frac12 t_0\,\re p'(\al_k) +
\oh(\al_k^2)\,. \ee
\end{prop}
\rmk\label{rmk-av} Suppose that $\bt_0\in \frac12+i(0,\infty)$ is an
unperturbed eigenvalue, {\sl i.e.}, $Z(0,\bt_0)=0$. Then it might be
associated to a singularity at $(0,\bt_0)$ of the extended scattering
matrix, as in~\ref{fact-us-escm}. This singularity might be visible
as a singularity of the coefficient $C_-$. For that case, the
assumptions in the proposition seem to describe the most general
situation. If these assumptions are satisfied then there is the curve
\[ K_{\bt_0}:\al\mapsto \bigl(\al,t_0+ \im p(\al) \bigr)\]
through $(0,t_0)$ such that the derivatives of the $\tau_k$ for all
large~$k$ are relatively small at the points where the graph of
$\tau_k$ crosses the curve~$K_{\bt_0}$.
(See also Remark~\ref{rmk-odd} in~\S\ref{sect-concl}.)

\begin{proof}[Proof of Proposition~\ref{prop-av}.]The assumption that
$C_-$ has a singularity at $\bt_0$ implies that the functions $p$ and
$n$ cannot be equal. {}From \ref{fact-fe} and~\eqref{Cpm-def} it
follows that $C_-(\al,\bt)$ is even in~$\al$. This evenness is
inherited by the zero set and the set of singularities. Hence $p$ and
$n$ are even functions. We also have
$\overline{C_-(\bar\al,1-\bar \bt)}=C_-(\al,\bt)^{-1}$. This implies
$n(\al) = -\overline{p(\bar \al)}$. For real $\al$ we write
$p_r(\al) = \re p(\al)$ and $p_i(\al) = \im p(\al)$. Hence we have
for $\al\in (0,\e_1)$ and $t\approx t_0$:
\[ Y_-\bigl(\al,\frac12+it\bigr)\= \ld\bigl(\al,\frac12+it\bigr)^{-1}
\frac{ \Gf(it)\,\bigl( -p_r(\al) +i(t-t_0 - p_i(\al))\bigr)
}{\Gf(-it)\, \bigl( p_r(\al)+i(t-t_0-p_i(\al))
\bigr)
}\]
So modulo $2\pi \ZZ$:
\be\label{phi-s} A_-\bigl(\al,\frac12+it\bigr)\;\equiv\; 2
\arg\bigl(p(\al)-i(t-t_0)\bigr)+\oh(1)\,, \ee
where the term indicated by $\oh(1)$ has also bounded derivatives. So
the gradient with respect to the variables $\al$ and $t$ is
\be \nabla A_-\bigl(\al,\frac12+it\bigr)\= \oh(1) + \im
\begin{pmatrix}
\frac{2\,p'(\al)}{p(\al)-i(t-t_0)} \\ \frac{-2i}{p(\al)-i(t-t_0)}
\end{pmatrix}\,. \ee

The function $\al\mapsto F\bigr(\al,t_0+\nobreak p_i(\al)\bigr)$, with
$F$ as in~\eqref{Fd}, tends to $-\infty$ as $\al\downarrow 0$, and
has derivative
\begin{align*} 2\,p_i'(\al)\,\log\pi\al &+ \frac{2\,
(t_0+p_i(\al))}{\al}
- \im \frac{2\, p'(\al)}{p_r(\al)} - \im \frac{-2i}{p_r(\al)}\,
p_i'(\al)
+ \oh(1)\\
&\= \frac{ 2t_0 + 2p_i(\al) }\al +
\frac{-2\,p_i'(\al)+2p_i'(\al)}{p_r(\al)} +\oh(1)
\= \frac{2t_0}\al + \oh( 1)\,,
\end{align*}
where we use that $p(\al)=\oh(\al^2)$ and $p'(\al)=\oh(\al)$ as
$\al\downarrow 0$, since $p$ is an even function vanishing at~$0$. So
there is an interval $(0,\e_1]$ on which
$\al\mapsto F\bigl(\al,t_0+\nobreak p_i(\al)\bigr)$ is increasing.
Hence for all sufficiently large integers~$k$ there are
$\al_k\in (0,\e_1]$ such that
$F(\al_k,t_0+\nobreak p_i(\al_k)\bigr)=-2\pi k$, and then
$\tau_k(\al_k)=t_0+p_i(\al_k)$.

We have $2\, \tau_k(\al_k) \,\log\pi\al_k \= -2\pi  k + \oh(1)$, since
the argument in \eqref{phi-s} stays bounded in a neighborhood of
$(0,\bt_0)$. So $\log\pi\al_k=\frac{-\pi k}{t_0}+ \oh(1)$ as
$k\rightarrow\infty$, and hence $\al_k\downarrow 0$.

We have, again using $p(\al)=\oh(\al^2)$ and $p'(\al)=\oh(\al)$,
\begin{align*}
\nabla F\bigl( \al_k&,\frac12+i\tau_k(\al_k) \bigr) \=
\begin{pmatrix}\frac{2\,\tau_k(\al_k)}{\al_k} - \im \frac{2\,
p'(\al_k)}{p_r(\al_k) }
\\
2\log\pi\al_k - \im \frac{-2i}{p_r(\al_k)}
\end{pmatrix} + \oh(1)\\
&\= \begin{pmatrix} \frac{2t_0}{\al_k} - \frac{2\,
p_i'(\al_k)}{p_r(\al_k)}
+\oh(1)\\ \frac 2{p_r(\al_k)} +\oh(\log\al_k)
\end{pmatrix}
\,.
\end{align*}
Since the graph of $\tau_k$ is a level curve of~$F$ the gradient
$\nabla F\bigl( \al_k,\frac12+i \tau_k(\al_k)\bigr)$ is orthogonal to
$\begin{pmatrix}
1\\\tau_k'(\al_k)\end{pmatrix}$. We use
$p(\al) =\frac12\, p''(0)\,\al^2+\oh(\al^3)$ and
$p'(\al) =  p''(0)\,\al + \oh(\al^2)$ as $\al\rightarrow0$, and
obtain:
\begin{align}
\nonumber
 \tau_k'(\al_k) &\= - \frac{2\,t_0/\al_k -2\, p_i'(\al_k)/p_r'(\al_k)
 + \oh(1)}{2/p_r(\al_k) + \oh\bigl( \log\al_k\bigr)}\\
\nonumber
&\= \frac{-2\,t_0\, \al_k + 4 \, \al_k\,
p_i''(0)/p_r''(0)+\oh(\al_k^2)}{4/p_r''(0)+\oh(\al_k^2\log\al_k)}\\
&\= \al_k\bigl(-\frac12 t_0\, p_r''(0)+ p_i''(0)\bigr)+ \oh(\al_k^2)
\= p_i'(\al_k) - \frac {t_0}2 p_r'(\al_k) + \oh(\al_k^2)\,.
\end{align}
\end{proof}

\subsection{Curves of resonances originating at
$\bigl(0,\frac12\bigr)$} To find resonances for $\al\in (0,1)$ we
have to look for singularities of the scattering ``matrix''
$D_\al(\bt)$ in~\ref{fact-pert-scm}. If we work on a region where the
extended scattering matrix has no singularities this means that we
look for solutions of $X(\al,\bt) \, C_+(\al,\bt)=1$ with the
requirement that the resulting singularity of $D_{0,0}(\al,\bt)$ is
not canceled by a zero of the numerator in~\eqref{D-def}.

\begin{prop}\label{prop-dn}Let $\Om$ be a region in
$(0,1)\times \{\bt\in \CC \;:\; \im\bt>0\}$ that is invariant under
$(\al,\bt)\mapsto (\al,1-\nobreak\bar\bt)$ (reflection in the central
line). Suppose that the matrix $\scm^+$ in \eqref{escm-split} is
holomorphic on a neighborhood op $\Om$ in~$\CC^2$.

The denominator $M=1-X \, C_+$ in the expression for $D_{0,0}$
in~\eqref{D-def} vanishes at $(\al,1-\nobreak\bar\bt)$, if and only
if the numerator
$N=C_{0,0}-X\,\left( C_{0,0}\, C_+-2\, C_{0,\infty}^2\right)$
vanishes at $(\al, \bt)$.
\end{prop}
\begin{proof}The function $\Dt=C_{0,0}\, C_+-2\, C_{0,\infty}^2$ is
the determinant of the matrix~$\scm^+$. It follows from \ref{fact-fe}
that $\overline{\scm^+(\al,1-\nobreak
\bar \bt)} = \bigl(\scm^+(\al,\bt)^t)^{-1}$ for $(\al,\bt)\in \Om$.
If $\Dt$ would have a zero at $(\al,\bt)\in\Om$, this would
contradict the holomorphy of $\scm^+$ on~$\Om$. Furthermore,
$\overline{X(\al,1-\bar\bt)} = X(\al, \bt)^{-1}$.

We have
\[ \overline{C_+(\al,1-\bar\bt)} \= \text{coefficient at position
$(2,2)$ of }\bigl(\scm^+(\al,\bt)^t\bigr)^{-1} \=
\frac{C_{0,0}(\al,\bt)}{\Dt(\al,\bt)} \]
Since $\Dt(\al,\bt)\in \CC^\ast$ we have equivalence of the following
assertions:
\begin{align*}
 X(\al,1-\bar \bt) \, C_+(\al,1-\bar \bt) &\= 1\,,&\quad
 \overline{C_+(\al,1-\bar \bt)} &\= \overline {X(\al,1-\bar
 \bt)^{-1}}\,,\\
C_{0,0}(\al,\bt)/\Dt(\al,\bt)&\= X(\al,\bt)\,,&\quad C_{0,0}(\al,\bt)
&\= X(\al,\bt)\, \Dt(\al,\bt)\,.
\end{align*}
\end{proof}
\rmk So zeros and singularities of $D_{0,0}(\al,\bt)$ are interchanged
by the reflection in the central line. The meromorphic function
$D_{0,0}$ is not the zero function, since its restriction to the
complex lines $\{\al\}\times\CC$ for $\al\in (0,1)$ are scattering
``matrices'', which are non-zero. So its sets of zeros and poles
intersect each other only in a discrete set in
$U_+\times\CC$.\bigskip

We now consider the equation $1=X\, C_+$, in a region where $\scm^+$
has no singularities. Analogously to~\eqref{Y-def}, we put
\be\label{Y+def} Y_+( \al, \bt) \=
\frac{\Gf(\bt-\frac12)}{\Gf(\frac12-\bt)}\, \frac 1{C_+(\al,\bt)}\,.
\ee
This is a meromorphic function on $U\times \CC$, and the equation
$X\, C_+=1$ on $U_+\times\CC$ is equivalent to
\be\label{Y+eq} (\pi\al)^{2\bt-1}\= Y_+(\al,\bt)\,. \ee
A complication is that now we cannot restrict our consideration to a
subset of $(0,1)\times\left( \frac12+i(0,\infty)\right)$, but have to
allow $\bt$ to vary over a neighborhood of the central line. The
presence of singularities of $C_+$ makes it impossible to choose a
well defined argument globally.

\begin{lem}\label{lem+I}
Let $I\subset [0,\infty)$ be a bounded closed interval such that $C_+$
in has no singularities at points $\bigl(0,\frac12+\nobreak it\bigr)$
with $t\in I$. There are $\e_1,\e_2>0$ such that the solution set of
\eqref{Y+eq} in
\be\label{Omee}
\Om_I(\e_1,\e_2)\= (0,\e_1] \times \Bigl(
\bigl[\frac12-\e_2,\frac12+\e_2\bigr]\times iI\Bigr)\ee
consists of sets $V_k$ parametrized by $k\in \ZZ$.

There exists $k_1\in \ZZ$ such that $V_k$ is for all $k\geq k_1$ a
real-analytic curve
\[ t\mapsto \bigl( \al_k(t),\, \s_k(t)+it\bigr)\qquad(t\in I)\,.\]
\end{lem}
\begin{proof}
Since $C_+$ is holomorphic at all points of the compact set
$\{0\}\times\left( \frac12+iI\right)$, it has the value given by the
restriction to $\al=0$, which value we know explicitly from
\eqref{sm0-expl} and~\eqref{Cpm-def}:
\[ C_+\bigl(0,\frac12+it\bigr)\=\frac{\pi^{-2it}\,\Gf(it)\,
\z(2it)}{(2^{1+2it}-1)\,\Gf(-it)\,\z(-2it)}\,. \]
So $C_+$ has also no zeros on $\{0\}\times\bigl(\frac12+iI\bigr)$. We
can choose $\e_1,\e_2>0$ such that $C_+$ also has no singularities or
zeros with $\al\in (0,\e_1]$,
$\bigl|\re\bt-\nobreak\frac12\bigr|\leq \e_2$, and $\im\bt\in I$. We
take real-analytic functions on $ \overline{\Om_I(\e_1,\e_2)}$
\bad\label{MA+def} M_+(\al,\bt) &\= \log\bigl|Y_+(\al,\bt)\bigr|\,,\\
  A_+(\al,\bt) &\= \arg Y_+(\al,\bt)\,. \ead
There is freedom in the choice of the argument. In this lemma we do
not choose a normalization.

The solution set of~\eqref{Y+eq} in $\Om _I(\e_1,\e_2)$ is the
disjoint union of components $V_k$ given by
\badl{+eq} 2t \,\log\pi\al &\= A_+(\al,\s+it)-2\pi k\,,\\
(2\s-1)\, \log\pi\al &\= M_+(\al,\s+it)\,. \eadl
Here and in the sequel we write $\s=\re \bt$ and $t=\im\bt$. Changing
the choice of $A_+$ causes a shift in the parameter~$k$.

We want to use the fixed-point theorem to show that for each $t\in I$,
$t>0$, and each $k\geq k_1$ there is exactly one solution
of~\eqref{+eq}. To do this, we write $\al(x)=e^{-1/x}/\pi$ and
$\bt(y,t) = \frac12-y+it$. Then $\al\in
(0,\e_1]$ corresponds to $x\in (0,x_1]$ with $x_1=-1/\log\pi\e_1$, and
$|\re\bt-\nobreak\frac12|\leq \e_2$ to $|y|\leq \e_2$. We take
\be \label{Ftk-def}F_{\!t,k}(x,y) \= \Bigl( \frac{2t}{2\pi k -
A_+\bigl(\al(x),\bt(y,t)\bigr)}, \frac{ t\, M_+\bigl( \al(x),\bt(y,t)
\bigr)}{2\pi k - A_+\bigl(\al(x),\bt(y,t)\bigr)}\Bigr)\,. \ee
By taking $k_1$ sufficiently large, we can make the denominators
in~\eqref{Ftk-def} as large as we want, in particular non-zero. So
$F_{\!t,k}$ is real-analytic on $(0,x_1]\times[-\e_2,\e_2]$. By
defining $\al(0)=0$ we extend $F_{\!t,k}$ to a $C^\infty$-function on
$[0,x_1]\times[-\e_2,\e_2]$.

Since $C_+$ has no zeros or poles in $\Om_I(\e_1,\e_2)$ we have
$M_+=\oh(1)$ and $A_+=\oh(1)$. So for all sufficiently large~$k$ we
have
\be \label{Ftk-incl}
F_{\!t,k}\Bigl(
[0,x_1]\times[-\e_1,\e_1]\Bigr)\;\subset(0,x_1)\times(-\e_2,\e_2)\,.
\ee

To show that $F_{\!t,k}$ is contracting it suffices to bound the
partial derivatives of the two components. For the first component we
have $\left(2\pi k-A_+\right)^2$ in the denominator, which can be
made large. In the numerator we have the derivatives
\[ \partial_x A_+ \= \frac{d\al}{dx}\, \partial_\al A_+ \;\ll\;
e^{-1/x}\, x^{-2}\, \al \;\ll\; 1\,, \qquad
\partial_y A_+ \=-\partial_\s A_+ \= \oh (1)
\]
The factor $\al$ is due to the fact that $C_+$, and hence $A_+$ is
even in~$\al$. The factor $t$ is bounded, since $t\in I$. For the
other component we proceed similarly.

Controlling~$k$, we can make all partial derivatives small. So
$F_{\!t,k}$ is contracting for all $k\geq k_1$ with a suitable~$k$.
So for a given~$t$ there is exactly one point
$(\al,\bt)\in \Om_I(\e_,\e_2)$ satisfying~\eqref{+eq}. the fixed
point is in the region where $F_{\!t,k}$ is real-analytic, jointly in
its variables and in the parameter~$t$. Hence the fixed point is a
real-analytic function of~$t$ by the analytic implicit function
theorem. (See, {\sl e.g.}\, \cite{KrPa}, Theorem~6.1.2.)
\end{proof}
In this lemma, we do not get information concerning the sets $V_k$
with $k$ under the bound~$k_1$. If $0\in I$ we can normalize the
argument $A_+$ like we did in the previous subsection.
\begin{lem}\label{lem+0}For sufficiently small $\e_1,\e_2,\e_3>0$ the
solution set of \eqref{Y+eq} in the set
\[ \Om(\e_1,\e_2,\e_3)\=(0,\e_1]\times \Bigl( \bigl[
\frac12-\e_2,\frac12+\e_2\bigr]\times i(0,\e_3]\Bigr)
\]
is equal to the union of real analytic curves
\[ t\mapsto \bigl(\al_k(t),\s_k(t)+it\bigr)\qquad(t\in [0,\e_3])\,,\]
with $k\geq 1$.
\end{lem}
\begin{proof}If $\e_1$, $\e_2$ and $\e_3$ are sufficiently small, then
$C_+$ has no singularities or zeros in the closure of
$\Om(\e_1,\e_2,\e_3)$ in $\RR\times\CC$. So $\log Y_+$ can be defined
holomorphically on a neighborhood of $\bigl(0,\frac12\bigr)$ in
$\CC^2$ that contains $\Om(\e_1,\e_2,\e_3)$. We choose the branch
that has the following expansion at $\bigl(0,\frac12\bigr)$:
\be -2\log\pi\,\bigl(\bt-\frac12\bigr) - 4\,(\log 2)^2\,
\bigl(\bt-\frac12\bigr)^2+ \oh\Bigl( \bigl(\bt-\frac12)^3\Bigr) +
\oh\bigl(\al^2\bigr)\,. \ee
For the behavior along $\{0\}\times\CC$ we use~\eqref{sm0-expl}. We
also use that $C_+$ is even in~$\al$. (See \ref{fact-fe}
and~\eqref{escm-split}.)
So in this lemma we can work with
\badl{AM+exp} M_+(\al,\s+it)&\= -2\,\log\pi\, \bigl(\s-\frac12\bigr) -
4\,(\log 2)^2\, \bigl(\s-\frac12\bigr)^2 +4\,(\log 2)^2\, t^2\\
&\qquad\hbox{}
+ \oh\Bigl( \bigl(\bt-\frac12)^3\Bigr) + \oh\bigl(\al^2\bigr)\,,\\
A_+(\al,\s+it)&\= -2t\,\log\pi- 8 \,(\log 2)^2
t\,\bigl(\s-\frac12\bigr)+ \oh\Bigl( \bigl(\bt-\frac12)^3\Bigr) +
\oh\bigl(\al^2\bigr)\,. \eadl
Now the parameter $k$ in the previous lemma can be anchored to this
choice of the argument.

For $t\in (0,\e_3]$ and $k\geq 1$ we define $F_{\!t,k}$ as in
\eqref{Ftk-def}, and revisit the estimates in the proof of the
previous lemma. We cannot use $k$ to make the denominator large.

By adapting the $\e$'s we can make $M_+$ and $A_+$ as small as we want
on~$\Om(\e_1,\e_2,\e_3)$. (See the expansions in \eqref{AM+exp}.) In
particular, we arrange
\[ |A_+|\;\leq\; 2\pi-4\quad\text{ and }\quad |M_+|\;\leq \; 1\,.\]
Then the denominator satisfies
$D\isdef 2\pi k - A_+ \geq 2\pi -A_+ >2$, hence $2t/D < t\leq \e_3$,
and $tM_+/D \leq \e_3/(2\pi-\nobreak 4)$. Arranging
$\e_3< x_1=-1/\log\pi\e_1$ and $\e_3< (2\pi-\nobreak 4)\e_2$, we
get~\eqref{Ftk-incl}.

To get all partial derivatives of $F_{\!t,k}$ small, we have to work
with the numerators of the derivatives, since we have lost control
over the denominators, except for the lower bound~$2$. In the
numerators we meet the following factors:
\begin{align*}
&\frac{\partial A_+}{\partial x}\,,& \quad
&t\,\frac{\partial M_+}{\partial x}\,,&\quad
&t\, M_+ \,\frac{\partial A_+}{\partial x}\,,\\
&\frac{\partial A_+}{\partial y}\,,& \quad
&t\,\frac{\partial M_+}{\partial y}\,,&\quad
&t\, M_+ \,\frac{\partial A_+}{\partial y}\,.
\end{align*}
We have $t=\oh(\e_3)$ and $M_+ = \oh(1)$, so we can concentrate on the
derivatives of $A_+$ and $M_+$ with respect to $x$ and~$y$. Both
derivatives $\frac{\partial A_+}{\partial \al}$ and
$\frac{\partial M_+}{\partial \al}$ are $\oh(\al)$ by~\eqref{AM+exp},
which is controlled by~$\e_1$. Since
$\frac{d\al}{dx} \ll x^{-2} e^{-1/x} = \oh(1)$, all contribution in
the first line can be made small by decreasing $\e_1$ and $\e_2$.

We have $\frac{\partial A_+}{\partial y}
=-\frac{\partial A_+}{\partial \s} = \oh(t) = \oh(\e_3)$.
Further, $\frac{\partial M_+}{\partial y} 
= -\frac{\partial M_+}{\partial \s} = \oh(1)$.
This derivative occurs only multiplied with $t= \oh(\e_3)$. So
adapting $\e_1$ and $\e_2$, and then $\e_3$, taking also into account
the requirements $\e_3<-1/\log\pi\e_1$ and
$\e_3<(2\pi-\nobreak 4)\e_2$, we can arrange that all partial
derivatives are very small on~$\Om(\e_1,\e_2,\e_3)$.

So $F_{\!t,k}$ is contracting. Its fixed point
$\bigl(\al_k(t),\s_k(t)\bigr)$ gives the sole point
$(\al,\bt)\in V_k$ with $\im\bt =t$. It depends on $t$ in a
real-analytic way.\smallskip

Now let $k\leq 0$. Suppose that there is a sequence
$(\al_n,\bt_n)=(\al_n,\s_n+\nobreak it_n)\in V_k$ that tends to
$\bigl(0,\frac12\bigr)$. The expansions in \eqref{AM+exp} imply that
$A_+(\al_n,\bt_n)=o(1)$ and $M_+(\al_n,\bt_n)=o(1)$. Then \eqref{+eq}
implies that $2\, t_n\, \log\pi\al_n$ tends to $-2\pi k$. If we
ensure that $\e_1<\frac1\pi$, we have $\log\pi\al_n \leq 0$, which
shows that $k\leq -1$ is impossible. Let $k=0$. We have by
\eqref{Y+eq} and \eqref{AM+exp}
\[ \log\pi\al_n \= - \log\pi + \oh\bigl( \s_n-\frac12\bigr)\,,\]
in contradiction to $\log\al_n \rightarrow-\infty$.
\end{proof}

\begin{proof}[Proof of Theorem~\ref{thm-res0}.]Lemma~\ref{lem+0} has
given us the solution curves parametrized by~$k\geq 0$. We have to
prove the invertibility of the~$\al_k$, the asymptotic behavior,
possibly further decreasing the~$\e$'s. Then the inequality
$\s_k <\frac12$ follows from~\eqref{as+0sg} (perhaps after adapting
the~$\e$'s).

We consider on of the curves in Lemma~\ref{lem+0}. In the next
computations we omit the index~$k$. Differentiation of the
relation~\eqref{+eq} with respect to~$t$ we get the system
\[\renewcommand\arraystretch{1.2}
\begin{pmatrix} \frac{2t}\al - \frac{\partial A_+}{\partial \al} &
-\frac{\partial A_+}{\partial\s} \\
\frac{2\s-1}\al - \frac{\partial M_+}{\partial\al} & 2\log \pi\al -
\frac{\partial
M_+}{\partial \s}\end{pmatrix}\,
\begin{pmatrix} \dot\al\\
\dot\s\end{pmatrix}
\=
\begin{pmatrix} -2\,\log\al + \frac{\partial A_+}{\partial t}\\
\frac{\partial M_+}{\partial t}\end{pmatrix}\,. \]
Here we consider $A_+$ and $M_+$ as functions of the three variables
$\al$, $\s$ and~$t$. By a dot we indicate differentiating with
respect to~$t$. The determinant of this system is
\[ \Bigl( \frac{2t}\al + \oh(\al) \Bigr)\,\Bigl(2\log\pi\al +
\oh(1)\Bigr)-\oh(\al^{-1})\,\oh(t) \= \frac{4t\,\log\pi\al}\al \,
\Bigl(1+ \oh\bigl( (\log\al)^{-1}\bigr)\Bigr)\,. \] Adapting the
$\e$'s we arrange that this quantity if negative. Then we have
\begin{align*}
\dot\al &\= \frac{\al}{4t\,\log\pi\al}\,
\Bigl(1+\oh\bigl((\log\al)^{-1}\bigr) \Bigr)\\
&\qquad\hbox{} \cdot
\biggl( \Bigl( 2\log\al +
\oh(1)\Bigr)\,\Bigl(-2\,\log\al+\oh(1)\Bigr)+ \oh(t)\, \oh(1)
\biggr)\\
&\= \frac{ \al\, |\log\pi\al|}t\;\Bigl( 1+ \oh\bigl(
(\log\al)^{-1}\bigr)\Bigr)\,.
\end{align*}
So we can arrange that $\al_k'(t)=\dot\al>0$. This shows that the
real-analytic function $t\mapsto \al(t)$ on $[0,\e_3]$ has a real
analytic inverse $t_k$ on some interval
$(\th_k,\z_k]\subset (0,\e_1]$.

In the proof of Lemma~\ref{lem+0} we have already arranged that
$|A_+|<2\pi-4$. Hence $2t\, \log\pi\al <-4$ and
$\log\pi\al < -\frac 2t$. So $\al\downarrow 0$ as $t\downarrow0$,
which shows that $\th_k=0$ for all $k\geq 1$.

To derive the asymptotic expansions, we consider $\s$ and $t$ as
functions of $\al$ along the curve parametrized by~$k\geq 1$. We omit
the subscript $k$ in the computations. We have along this curve
\bad \label{L0rel}
2\,\bigl(\bt-\frac12\bigr)\, \log\pi\al &\= -2\pi i k-2\,\log\pi
\,\bigl(\bt-\frac12\bigr) - 4\,(\log2)^2\,\bigl(\bt-\frac12\bigr)^2\\
&\qquad\hbox{}
+ \oh\Bigl(\bigl(\bt-\frac12\bigr)^3\Bigr) + \oh(\al^2)\,. \ead
This implies that $\left(\bt-\frac12\right)\, \log\pi\al=\oh(1)$,
hence
\[\bt-\frac12 \= \oh\bigl((\log\pi\al)^{-1}\bigr) \=
\oh\bigl((\log\pi^2\al)^{-1}\bigr)\,.\]
Next we get
\[\left( \bt-\frac12\right)\, \log\pi^2\al = -\pi i k - 2\,(\log
2)^2\, \left( \bt-\frac12\right)^2 + \oh\left(
(\log\pi^2\al)^{-3}\right)\,. \]
We write $\ell=\log\pi\al$ and $L=\log\pi^2\al=\ell+\log\pi$, which
are large negative quantities. We obtain
\begin{align*}
\bt&-\frac12 \= \frac{-\pi i k}L - \frac{2\,(\log
2)^2\,(\bt-\frac12)^2}L + \oh\bigl( L^{-4}\bigr)
\displaybreak[0]
\\
&\= -\frac{\pi i k}L - \frac{2\,(\log 2)^2}L\, \Bigl( -
\frac{\pi^2k^2}{L^2}+\oh(L^{-3})\Bigr)
\= - \frac{\pi i k}{\ell+\log\pi} +\frac{2(\pi k \log
2)^2}{(\ell+\log\pi)^3}
+ \oh\bigl( \ell^{-4}\bigr)\displaybreak[0]\\
&\= - \frac{\pi i k}\ell + \frac{\pi i k \,\log\pi}{\ell^2} +
\frac{2(\pi k \log2)^2-\pi i k(\log\pi)^2}{\ell^3}+ \oh(\ell^{-4})\,.
\end{align*}
Taking real and imaginary parts gives the asymptotic relations
\eqref{as+0t} and~\eqref{as+0sg}.
\end{proof}

\begin{proof}[Proof of Theorem~\ref{thm+I}.] Lemma~\ref{lem+I} gives
curves of this type in a region $\Om_I(\e_1,\e_2)$ near
$\{0\} \times\bigl( \frac12+\nobreak iI\bigr)$. In this region $C_+$
has no singularities or zeros, so $\log Y_+$ and its argument $A_+$
can be chosen in a continuous way. There seems no way to connect the
branches of $\log Y_+$ more globally, so we may as well normalize
$A_+$ by $A_+(0,\bt) = A(\bt)$ for $\bt\in \frac12+iI$, where $A$ and
$M$ are as in the theorem. We have on $\Om_I(\e_1,\e_2)$
\bad M_+(\al,\s+it)&\= M\bigl(\frac12+it\bigr) +
\oh\bigl(\s-\frac12\bigr) + \oh(\al^2)\,,\\
A_+(\al,\s+it)&\= A\bigl(\frac12+it\bigr)+ \oh\bigl(\s-\frac12\bigr) +
\oh(\al^2)\,. \ead
The relations \eqref{+eq} hold for points $(\al,\s+\nobreak it)$ on
curves with number $k\geq k_1$. (we omit the index~$k$.)
\[ \log\pi\al \= \oh\bigl(t^{-1}\bigr)\,,\qquad \frac12-\s\= \oh
\bigl( |\log\pi\al|^{-1}\bigr)\,.\]
Working more precisely, we get
\[ 2t\,\log\pi\al \= -2\pi k + A\bigl( \frac12+it\bigr)+ \oh \bigl(
|\log\pi\al|^{-1}\bigr)\,,\]
and hence $\log\pi\al \leq C\, k$ for some positive $C$. This gives
\[ \log\pi\al \= -\frac{\pi k}t + \frac{A\bigl(\frac12+it\bigr)}{2t} +
\oh\bigl(k^{-1}\bigr)\,. \]
This gives~\eqref{asal+t}. We also have
\[ 2\bigl(\s-\frac12\bigr)\Bigl( -\frac{\pi k}t + \oh(1) \Bigr) \=
M\bigl(\frac12+it\bigr) + \oh\bigl(k^{-1}\bigr)\,,\]
hence
\begin{align*}
\bigl(\frac12-\s\bigr) &\= \frac{ M\bigl(\frac12+it\bigr) + \oh\bigl(
k^{-1}\bigr)}{ 2\pi k t^{-1}+\oh(1)
} \= \frac t{2\pi k} \Bigl( M\bigl(\frac12+it\bigr)+\oh\bigl(
k^{-1}\bigr)\Bigr)\,\Bigl( 1 + \oh(t/k) \Bigr)\\
&\= \frac{t\, M\bigl(\frac12+it\bigr)}{2\pi k} + \oh
\bigl(k^{-2}\bigr)\,,
\end{align*}
which is~\eqref{assg+t}.
\end{proof}

\begin{proof}[Proof of Theorem~\ref{thm+tch}.]First we consider
several statements equivalent to the statement that a curve as in
Theorem~\ref{thm+I} touches the central line at a point $t\in I$. Of
course, this is equivalent to $\s_k(t)=\frac12$. In~\eqref{+eq} we
see that it implies that
$M_+\bigl(\al_k(t),\frac12+\nobreak it\bigr)=0$. And since each curve
with $k\geq k_1$ has $t$ as a parameter, touching the central line in
$\frac12+it$ is equivalent to
$M_+\bigl(\al_k(t),\frac12+\nobreak it\bigr)=0$. By~\eqref{MA+def}
and~\eqref{Y+def} this is equivalent to $\bigl|
C_+(\al_k(t),\frac12+\nobreak it)\bigr|=1$. At points with
$\al\in (-1,1)$ and $\re\bt=\frac12$ the matrix $\scm^+(\al,\bt)$
in~\eqref{escm-split} is unitary. So the statement is also equivalent
to $C_{0,\infty}\bigl( \al_k(t) , \frac12+\nobreak it\bigr)=0$.

In~\eqref{sm0-expl} we see that $\bt\mapsto C_{0,\infty}(0,\bt)$ has a
simple zero at $\bt=\bt_\ell\isdef \frac12+it_\ell$. Since $\bt_\ell$
is not a zero of $Z(0,\cdot)$, we know from~\ref{fact-us-escm} that
all matrix elements of $\scm$ are holomorphic at $(0,\bt_\ell)$, in
particular $C_{0,\infty}$ is holomorphic at~$(0,\bt_\ell)$. So we
have $C_{0,\infty}(\al,\bt) = \ld(\al,\bt)
\, P(\bt-\nobreak\bt_\ell,\al)$ on a neighborhood of~$(0,\bt_\ell)$,
where $P$ is a polynomial in $\bt-\bt_\ell$ with coefficients that
are holomorphic in $\al$ on a neighborhood of $0$ in~$\CC$, vanishing
at~$0$, and where $\ld$ is holomorphic without zeros. (This follows
from the Weierstrass preparation theorem. See, {\sl e.g.},
Corollary~6.1.2 in~\cite{Ho}.) The restriction of $C_{0,\infty}$ to
the complex line $\{0\}\times\CC$ has a zero of order~$1$
at~$\bt_\ell$, hence $P(X,\al)=X-i\eta(\al)$, with $\eta$ holomorphic
on a neighborhood of $0$ in~$\CC$ and $\eta(0)=0$. Since
$C_{0,\infty}(\al,\bt)$ is even in~$\al$, its zero set is also
invariant under $\al\mapsto -\al$. Hence $\eta$ is an even function.
{}From~\ref{fact-fe} and~\eqref{escm-split} it follows that
$(\al,\bt) \mapsto \overline{C_{0,\infty}(\bar\al,1-\bar\bt)}$ has
the same zero set as $C_{0,\infty}$. This implies
$\overline{\eta(\bar\al)} = \eta(\al)$. So $\eta(\al)\in \RR$ for
real~$\al$. The power series expansion of $\eta$ at~$0$ starts with
$\eta(\al) = \eta_2\al^2+\cdots$, with $\eta_2\in \RR$.

The asymptotic behavior of $\al_k$ in \eqref{asal+t} shows that the
curve $t\mapsto \bigl( \al_k(t),t)$ and the curve $\al\mapsto
\bigl(\al,t_\ell+\nobreak\eta(\al)\bigr)$ intersect each other for
all sufficiently large~$k$. We call the intersection
point$\bigl( a_k, t_\ell+\nobreak\dt_k)$. So we have
\[ a_k \= \al_k\bigl( t_\ell+\dt_k)\,, \quad \dt_k \= \eta(a_k)\,.\]
Furthermore, $C_{0,\infty}\bigl( a_k,\bt_\ell+i\dt_k\bigr)=0$, hence
the curve with number~$k$ touches the central line at
$\bt_\ell+i\dt_k$. \smallskip

Now we carry out estimates as $k\rightarrow\infty$.

Theorem~\ref{thm+I} gives
$\al_k(t) \ll \exp\bigl( \oh(1)-\pi k/t\bigr)$ uniformly on $I$. In
particular, $a_k = \oh(k^{-n})$ for each $n\geq 0$. In particular
$a_k\downarrow0$. Then $\dt_k = \eta(a_k)$ implies that
$\dt_k = \oh(k^{-2n})$ and also tends to zero.

We have
\begin{align*}
\frac{A(\frac12+it_\ell+i\dt_k)-2\pi k}{2(t_\ell+\dt_k)} &\= \frac
1{2t_\ell} \,\Bigl( 1+ \oh(\dt_k/t_\ell)\Bigr) \, \Bigl(
A\bigl(\frac12+i t_\ell\bigr) -2\pi k+ \oh(\dt_k) \Bigr)\\
&\= \Bigl( \frac{A\bigl(\frac12+it_\ell)-2\pi k}{2t_\ell} + \oh(\dt_k)
\Bigr)\, \Bigl( 1+\oh(\dt_k)\Bigr)\\
&\= \frac{A\bigl(\frac12+it_\ell)-2\pi k}{2t_\ell} + \oh\bigl(
k^{1-2n}\bigr)\,.
\end{align*}
Hence
\begin{align*}
a_k &\= \al_k\bigl(t_\ell+\bt_k)
\= \frac 1\pi\, \exp\Bigl( \frac{A\bigl(\frac12+it_\ell)-2\pi
k}{2t_\ell} + \oh\bigl( k^{1-2n}\bigr) \Bigr)\, \Bigl( 1+
\oh\bigl(k^{-1}\bigr)\Bigr)\\
&\= \frac1\pi \, e^{A(\frac12+it_\ell)/2t_\ell-\pi k/t_\ell}\, \Bigl(
1
+\oh\bigl( k^{1-2n}\bigr)\Bigr)\, \, \Bigl( 1+
\oh\bigl(k^{-1}\bigr)\Bigr)\\
&\= \frac 1\pi\, e^{A(\frac12+it_\ell)/2t_\ell-\pi k/t_\ell}\, \Bigl(
1+ \oh\bigl(k^{-1}\bigr)\Bigr)\,.
\end{align*}
We obtain
\begin{align*}
\dt_k &\= \eta(a_k) \= \eta_2 \, a_k^2 + \oh(a_k^4)
\= \frac{\eta_2}{\pi^2}\, e^{A\bigl(\frac12+it_\ell\bigr)/t_\ell -
2\pi k/t_\ell}\, \Bigl( 1+ \oh\bigl(k^{-1}\bigr) + \oh\bigl(
k^{-4n}\bigr)\,.
\end{align*}
This gives~\eqref{dtk-as}.

We know that $\s_k(t)\leq \frac 12$ for all $t\in I$,
by~\ref{fact-pert-scm}. So the points where the real-analytic curves
touch the central line are tangent points.
\end{proof}

\subsection{Curves of resonances originating higher up on the central
line}\label{sect-loops}We turn to a tentative explanation of curves
of resonances like those in Figure \ref{fig-loop2}
and~\ref{fig-loop2-detail}. We cannot prove that these loops
necessarily exist, and have to be content with a result that depends
on a number of assumptions
\begin{ass}\label{ass}
\ (1) \ Let $\bt_0=\frac12+it_0$ with $t_0>0$. {\em We assume that the
conjugated scattering matrix $\scm^+$ in~\eqref{escm-split} has a
singularity at $(0,\bt_0)$.} So $\bt_0$ is a zero of the unperturbed
Selberg zeta-function $Z(0,\cdot)$ on the central line. (Not all such
unperturbed eigenvalues need to be related to a singularity of
$\scm^+$.)
\par\noindent(2) \ The singularity of $\scm^+$ at $(0,\bt_0)$ is as
simple as possible, with a common denominator for all matrix
elements. To make this precise, {\em we assume that there are
holomorphic functions $p$, $r_{0,0}$, $r_{0,\infty}$ and $r_+$ on a
neighborhood of $0$ in~$\CC$ that all vanish at $\CC$, such that on a
neighborhood $\Om$ of $(0,\bt_0)$ in~$\CC^2$}
\bad \label{Cq} C_{0,0}(\al,\bt)&\=
\gm_{0,0}(\al,\bt)\,\frac{\bt-\bt_0-r_{0,0}(\al)}{\bt-\bt_0-p(\al)}\,,\\
 C_{0,\infty}(\al,\bt)&\=
\gm_{0,\infty}(\al,\bt)\,\frac{\bt-\bt_0-r_{0,\infty}(\al)}
 {\bt-\bt_0-p(\al)}\,,\\
 C_+(\al,\bt)&\=
\gm_+(\al,\bt)\,\frac{\bt-\bt_0-r_+(\al)}{\bt-\bt_0-p(\al)}\,,\\
\ead
{\em where the $\gm$'s are holomorphic on~$\Om$ without zeros
in~$\Om$. }
\par\noindent(3) \ Since $\scm^+(\al,\bt)$ is even in $\al$, the zeros
sets of the matrix elements and the set of singularities are
invariant under $\al\mapsto -\al$. So $p$ and the $r$'s are even
functions. We assume that already the first terms in their power
series expansions are non-zero and all different: {\em $p''(0)$,
$r_{0,0}''(0)$, $r_{0,\infty}''(0)$ and $r_+''(0)$ are four different
non-zero complex numbers.}
\par\noindent(4) \ The restriction of $\scm^+$ to the complex line
$\{0\}\times \CC$ is equal to
\[ \begin{pmatrix} \gm_{0,0}(0,\bt)& \sqrt 2\, \gm_{0,\infty}(0,\bt)\\
\sqrt 2\, \gm_{0,\infty}(0,\bt)& \gm_+(0,\bt)
\end{pmatrix}\,. \]
See~\eqref{ucu0}. {\em We assume that for $\bt=\bt_0$ all elements of
this matrix are non-zero.}
\end{ass}
Most of these assumptions mean that ``nothing special happens'', and
hence seem not too unreasonable. Only the assumption that all matrix
elements have the same set of singularities might be considered to be
really restrictive.

\begin{lem}\label{lem-npr}Under the assumptions~\ref{ass} the
neighborhood $\Om$ of $(0,\bt_0)$ can be chosen such that
\[ \Om \cap \Bigl( \RR \times\bigl( \frac12+i\RR \bigr) \Bigr) \=
\bigl\{(0,\bt_0) \bigr\}\,. \]
\end{lem}
\begin{proof}Let $\al_1\in \RR$ and $\re\bt_1=\frac12$,
$(\al_1,\bt_1)\in \Om$. The restriction
$\bt\mapsto \scm^+(\al_1,\bt)$ on $\frac12+i\RR$ is a family of
unitary matrices, hence any singularity is of indeterminate type.
Such singularities occur discretely, so taking $\Om$ sufficiently
small the sole possibility is $(\al_1,\bt_1)=(0,\bt_0)$.
\end{proof}

\begin{lem}Under the assumptions~\ref{ass} there is a neighborhood of
$0$ in~$\CC$ such that for all $\al$ in that neighborhood:
\be r_+(\al) \= - \overline{r_{0,0}(\bar \al)}\,,\qquad
r_{0,\infty}(\al) \= - \overline{r_{0,\infty}(\bar \al)}\,. \ee
\end{lem}
\begin{proof}
We have $\det\scm^+ = C_{0,0}\, C_+ - 2 \, C_{0,\infty}^2$. Hence
\begin{align*}
\bigl(\bt&-\bt_0-p(\al)\bigr)^2\, \det\scm^+(\al,\bt)
\\
&\= \gm_{0,0}(\al,\bt)\, \gm_+(\al,\bt) \,
\bigl(\bt-\bt_0-r_{0,0}(\al)\bigr)\, \bigl(\bt-\bt_0-r_+(\al\bigr)\\
&\qquad\hbox{}
- 2\, \gm_{0,\infty}(\al,\bt)^2
\,\bigl(\bt-\bt_0-r_{0,\infty}(\al)\bigr)^2
\end{align*}
is holomorphic on $\Om$, and its restriction to the complex line
$\al=0$ has a zero of at most order~$2$ at $\bt=\bt_0$. We use the
Weierstrass preparation theorem to write it in the form
$\dt(\al,\bt)\, Q(\bt-\nobreak \bt_0,\al)$, with $\dt$ holomorphic
without zeros on $\Om$ and $Q$ a polynomial in its first variable of
degree at most~$2$ with coefficients that are holomorphic functions
of~$\al$ vanishing at $\al=0$, and with highest coefficient equal
to~$1$. So we have
\[ \det \scm^+(\al,\bt) \= \dt(\al,\bt) \, \frac{Q(\bt-\bt_0,\al)}
{P(\bt-\bt_0,\al)^2}\,,\]
where $P(T,\al)=T-p(\al)$. We define an involution $K\mapsto K^\ast$
in the space of polynomials in $T$ with holomorphic coefficients
in~$\al$ by
$K^\ast(T,\al) =(-1)^{\textrm{degree}\,K}\,  \overline{K(-\bar T,\bar \al)}$.
So $P^\ast(T,\al) = T+\overline{p(\bar \al)}$. Lemma~\ref{lem-npr}
implies that $P^\ast\neq P$.

The relation $\overline{\det \scm^+(\bar \al, 1-\bar \bt)} =
\det\scm^+(\al,\bt)^{-1}$, from \ref{fact-fe} and~\eqref{escm-split},
implies
\[ \overline{\dt\bigl(\bar \al,1-\bar\bt\bigr)} \,
\frac{(-1)^{\mathrm{degree}\,Q}\,
Q^\ast(\bt-\bt_0,\al)}{P^\ast(\bt-\bt_0,\al)^2} \=
\dt(\al,\bt)^{-1}\, \frac{P(\bt-\bt-0,\al)^2}{Q(\bt-\bt_0,\al)}\,,\]
and hence
\begin{align*} (-1)^{\mathrm{degree}\,Q}&\, \overline{\dt\bigl(\bar
\al,1-\bar\bt\bigr)}\, \dt(\al,\bt) \, Q^\ast(\bt-\bt_0,\al)\,
Q(\bt-\bt_0,\al)
\\
&\= P^\ast(\bt-\bt_0,\al)^2 \, P(\bt-\bt_0,\al)^2\,.
\end{align*}
On the right is a fourth degree polynomial in $\bt-\bt_0$ with highest
coefficient~$1$. This means that on the left we have also a
polynomial of degree four, and that the highest coefficient is also
equal to~$1$. So the product of the sign and the two $\dt$'s is equal
to~$1$. (We note that not only $\dt(\al,\bt)$ but also
$\dt(\bar\al, 1-\nobreak\bar\bt)$ is non-zero for $(\al,\bt)$
sufficiently close to $(0,\bt_0)$.) Hence
 $Q^\ast\; Q = P^2 \, (P^\ast)^2$, and since $Q^\ast$ and $Q$ have the
 same degree, this degree is equal to~$2$.

The polynomials $P$ and $P^\ast$ are irreducible, hence $Q$ is equal
to one of $P^2$, $P\, P^\ast$, and $(P^\ast)^2$. Hence
\[ \det \scm^+(\al,\bt) \= \dt(\al,\bt) \, \Bigl(
\frac{P^\ast(\bt-\bt_0,\al)}{P(\bt-\bt_0,\al)}\Bigr)^\ell\,,\]
with $\ell\in \{0,1,2\}$.

We also define $R_{0,0}(T,\al) = T-r_{0,0}(\al)$, so
$R_{0,0}^\ast(T,\al) = T + \overline{r_{0,0}(\bar \al)}$, and
similarly for $r_{0,\infty}$ and $r_+$. Considering the relation
$\overline{\scm^+(\bar\al, 1-\bar \bt)} = \scm^+(\al,\bt)^{-1}$
itself we arrive at
\be\label{ar} R_{0,0}^\ast\, (P^\ast)^{\ell-1} \= P^{\ell-1}\, R_+\,,
\qquad R_{0,\infty}^\ast\, (P^\ast)^{\ell-1} \= R_{0,\infty} \,
P^{\ell-1}\,. \ee

If $\ell=0$ we find $R_{0,0}^\ast \, P= R_+\, P^\ast$. Assumption~(3)
implies that $P$ and $R_+$ are different polynomials of the first
degree with highest coefficient~$1$. So $P=P^\ast$, but we have
already shown that that is impossible. So $\ell\in \{1,2\}$.

If $\ell=2$ then $R_{0,0}^\ast\, P^\ast = P\, R_+$, and $P$ divides $
R_{0,0}^\ast$, and hence $P=R_{0,0}^\ast$, and then also
$P^\ast= R_+$. We obtain $R_+= (R_{0,0}^\ast)^\ast=P^\ast=R_{0,0}$,
in contradiction with assumption~(3). Hence $\ell=1$, and
$R_{0,0}^\ast = R_+$, which gives the relation
$\overline{r_{0,0}(\bar \al)}=-r_+(\al)$. {}From~\eqref{ar} we now
also get $R_{0,\infty}^\ast=R_{0,\infty}$, hence
$r_{0,\infty}(\al) = - \overline{r_{0,\infty}(\bar \al)}$.
\end{proof}

\begin{lem}\label{lem-el}Under the assumptions~\ref{ass} there are
$\e>0$ and a neighborhood $U$ of~$\bt_0$ in~$\CC$ such that for each
$\al\in (0,\e]$ there is exactly one $\z(\al) \in U$ such that
$X\bigl(\al, \z(\al) \bigr)\, C_+\bigl(\al,\z(\al) \bigr)=1$.

We have $\lim_{\al\downarrow0}\z(\al)=\bt_0$, and as the point
$\z(\al)-\bt_0$ moves to zero it passes the line segment between
$r_+(\al)$ and $p(\al)$ infinitely often, circling around $r_+(\al)$
in negative direction or around $p(\al)$ in positive direction.
\end{lem}
\begin{proof}We write $\bt=\bt_0+z$. On $\Om$, the equation
$X\, C_+=1$ becomes
\[ (\pi\al)^{2it_0+2z}\, \tilde\gm(\al,\bt_0+z)\,
\frac{z-r_+(\al)}{z-p(\al)}\=1\,,\]
with $\tilde\gm(\al,\bt) 
= \Gf\bigl(\frac12-\nobreak\bt)\,\Gf(\bt-\nobreak\frac12)\, \gm_+(\al,\bt)$.
We take $\e>0$ and a simply connected, connected neighborhood $U$
of~$\bt_0$ such that $(0,\e]\times U \subset \Om$. In the course of
the proof we adapt $\e$ and~$U$. For sufficiently small~$\e>0$ the
two points $r_+(\al)$ and $p(\al)$ are different points of~$U$ for
all $\al\in (0,\e]$. The corresponding points
$\bigl(\al,\bt_0+\nobreak r_+(\al)\bigr)$ and
$\bigl(\al,\bt_0+\nobreak p(\al)\bigr)$ cannot be in the solution set
of~$X\, C_+=1$.

Taking a logarithm, we get the equation
\[ 2(it_0+z)\log\pi\al + \log\tilde\gm(\al,\bt_0+z) +
\log\frac{z-r_+(\al)}{z-p(\al)} \,\equiv\, 0 \bmod 2\pi i \ZZ\,, \]
where for $\log\tilde\gm(\al,z)$ we use a continuous choice of the
logarithm. The logarithm of the quotient $\frac{z-r_+}{z-p}$ is
multivalued on $(0,\e]\times U$ and has branch points. We go over to
the covering space by the parametrization
\be\label{zu} z\=z(u) \= \frac{e^u\, p(\al) - (\pi \al)^{2it_0} \,
\tilde\gm(0,\bt_0)\, r_+(\al)}{e^u
- (\pi\al)^{2it_0} \,\tilde\gm(0,\bt_0)}\,.\ee
The variable $u$ runs over a suitable subset of~$\CC$. The equation
becomes
\be \label{u-eq} 2 \, z(u)\, \log\pi \al +
\log\frac{\tilde\gm\bigl(\al,\bt_0+z(u)\bigr)} {\tilde\gm(0,\bt_0)} +
u \= 0\,. \ee
On the covering space the ambiguity modulo $2\pi i \ZZ$ is hidden in
the choice of the variable~$u$.

To make precise what is a suitable set in the $u$-plane, we use
assumption~(4). With $\bt\in \frac12+i\RR$ all elements of the
unitary matrix are non-zero, and hence have absolute value between
$0$ and~$1$. This implies that $0<|\tilde \gm(0,\bt)|<1$, and we can
take $\dt_-<0$ such that $e^{\dt_-}>|\tilde \gm(0,\bt_0)|$. We
consider the region determined by $\dt_-\leq \re u \leq \dt_+$ with
some $\dt_+>0$. For these values of $u$ the denominator of $z(u)$
satisfies
\be\label{z-den} \bigl| e^u - (\pi \al)^{2it_0}\,\tilde\gm(0,\bt_0)
\bigr| \;\geq\; c_1\= c_1(\dt_-)\;>\;0 \,. \ee
Hence we find
\be|z(u)| \;\leq \; \frac{ e^{\dt_+} \, |p(\al)| +
|\tilde\gm(0,\bt_0)|\, |r_+(\al)|}{c_1} \;\leq \; c_2\, \al^2\=
c_2(\dt_-,\dt_+)\, \al^2\,. \ee

We have
\be z'(u) \= \frac{ (\pi \al)^{2it_0}\, \tilde\gm(0,\bt_0)\,
\bigl(r_+(\al)-p(\al)\bigr)\, e^u}{\bigl(e^u - (\pi
\al)^{2it_0}\,\tilde\gm(0,\bt_0)\bigr)^2}\,, \ee
with can be estimated in the following way:
\begin{align}\nonumber
 z'(u)&\;\ll\; \frac{ |\tilde\gm(0,\bt_0)|\,\al^2\,
e^{\dt_+}}{c_1^2}\,,
\displaybreak[0]\\
|z'(u)|&\;\leq\; c_3\,\al^2 \= c_3(\dt_-,\dt_+)\, \al^2\,.
\end{align}

We consider on the region $\dt_-\leq |\re u| \leq \dt_+$ the
holomorphic function
\be F(u) \= -2\,z(u)\, \log\pi\al -\log
\frac{\tilde\gm\bigl(\al,\bt_0+z(u)\bigr)}{\tilde\gm(0,\bt_0)} \,.
\ee
We have
\begin{align*} \log
\frac{\tilde\gm\bigl(\al,\bt_0+z(u)\bigr)}{\tilde\gm(0,\bt_0)}&\;\ll\;
\tilde \gm(0,\bt_0)^{-1}\, \Bigl( \frac{\partial^2
\tilde\gm}{\partial\al^2}
(0,\bt_0)\cdot \al^2 + \frac{\partial\tilde\gm}{\partial\bt}(0,\bt_0)
\cdot z(u)
\Bigr)\\
&\;\ll\; \al^2 + |z(u)| \;\ll\; \al^2\,.
\end{align*}
(We have used that $\tilde\gm$ is even in~$\al$.)
We get $ |F(u)| \;\leq \; 2\, c_2\,\al^2\, |\log\pi\al|+ \oh(\al) $,
and hence there is $c_4=c_2(\dt_-,\dt_+)$ such that
\be \bigl| F(u)\bigr| \;\leq \; c_4\, \al^2\,.\ee
Taking $\e$ such that $\e^2\,c_4 \in [ \dt_-,\dt_+]$ we arrange that
$F$ maps the set
\be E\= \{ u\in \CC\;:\; \dt_-\leq \re u \leq \dt_+,\, -\e^2\, c_4\leq
\im u \leq \e^2\, c_4\} \ee
into itself.

The solutions of~\eqref{u-eq} in $E$ are precisely the fixed points
of~$F$ in~$E$. The question is whether $F$ is contracting on~$E$.
\begin{align*}
F'(u) &\= -2\, z'(u)\, \log\pi\al - \frac{\partial\tilde\gm}{\partial
\bt}\bigl(\al,\bt_0+z(u)\bigr)\; z'(u)\\
&\;\ll\; \al^2 \, \log\pi \al + \oh(1)\, \al^2\,.
\end{align*}
Hence there is $c_5=c_5(\dt_-,\dt_+)$ such that $|F'(u)|\leq c_5 \,
\al^2\,|\log\pi\al|$ on~$E$. We can adapt $\e$ such that
$c_5\, \al^2\, |\log\pi\al|\leq c_6$ with some $c_6\in (0,1)$. So $F$
is contracting on~$E$, and we find a unique fixed point
$u(\al)\in E$. Projecting back we find a unique solution
$\z(\al) = \bt_0+z\bigl(u(\al)\bigr)$ of the equation $X\, C_+=1$.
\smallskip

The denominator in \eqref{zu} stays away from zero, by~\eqref{z-den}.
Since $\re u$ is bounded we have
$z(\al)\isdef z\bigl(u(\al) \bigr)=\oh(\al^2)$. So $z(\al)$ tends
to~$0$, and $\z(\al)$ tends to~$\bt_0$. The relation
\be \frac{z(\al)-r_+(\al)}{z(\al)-p(\al)} \= e^{u(\al)} \,(\pi
\al)^{-2it_0} \, \tilde\gm(0,\bt_0)^{-1} \ee
shows that the argument of $\frac{z-r_+}{z-p}$ tends to $\infty$ as
$\al\downarrow0$. So $z(\al)$ crosses between $p(\al)$ and $r_+(\al)$
infinitely often, such that a continuous choice of the argument
increases. This means that $z(\al)$ turns around $r_+(\al)$ in
positive direction, or around $p(\al)$ in negative direction.
\end{proof}

\begin{lem}\label{lem-tcl}Under the assumptions~\ref{ass} there is a
decreasing sequence $(\al_k)_{k\geq k_0}$ of positive numbers with
limit zero such that for all $\al \in (0,\al_{k_0})$
\be X\bigl(\al,\bt_0+r_{0,\infty}(\al)\bigr)\,
C_+\bigl(\al,\bt_0+r_{0,\infty}(\al)\bigr)\=1 \ee
if and only if $\al$ is one of the~$\al_k$.

The $\al_k$ satisfy
\be\label{alk-e} \al_k \= \frac1\pi\, e^{-(2\pi k+s_0)/2t_0}\, \Bigl(
1+\oh\bigl(k\,e^{-2\pi k/t_0}\bigr)\Bigr)\,, \ee
for some real number $s_0$.
\end{lem}
The value of $s_0\bmod 2\pi\,\ZZ$ depends on the functions
$r_{0,\infty}$, $r_+$ and $p$. We do not know it explicitly. The
choice of $s_0$ in its class and the choice of the parameter $k$ are
related.
\begin{proof} We consider the function
\[ f(\al) \= X\bigl(\al,\bt_0+r_{0,\infty}(\al)\bigr) \,
C_+\bigl(\al,\bt_0+r_{0,\infty}(\al)\bigr)\]
on an interval $(0,\e_1]$ such that
$\bigl(\al,\bt_0+r_{0,\infty}(\al)\bigr)\in \Om$. For small real
values of~$\al$ the values of $r_{0,\infty}(\al)$ are purely
imaginary. Fact \ref{fact-fe} and \eqref{escm-split} imply that the
matrix $\scm^+\bigl(\al,\bt_0+\nobreak r_{0,\infty}(\al)\bigr)$ is
unitary. In \eqref{Cq} we see that
$C_{0,\infty}\bigl(\al,\bt_0+\nobreak r_{0,\infty}(\al)\bigr)=0$. So
$\scm^+\bigl( \al,\bt_0+\nobreak r_{0,\infty}(\al)\bigr)$ is a
unitary diagonal matrix. This implies that $|f(\al)|= 1$ for
$\al\in (0,\e_1]$.

We make a continuous choice of $\al\mapsto s(\al)$ for
$\al\in [0,\e_1)$ such that
\be e^{is(\al)} \= \frac{\Gf(\frac12-\bt_0-r_{0,\infty}(\al))}
{\Gf(-\frac12+\bt_0+r_{0,\infty}(\al))}\,
\gm_+\bigl(\al,\bt_0+\nobreak r_{0,\infty}(\al)\bigr)\,
\frac{r_{0,\infty}(\al)-r_+(\al)}{r_{0,\infty}(\al)-p(\al)}\,. \ee
We note that $s$ is an even function. The number $s_0$ in the
statement of the lemma is equal to $s(0)$. Now
\be \label{ad}
 a(\al)\= 2t_0\, \log\pi\al - 2i\, r_{0,\infty}(\al)\,
\log\pi\al+s(\al)\qquad \text{ for } \al\in (0,\e_1)\ee
is a continuous choice of the argument of $f(\al)$. We have
$a(\al) = 2t_0\, \log\pi\al + \oh(1)$ as $\al\downarrow0$. The
derivatives of the three term in~\eqref{ad} are
\[ \frac {2t_0}\al\,,\quad \oh(\al\log\pi\al)\,,\quad \oh(\al)\,.\]
So for sufficiently small $\e_1$ the argument of $f(\al)$ is
monotonely decreasing to $-\infty$ as $\al\downarrow0$, and there is
a sequence $(\al_k)_{k\geq k_0}$ of elements of $(0,\e)$ decreasing
to zero such that for each $k\geq k_0$
\be 2t_0\, \log\pi\al_k - 2i r_{0,\infty}(\al_k) \, \log\pi \al_k +
s(\al_k) \=
-2\pi \, k\,.\ee
So the points $\bigl(\al_k,\bt_0+\nobreak r_{0,\infty}(\al_k)\bigr)$
are solutions of $X\, C_+=1$, and the $\al$'s between two successive
$\al_k$ do not satisfy this equation.

Since $\al_k=\oh(1)$, we have directly
$c_1 \, e^{-\pi k/t_0}\leq \al_k\leq c_2 \, e^{-\pi k/t_0}$ with
positive $c_1$ and $c_2$. This gives
\[ \log\pi\al_k \= \frac{-2\pi k-s(\al_k)}{2t_0+\oh(\al_k^2)} \=
-\frac{2\pi k+s_0}{2t_0} + \oh(k\,e^{-2\pi k/t_0})\,. \]
This gives~\eqref{alk-e}.
\end{proof}

\begin{prop}\label{prop-loops}Under the assumptions~\ref{ass} there is
one curve $\al\mapsto \bigl(\al,\z(\al)\bigr)$ on an interval
$(0,\e_1)$ with limit $(0,\bt_0)$ for which
$Z\bigl(\al,\z(\al)\bigr)=0$ for all $\al\in (0,\e_1)$.

The curve touches the central line in $\bigl(\al_k,\z(\al_k)\bigr)$
for a monotone sequence of $\al_k$ in $(0,\e_1]$ with limit~$0$.
Hence the $\z(\al_k)$ are eigenvalues. The $\al_k$ satisfy the
relation~\eqref{alk-e}.

As $\al$ runs through $ (\al_{k+1},\al_k)$ the point
$\bigl(\al,\z(\al)\bigr)$ describes a curve in the region
$\re \bt<\frac12$. The corresponding $\z(\al)$ are resonances, and
$\z(\al_k)-\bt_0$ is proportional to $\al_k^2$.
\end{prop}
\begin{proof}The $\bigl(\al,\z(\al)\bigr)$ in Lemma~\ref{lem-el} are
solutions of $X \, C_+ =1$. They satisfy
$C_{0,\infty}\bigl(\al,\z(\al)\bigr)=0$ precisely for the sequence
$(\al_k)$ in Lemma~\ref{lem-tcl}. Proposition~\ref{prop-dn} shows
that at these points the scattering `matrix'~$D_{0,0}$ has a
singularity of indeterminate type. Hence the $\z(\al_k)$ are
eigenvalues. For the other $\al$, the function $D_{0,0}$ has value
$\infty$ at $\bigl(\al,\z(\al)\bigr)$, so $\re \bt<\frac12$ by
\ref{fact-pert-scm}, and $\z(\al)$ is a resonance.
\end{proof}

\rmk The computations reported in~\cite{Fr} provide us with six curves
of resonances tending to a point on the central line with positive
imaginary part. For two of these curves Figure~8.26 in~\cite{Fr}
suggests that indeed $\z(\al_k)-\bt_0$ is proportional to~$\al_k^2$.

\section{Spectral theory of automorphic forms}\label{sect-spth}
We discuss parts of the spectral theory of automorphic forms to
explain the ``facts'' in~\S\ref{sect-facts}. There is a vast
literature on the spectral theory of Maass forms. In 1966/'67
Roelcke, \cite{Roe}, has given a thorough discussion; however the
continuation of Eisenstein series was not yet fully known at that
time. Of later literature we mention Venkov, \cite{Ve90}, Iwaniec,
\cite{Iw95}, and Bump, \cite{Bu98}. The material we need is also
present in Chapters VI and~VII of Hejhal's treatment \cite{He83} of
the Selberg trace formula.
\smallskip

For the purpose of this section we need not know how the Selberg
zeta-function is defined by a product involving the closed geodesics.
That description is important for the computations in~\cite{Fr}. Here
we only need to know that and how it is related to automorphic forms
(Maass forms). In \S\ref{sect-zSz} we will quote results concerning
this relation.

\subsection{$\Gm_0(4)$, cusps and characters}The group $\Gm=\Gm_0(4)$
leaves invariant the set $\proj\QQ$ of cusps in the boundary of the
upper half-plane~$\uhp$. This set consists of three
$\Gm_0(4)$-orbits, for which we choose the representatives $0$,
$\infty$ and $-\frac12$. Each cusp $\x\in \proj\QQ$ is left invariant
under the group $\Gm_\x$, which is generated by
$\pi_\x = g_\x\matc 1101 g_\x^{-1}$. The matrices $\pi_\x$ and our
choice of $g_\x$ are as follows:
\be\renewcommand\arraystretch{1.7}
\begin{array}{|c|ccc|}\hline
\x&0& \infty&-\frac12\\ \hline
\pi_\x & \matr10{-4}1&\matc1101&\matr31{-4}{-1} \\
g_\xi&\matr0{-\frac12}20&1& \matr10{-2}1\\
\hline
\end{array}
\ee

The group $\Gm$ is free on the generators $\pi_\infty$ and $\pi_0$. So
the character group $\Gm^\vee$ is isomorphic to $(\CC^\ast)^2$, and
can be parametrized by $(\al,\al')\in(\CC/\ZZ)^2$:
\be \ch_{\al,\al'}(\pi_\infty)\= e^{2\pi i \al}\,,\qquad
\ch_{\al,\al'}(\pi_0)=e^{2\pi i \al'}\,. \ee
The character $\ch_{\al,\al'}$ is unitary if and only if
$(\al,\al')\in
(\RR\bmod\ZZ)^2$.

Fraczek uses in~\cite{Fr} the family $\al\mapsto \ch_\al=\ch_{\al,0}$.
The family of characters considered in \S3 of~\cite{Se90} is
$\al\mapsto \ch_{0,\al/2\pi}$. We shall see below that $\ch_{\al,0}$
and $\ch_{0,\al}$ are conjugate, and hence lead to the same Selberg
zeta-function. The characters of $\Gm(2)$ used in~\S6 of~\cite{PS92}
and in~\cite{PS94} correspond to $\ch_{0,\al}$ under the
straightforward isomorphism $\Gm(2) \cong \Gm_0(4)$ by conjugation.

The group~$\Gm=\Gm_0(4)$ is invariant under conjugation by elements of
a subgroup of $\PGL_2(\RR)$ generated by $\matc 1{\frac12}01$,
$\matr0{-\frac12}20$ and $\matr{-1}001$, inducing symmetries of the
characters described by $(\al,\al') \mapsto
(\al,-\al-\nobreak\al')$, $(\al,\al')\mapsto (\al',\al)$, and
$(\al,\al')\mapsto(-\al,-\al')$. The Selberg zeta-function is
unchanged under such automorphisms. Hence the results for the family
$\ch_\al$ are also valid for Selberg's family. Moreover, we see that
$\ch_\al$ and $\ch_{-\al}$ are conjugated, and hence have the same
Selberg zeta function.

Conjugation with
\be\label{jdef}
j\isdef\matr{-1}0{-2}1=
\matr{-1}001\matc0{-\frac12}20\matr1{\frac12}01\matc0{-\frac12}20\ee
leaves $\ch_\al$ invariant.

All cusps of~$\Gm$ are singular for the trivial character~$\ch_0$.
This means that $\ch_0(\pi_\x)=1$ for each cusp~$\xi$. For
$\al\in \CC\setminus \ZZ$, the cusps in the $\Gm$-orbit of $0$ stay
singular, the cusps in the other orbits become regular:
$\ch_\al(\pi_\xi)\neq 1$.
\smallskip

For general $\al\in \CC$ the values $\ch_\al\Bigl(\matc abcd\Bigr)$
are not given by an explicit formula in terms of the matrix elements.
The only way to compute the character is by recursion. We have
$\ch_\al(\gm) = e^{i\al \Om(\gm)}$, where $\Om:\Gm\rightarrow \ZZ$ is
the group homomorphism determined by $\Om(\pi_\infty)=1$,
$\Om(\pi_0)=0$. For later use we discuss a few properties of~$\Om$.

Since $\pi_0\in \ker\Om$, the value $\Om\Bigl(\matc abcd \Bigr)$ is
determined by the upper row of the matrix, and is determined by
\bad \Om(1,0) &\= 0\,,&\quad \Om(-a,-b)&\= \Om(a,b)\,,\\
\Om(a,b)&\= \Om(a-4b,b)\,,& \Om(a,b)&\= \Om(a,b-a)+1\,, \ead
where $a,b\in \ZZ$, $a\equiv 1\bmod 2$ and $(a,b)=1$. With induction
we derive that
\be \label{Om-est}\bigl|\Om(a,b) \bigr|\;\leq \; |b|\,. \ee
Other relations are
\be \Om(a,-b)\=-\Om(a,b)\,,\qquad \Om(d,-b)\=-\Om(a,b)\quad(ad\in
1+4b\ZZ)\,. \ee
The latter relation follows from $\Om(\gm^{-1})=-\Om(\gm)$, the former
can be shown by induction on $|b|$.

\subsection{Maass forms}\label{sect-mf}By an automorphic form on
$\Gm=\Gm_4(0)$ for the character $\ch_\al$ with eigenvalue
$\ld\in \CC$ we mean a real-analytic function $u$ on the upper
half-plane~$\uhp$ such that
\begin{enumerate}
\item[a)] $u(\gm z) = \ch_\al(\gm)\, u(z)$ for all $\gm\in \Gm$ and
$z\in \uhp$;
\item[b)] $\Dt u = \ld\, u$ for
$\Dt = -y^2\partial_x^2-y^2\partial_y^2$
(with $x=\re z$ and $y=\im z$).
\end{enumerate}
We consider only automorphic forms of weight zero, and have no power
of $c\tau+d$. Any distribution satisfying
$\Dt u=\bt(1-\nobreak\bt)\, u$ is automorphically a real-analytic
function by elliptic regularity. It is convenient to parametrize the
eigenvalue by $\ld=\bt-\bt^2$; the complex number $\bt$ is called the
spectral parameter. In the previous sections we allowed ourselves to
call $\bt$ the eigenvalue.

For given $(\al,\bt)\in \CC^2$ the space of such automorphic forms is
huge. We define the (finite dimensional) subspace $\mf(\al,\bt)$ of
\emph{Maass forms} by imposing \emph{polynomial growth} at the cusps:
\begin{enumerate}
\item[c)] There is a constant $A=A(u)\in \RR$ such that
$u\bigl( g_\x  z)= \oh(y^A)$ as $y\rightarrow \infty$ for all cusps
$\x= g_\x\infty$ of~$\Gm$.
\end{enumerate}
It suffices to impose this condition for the representatives $0$,
$\infty$ and $-\frac12$ of the $\Gm$-orbits of cusps. We note that
$\mf(\al,1-\nobreak \bt)=\mf(\al,\bt)$.

Inside $\mf(\al,\bt)$ is the space $\mf^0(\al,\bt)$ of \emph{cusp
forms}, characterized by \emph{quick decay} at the cusps:
\begin{enumerate}
\item[c')] $u(g_\x z) = \oh(y^{-A})$ as $y\rightarrow\infty$ for all
$A=A(u)\in \RR$.
\end{enumerate}

Each automorphic form $u$ for $\ch_\al$ with spectral parameter~$\bt$
has an absolutely converging Fourier expansion at each cusp $\x$ of
the form
\be\label{Fe} u(g_\x z) \= \sum_{n\equiv \k_\x(\al)\bmod 1} F^\x_n
u(z)\,, \ee
where $F^\x_n u$ satisfies $\Dt F^\x_nu= \bt(1-\nobreak\bt)F^\x_nu$
and $F^\x_nu(z+\nobreak x') = F^\x_n u(z)\, e^{2\pi i n x'}$. The
number $\k_\x(\al)$ is such that
$\ch(\pi_\x) =e^{2\pi i  \k_\x(\al)}$; we use $\k_{-1/2}(\al)=-\al$,
$\k_0(\al)=0$, $\k_\infty(\al) = \al$.

The Fourier terms are in a $2$-dimensional space, depending on $n$ and
$\bt$. If $\k_\x(\al)=0$ and $\bt\neq \frac12$, then $F_0^\x u$ is a
linear combination of $y^{1-\bt}$ and $y^\bt$. The Fourier terms
inherit the growth conditions from~$u$. For $u\in \mf(\al,\bt)$ and
$\re n\neq 0$ this restricts $F_n^\x u$ to be a multiple of
\be\label{omdef} \om(n,\bt;z) \= 2 \,(\e n)^{1/2}\, e^{2\pi i n x}\,
\sqrt y\, K_{\bt-1/2}(2\pi \e n y)\,, \ee
with $\e=\sign\re n$. Of course, $\e n=|n|$ if $\al$ and hence $n$ are
real. The factor $ 2(\e n)^{1/2}$ allows $n$ and $\al$ to be non-real
complex, and keeps the notation consistent with \S4.2.8
in~\cite{Br94} (except that
$\bt_{\mathrm{here}}=s_{\mathrm{there}}+\frac12$).

We can characterize the Maass forms in the space of all automorphic
forms by the condition that $F_n^\x u$ is a multiple of $\om(n,\bt)$
for $\x\in \{0,\infty,-\frac12\}$. For cusp forms there is the
additional condition that $F_0^\x u=0$ for all $\x$ with
$\k_\x(\al)=0$.\smallskip

Associated to conjugation by $j=\matr{-1}0{-2}1$ in~\eqref{jdef} is
the involution $\iota:z\mapsto \frac{\bar z}{2\bar z-1}$ in~$\uhp$,
and the involution $Jf(z) = f(\iota z)$ in the functions on~$\uhp$.
Conditions a) and~b)
are preserved by $J$. The space of automorphic forms for $\ch_\al$
with spectral parameter $\bt$ splits in the $1$- and
$(-1)$-eigenspace of~$F$. We speak of \emph{even}, respectively
\emph{odd} automorphic forms.

We have
\bad\label{iotag}
\iota(g_{-1/2}z)&\=\pi_0^{-1}g_\infty(-\bar z)\,,\qquad \iota(g_\infty
z)\=\pi_0^{-1}g_{-1/2}(-\bar z)\\
\iota(g_0z) &\= g_0(-\bar z-\txtfrac12) \,. \ead
This implies that $J$ respects the growth conditions c) and c'). Hence
we have a direct sum decomposition
\be \mf(\al,\bt)\= \mf_+(\al,\bt)\oplus \mf_-(\al,\bt)\,, \ee
in \emph{even} and \emph{odd Maass forms}, and similarly for the
spaces of cusp forms.

The Fourier terms $F^\x_n u(z)$ satisfy
\bad \label{FtJ}
(F^0_nJu)(z)&\= (-1)^n (F^0_{-n}u)(-\bar z)\,,\\
(F^\infty_nJu)(z) &\= (F^{-1/2}_{-n}u)(-\bar z)\,,\quad
(F^{-1/2}_nJu)(z) \= (F^\infty_{-n}u)(-\bar z)\,. \ead
This can be checked as follows:
\begin{align*} (Ju)\bigl(g_\infty&(x+iy) \bigr)
\= u\bigl(\iota g_\infty(x+iy)\bigr)
\= u\bigl(\pi_0^{-1}\, g_{-1/2}(-x+iy)\bigr)
\displaybreak[0]\\
&\= \ch_\al(\pi_0)^{-1}\, \sum_{n\equiv -\al\bmod 1} e^{2\pi i
n(-x)}\, F^{-1/2}_n u(iy)
\displaybreak[0]\\
&\= \sum_{n\equiv \al\bmod 1} e^{2\pi i n x} \, F^{-1/2}_{-n}u(iy)\,,
\end{align*}
which shows that
$F_n^\infty J u(x+\nobreak iy) = F^{-1/2}_{-n}u(x+\nobreak iy)$. The
other relations go similarly.

We conclude from~\eqref{FtJ} that $F^0_0 f=0$ and
$F^\infty_\al f = -F^{-1/2}_{-\al}f$ for odd automorphic forms $f$
with character $\ch_\al$, and $F^\infty_\al f = F^{-1/2}_{-\al}f$ for
even automorphic forms.

\subsection{Eisenstein series and Poincar\'e series} There are many
Maass forms with polynomial growth that are no cusp forms. Let $\al$
be real, and let $\x$ be a cusp that is singular for the
character~$\ch_\al$, {\sl i.e.}, $\ch_\al(\pi_\x)=1$. The three
representatives $0$, $\infty$ and $-\frac12$ of cuspidal $\Gm$-orbits
are all singular if $\al=0$, and if $0<\al<1$ only $0$ is singular.
For $\re\bt>1$ the \emph{Eisenstein series}
\be\label{eisdef} E^\x_\al(\bt;z) \= \sum_{\gm\in \Gm_\x\backslash
\Gm} \ch_\al(\gm)^{-1} \, \bigl( \im g_\x^{-1} \gm z\bigr)^\bt \ee
defines an element of $\mf(\al,\bt)$. It has a meromorphic
continuation as a function of~$\bt\in \CC$ representing a meromorphic
family of Maass forms providing an element of $\mf(\al,\bt)$ for each
value of $\bt$ at which it is defined.

For $\al=0$ we have three Eisenstein series. Each of these Eisenstein
series has a Fourier expansion at each of the three cuspidal
representatives. This gives nine Fourier terms of order~$0$, of the
form
\be F_0^\eta E^\x_0(\bt;g_\eta z) \= \dt_{\x,\eta}\, y^\bt +
C_0(\eta,\x;\bt)\, y^{1-\bt}\,. \ee
The coefficients $C_0(\eta,\x;\bt)$ are meromorphic functions
on~$\CC$. One collects them in a $3\times3$-matrix, with
$C_0(\eta,\x;\bt)$ in the row indexed by~$\eta$ and the column
indexed by~$\x$. This is the scattering matrix $\scm_0(\bt)$
in~\eqref{sm0-expl}. It is a unitary matrix for $\re\bt=\frac12$, and
satisfies the functional equation
$\scm_0(1-\nobreak\bt)\, \scm(\bt) = I$. The use of the word
``scattering'' comes from the view on Eisenstein series explained by
Lax and Phillips in~\cite{LP}.\smallskip

Since $\ch_\al=\ch_{\al+1}$ the other case to be considered is
$0<\al<1$. Now there is only one Eisenstein series, at the cusp $0$.
The Fourier term of order zero at this cusp has the form
\be\label{E00-D} F^0_0 E^0_\al(\bt;z) \= y^\bt + D_\al(\bt)\,
y^{1-\bt}\,. \ee
In this case, $D_\al(\bt)$ is called the scattering matrix (with size
$1\times1$). For general values of $\al\in (0,1)$ it cannot be
computed as explicitly as in~\eqref{sm0-expl}. It has absolute value
$1$ if $\re\bt=\frac12$, and satisfies the functional equation
$D_\al(1-\nobreak\bt)\, D_\al(\bt)=1$.\medskip

Other automorphic forms can be constructed in a similar way. The
Eisenstein family was based on $y^\bt=\im(z)^\bt$. Another
eigenfunction of $\Dt$ with eigenvalue $\bt(1-\nobreak\bt)$ is
\be\label{mudef} \mu(n,\bt;z) \= e^{2\pi i n z} \,y^\bt \,
\hypg11\bigl( \bt ;2\bt;4\pi n y)\,, \ee
which is holomorphic in $n\in \CC$ and meromorphic in~$\bt$, with
singularities at $\bt\in \frac12\,\ZZ_{\leq 0}$.

For $\al \in \RR$, $n\equiv \k_\x(\al)\bmod 1$, and $\re\bt>1$ the
\emph{Poincar\'e series}
\be\label{Pald} P^{\xi,n}_{\al}(\bt;z) \= \sum_{\gm\in
\Gm_\xi\backslash\Gm}\ch_{\al}(\gm)^{-1}\, \mu(n,\bt;g_\x^{-1}\gm
z)\ee
defines an automorphic form for the character $\ch_\al$ with spectral
parameter~$\bt$. It has a meromorphic continuation in $\bt$ as a
family of automorphic forms. We note that
$P^{0,0}_\al(\bt)=E^0_\al(\bt)$. For general combinations of $\al$
and $\bt$ the function $\mu(n,\bt;z)$ has exponential growth. For
$\re n \neq 0$
\be \mu(n,\bt;iy) \;\sim\; \frac{\Gf(2\bt)}{\Gf(\bt)}\, (4\pi \e
n)^{-\bt} \, e^{2\pi \e n y}\qquad(y\rightarrow\infty)\,, \ee
with $\e=\sign \re n$. (Use the asymptotic results in \S4.1
of~\cite{Sl}.) This gives $P^{\infty,\al}_\al(\bt;z)$ and
$P^{-\frac12,-\al}_\al(\bt;z)$ exponential growth; these Poincar\'e
series are no Maass forms, except for special combinations of the
parameters.

\subsection{Zeros of the Selberg zeta-function}\label{sect-zSz}
In the introduction of \S\ref{sect-res} we already mentioned that the
Selberg zeta-function $Z(\al,\bt)$ is defined for
$\al\in \RR\bmod \ZZ$ and is a meromorphic function of~$\bt$, and
have given references for it. In this paper we have no need to go
into the definition, but need the relation to automorphic forms,
which is summarized in Theorem~5.3 in Chapter~X of~\cite{He83},
p.~498.

We quote the results concerning zeros in the region $\im\bt>0$. These
results enable us to use the theory of automorphic forms to
understand aspects of the results of the computations in~\cite{Fr}.
\begin{enumerate}
\item[a)] At points $\bt$ on the central line~$\frac12+i(0,\infty)$
the function $Z(\al,\cdot)$ has a zero of order
$\dim \mf^0(\al,\bt)$.

These values are called \emph{eigenvalues} in \S\ref{sect-res}
and~\S\ref{sect-prfs}, although the cumbersome description `spectral
parameters of cusp forms' would be more correct.
\item[b)] At points with $\re\bt<\frac12$ and $\im \bt>0$ the function
$Z(\al,\cdot)$ has a zero of the same order as the zero of the
determinant of the scattering matrix at $1-\bar\bt$. These are the
\emph{resonances}.
\end{enumerate}

The space $\mf^0(\al,\bt)$ with $\al\in [0,1)$ can be non-zero if and
only if $\bt(1-\nobreak\bt)> 0$. For $\Gm_0(4)$ and the trivial
character ($\al=0$) it is known that $\bt(1-\nobreak \bt)>\frac14 $
if $\mf^0(0,\bt)\neq \{0\}$. See the theorem on p.~250 of \cite{Hu85}
for $\Gm^0(4)$, which is conjugate to $\Gm_0(4)$.

The determinant of the scattering matrix has a zero at $1-\bar \bt$ if
and only if it has a singularity at~$\bt$. For $0<\al<1$ the
scattering matrix is the coefficient $D_\al$ in~\eqref{E00-D}.

\subsection{Eisenstein families and Poincar\'e
families}\label{sect-fam}
The correspondence in~\S\ref{sect-zSz} is valid for each
$\al\in [0,1)$ separately. To use it to get insight in the behavior
of zeros of the Selberg zeta-function as $\al\downarrow0$, we need a
relation between automorphic forms for several characters. Such a
relation is provided by Theorem~10.2.1 in~\cite{Br94}, which we will
use twice.
\smallskip

The first application provides use with families of automorphic forms
that extend the Eisenstein series for the unperturbed situation:
There is a neighborhood $U$ of $(-1,1)$ in $\CC$ such that there are
three meromorphic families of automorphic forms $Q^{-1/2}(\al,\bt)$,
$Q^0(\al,\bt)$ and $Q^\infty(\al,\bt)$ on $U\times\CC$, uniquely
characterized by the Fourier terms of low order: For
$\x,\eta\in \{0,\infty,-\frac12\}$
\bad\label{EFour}
F^\eta_{\k_\eta(\al)} Q^\x(\al,\bt) &\= \dt_{\x,\eta}
\mu(\k_\eta(\al),\bt) + C_{\eta,\x}(\al,\bt)\,
\mu(\k_\eta(\al),1-\bt)\,,\\
F^\eta_n Q^\x(\al,\bt)&\= C_{\eta,\x}(n;\al,\bt) \, \om(n,\bt)
\quad \text{ if }n \equiv \k_\eta(\al)\bmod 1,\, n\neq
k_\eta(\al)\,,\ead
with meromorphic functions $C_{\eta,\x}$ and $C_{\eta,\x}(n)$ on
$U\times\CC$. The uniqueness holds in the sense that if $f$ is any
meromorphic family of automorphic forms on a neighborhood in~$\CC^2$
 of a point of $(-1,1)\times \CC$ and all Fourier terms of $f$ with
 $n\neq k_\x(\al)$ are multiples of $\om(n,\bt)$, then $f(\al,\bt)$
can be expressed as a linear combination of the $Q^\x(\al,\bt)$ in
the following way: Determine for $\x\in \{0,\infty,-\frac12\}$ the
coefficients $p_\x$ and~$q_\x$ in
$F^\x_{\k_\x(\al)} f(\al,\bt) = p_\x(\al,\bt)\, \mu(\k_\x(\al),\bt) + 
q_\x(\al,\bt)\, \mu(\k_\x(\al),1-\nobreak \bt)$; then
$ f(\al,\bt)=\sum_\x p_\x(\al,\bt) Q^\x (\al,\bt)$ as an identity of
meromorphic families of automorphic forms.

For $\al=0$ we have $\mu(0,\bt;z)=y^\bt$, so these Fourier terms
generalize the Fourier terms of Eisenstein series for $\al=0$. In
fact, Theorem~10.2.1 in~\cite{Br94} also states that the restriction
of $Q^\x(\al,\bt)$ to the complex line $\{0\}\times\CC$ exists, and
is the family $E^\x_0(\bt)$ of Eisenstein series. (For a meromorphic
family of functions restriction to a complex line might be
impossible, if that line is contained in the set of singularities of
the family.)

We form the \emph{extended scattering matrix} from the coefficients
in~\eqref{EFour}
\be \label{escm-def}
\scm(\al,\bt) \= \begin{pmatrix}
C_{0,0}(\al,\bt) &C_{0,\infty}(\al,\bt) & C_{0,-1/2}(\al,\bt) \\
C_{\infty,0}(\al,\bt) &C_{\infty,\infty}(\al,\bt) &
C_{\infty,-1/2}(\al,\bt) \\
C_{-1/2,0}(\al,\bt) &C_{-1/2,\infty}(\al,\bt) & C_{-1/2,-1/2}(\al,\bt)
\end{pmatrix}\,. \ee
The restriction of this family of matrices to $\al=0$ gives the
unperturbed scattering matrix, explicitly given in~\eqref{sm0-expl}.

A vector notation is convenient in this context. We write
\be \eis(\al,\bt) \= \bigl( Q^0(\al,\bt), Q^\infty(\al,\bt),
Q^{-1/2}(\al,\bt)
\bigr)\,, \ee
and the operator $\four\al$ of taking three Fourier terms:
\be \four\al \= \begin{pmatrix}F^0_0\\
F^\infty_\al\\
F^{-1/2}_{-\al}
\end{pmatrix}\= \begin{pmatrix}
\text{Fourier term at $0$ of order $0$}\\
\text{Fourier term at $\infty$ of order $\al$}\\
\text{Fourier term at $-\frac12$ of order $-\al$}
\end{pmatrix}
\,, \ee
in the notation of~\eqref{Fe}. Then \eqref{EFour} gives
\bad \label{FEv}\four\al \eis(\al,\bt) &\= \mum(\al,\bt)+
\mum(\al,1-\bt)
\,\scm(\al,\bt)\,,\\
\text{where }\quad \mum(\al,\bt)&\=\begin{pmatrix}\mu(0,\bt)&0&0\\
0&\mu(\al,\bt)&0\\
0&0&\mu(-\al,\bt)
\end{pmatrix}\,. \ead

The uniqueness implies the functional equations
\be\label{fea} \eis(\al,1-\bt) \= \eis(\al,\bt)\,
\scm(\al,1-\bt)\,,\qquad \scm(\al,1-\bt)\= \scm(\al,\bt)^{-1}\,. \ee
Checking the effect of conjugation on the basis functions $\mu$ and
$\om$ for Fourier terms, we can use the uniqueness also to show that
\be\label{fec} \overline{\scm(\bar \al,\bar \bt)} \= \scm(-\al,\bt)\,,
\ee
where $\overline{\scm(\bar \al,\bar \bt)} $ is understood to have
$\overline{C_{\eta,\x}(\bar \al,\bar \bt)}$ at position $(\eta,\x)$.
The Maass-Selberg relation (as discussed in, {\sl e.g.}, Theorem
4.6.5 in~\cite{Br94}) imply that
\be\label{feMS}\scm(-\al,\bt) \= \scm(\al,\bt)^t\,. \ee
(The $t$ means matrix transpose.) These are the identities
in~\ref{fact-fe}.
\medskip

From the relations in \eqref{FtJ} it follows that
$\four\al J u (z) = \Jm \four\al u(-\bar z)$, where
\be \label{Jdef}\Jm\=
\begin{pmatrix}1&0&0\\0&0&1\\0&1&0\end{pmatrix}\,. \ee
Hence
\begin{align*}
\four\al &J \eis(\al,\bt)(z)
\= \Jm\, \bigl( \mum(\al,\bt;-\bar z)+ \mum(\al,1-\bt;-\bar z)\,
\scm(\al,\bt)\bigr)
\displaybreak[0]\\
&\= \Jm\, \bigl( \mum(-\al,\bt;z)+\mum(-\al,1-\bt;z) \,
\scm(\al,\bt)\bigr)
\displaybreak[0]\\
&\= \mum(\al,\bt;z)\, \Jm + \mum(\al,1-\bt;z)\, \Jm\, \scm(\al,\bt)\,.
\end{align*}
So by uniqueness
\be J\eis(\al,\bt) \= \eis(\al,\bt)\,\Jm\,,\qquad
\Jm\,\scm(\al,\bt)\,\Jm \= \scm(\al,\bt)\,. \ee

Thus we have relations between the matrix elements of the extended
scattering matrix:
\be\label{escm-2}
\scm(\al,\bt) \= \begin{pmatrix}
C_{0,0}(\al,\bt) &C_{0,\infty}(\al,\bt) & C_{0,\infty}(\al,\bt) \\
C_{\infty,0}(\al,\bt) &C_{\infty,\infty}(\al,\bt) &
C_{\infty,-1/2}(\al,\bt) \\
C_{\infty,0}(\al,\bt) &C_{\infty,-1/2}(\al,\bt) &
C_{\infty,\infty}(\al,\bt)
\end{pmatrix}\,. \ee

We also see that we can form two even Eisenstein families and one odd
one:
\be \eis^+(\al,\bt) \= \eis(\al,\bt) \,
\begin{pmatrix}1&0\\0&\frac1{\sqrt2}\\
0&\frac1{\sqrt2}&\end{pmatrix}\,,\quad \eis^-(\al,\bt) \=
\eis(\al,\bt)\, \begin{pmatrix}
0& \frac1{\sqrt 2}&-\frac 1{\sqrt 2}\end{pmatrix}\,. \ee
Conjugation with
\be \Um\= \begin{pmatrix}1&0&0\\0&\frac1{\sqrt2}&\frac1{\sqrt2}\\
0&-\frac1{\sqrt2}&\frac1{\sqrt2}\end{pmatrix} \ee
gives the full decomposition with respect to parity:
\begin{align}
\Um \four \al \eis(\al,\bt) \Um^{-1} &\= \begin{pmatrix}
\four\al^+ \eis^+(\al,\bt) & 0 \\
0 &\four\al^- \eis^-(\al,\bt)
\end{pmatrix}\,,\\
\nonumber
\four\al^+ &\= \begin{pmatrix} F^0_0\\ \frac1{\sqrt 2}\bigl(
F^\infty_\al+ F^{-1/2}_{-\al}
\end{pmatrix}\,,\qquad \four \al^- = \frac1{\sqrt
2}(F^\infty_\al-F^{-1/2}_{-\al} \,;\\\label{UCUi}
\Um\,\scm(\al,\bt)\,\Um^{-1} &\= \begin{pmatrix}
C_{0,0}(\al,\bt) & \sqrt 2\, C_{0,\infty}(\al,\bt)&0\\
\sqrt 2\, C_{\infty,0}(\al,\bt) & C_+(\al,\bt)&0\\
0&0& C_-(\al,\bt)
\end{pmatrix}\\
&\= \begin{pmatrix} \scm_+(\al,\bt) & 0 \\ 0 & C_-(\al,\bt)
\end{pmatrix}
\,, \\
\label{Cpm} C_\pm(\al,\bt)&\= C_{\infty,\infty}(\al,\bt) \pm
C_{\infty,-1/2}(\al,\bt)\,.
\end{align}
The $2\times2$-matrix $\scm^+$ and the $1\times 1$-matrix $C_-$
inherit the properties of~$\scm$ in
\eqref{fea}--\eqref{feMS}.\bigskip

The family $\eis(\al,\bt)$ is based on the choice of $\mu(\cdot,\bt)$
and $\mu(\cdot,1-\nobreak\bt)$ as the basis for the Fourier terms
$F^0_0$, $F^\infty_\al$ and $F^{-1/2}_{-\al}$.
(Basis is used in the meromorphic sense. The basis functions may have
singularities, and they may be linearly dependent on subset of lower
dimension.)
If we restrict ourselves to a region with $\re \al\neq0$, we can also
use $\mu$ and $\om$ as a basis.

Applying Theorem~10.2.1 in~\cite{Br94} with this basis, we obtain on
$U_+\times \CC$, with $U_+$ a neighborhood of $(0,1)$ in the right
half-plane, a vector of meromorphic families of automorphic forms
\be\label{Poinc}
\poinc(\al,\bt) \= \bigl( P^0(\al,\bt), P^\infty(\al,\bt) ,
P^{-1/2}(\al,\bt)
\bigr)\ee
uniquely characterized by having a Fourier expansion in which all
terms have exponential decay, except for those in
$\four\al \poinc(\al,\bt)$, which has the form
\bad\label{FPv} \four\al \poinc(\al,\bt) &\= \mum(\al,\bt) +
\omm(\al,\bt) \, \dscm(\al,\bt)\,,\\
\omm(\al,\bt) &\=
\begin{pmatrix}\mu(0,1-\bt)&0&0\\
0&\om(\al,\bt)&0\\
0&0&\om(-\al,\bt)
\end{pmatrix}\,, \ead
with a $3\times 3$-matrix $\dscm(\al,\bt)$ of meromorphic functions on
$U_+\times\CC$. (Writing this out give a description like that in
\eqref{EFour}, with $\mu\bigl(\k_\eta(\al,1-\nobreak \bt\bigr)$
replaced by $\om\bigl(\k_\eta(\al),\bt\bigr)$.)For each
$\al_0\in (0,1)$ the restriction of $\poinc(\al,\bt)$ to the line
$\al=\al_0$ exists, and has the form
\[\bt\mapsto \bigl( E^0_{\al_0}(\bt),
P^{\infty,\al_0}_{\al_0}(\bt),P^{-1/2,-\al_0}_{\al_0}(\bt)
\bigr)\,,\]
with Poincar\'e and Eisenstein series as in \eqref{eisdef},
respectively~\eqref{Pald}.

The uniqueness implies that on $U_+\times \CC$ the family $\eis$ can
be expressed in terms of the family~$\poinc$. The factor
\be v(n,\bt)\= \pi^{-1/2}\, (\e\pi
n)^{\bt}\,\Gf\bigl(\frac12-\bt\bigr)\,,\ee
with $\e=\sign\re n$, is used in the relation
\be \om(n,\bt) \= v(n,\bt)\, \mu(n,\bt) + v(n,1-\bt)\, \mu(n,1-\bt)\,,
\ee
between both basis elements. Then
\bad \omm(\al,\bt) &\= \mum(\al,\bt) \,\Pm\, \Vm(\al,\bt) +
\mum(\al,1-\bt)\, \Vm(\al,1-\bt)\,,\\
\Pm&\= \begin{pmatrix}0&0&0\\0&1&0\\0&0&1\end{pmatrix}\,,\qquad
\Vm(\al,\bt) \= \begin{pmatrix}1&0&0\\
0&v(\al,\bt)&0\\0&0&v(\al,\bt)\end{pmatrix}\,. \ead

Uniqueness of $\eis$ on $U_+\times \CC$ implies that
$\poinc(\al,\bt) =\eis(\al,\bt)\,\, \Wm(\al,\bt)$ for some
meromorphic family of $3\times3$-matrices on $U_+\times \CC$.
Uniqueness of $\poinc$ implies that $\Wm(\al,\bt)$ is invertible (as
a meromorphic family of matrices; invertibility may fail on a subset
of lower dimension). Hence
\begin{align*}
\mum(\al,\bt)&+ \omm(\al,\bt)\, \dscm(\al,\bt) \= \Bigl(
\mum(\al,\bt)+\mum(\al,1-\bt)\,\scm(\al,\bt) \Bigr) \, \Wm(\al,\bt)\\
&\= \mum(\al,\bt) \,\Bigl( \Id - \Pm\, \Vm(\al,\bt)
\,\Vm(\al,1-\bt)^{-1} \scm(\al,\bt) \Bigr)
\, \Wm(\al,\bt)\\
&\qquad\hbox{}
+ \omm(\al,\bt) \, \Vm(\al,1-\bt)^{-1}\, \scm(\al,\bt)\, \Wm(\al,\bt)
\end{align*}
This implies
\bad \Wm(\al,\bt) &\= \Bigl( \Id-\Pm\, \Vm(\al,\bt)\,
\Vm(\al,1-\bt)^{-1} \, \scm(\al,\bt)\Bigr)^{-1} \,,\\
\dscm(\al,\bt) &\= \Vm(\al,1-\bt)^{-1}\, \scm(\al,\bt)\, \Wm(\al,\bt)
\ead

On this basis we can compute $\dscm$ in terms of the matrix elements
of $\scm$. These matrix elements are meromorphic on $U\times \CC$,
which contains $\{0\}\times\CC$. The matrix elements cannot be
extended across $\{0\}\times\CC$ in a meromorphic way by the presence
of a factor $(\pi\al)^{2\bt-1}$, which we have seen playing an
important role in~\S\ref{sect-prfs}.

We find
\be\label{UDUi} \Um \,\dscm\, \Um^{-1} \=
\begin{pmatrix} C_{0,0}+\frac{2\, X\, C_{0,\infty}\,
C_{\infty,0}}{1-X\, C_+}
& \frac{\sqrt 2\, C_{0,\infty}}{1-X\, C_+}& 0\\
\frac{\sqrt 2\, C_{\infty,0}}{\tilde v\,(1-X\, C_+)}&
\frac{C_+}{\tilde v\,(1-X\, C_+)}&0\\
0&0&\frac{C_-}{\tilde v\, (1-X\, C_-)}
\end{pmatrix}\,, \ee
where we have omitted $(\al,\bt)$, and use
$\tilde v \= v(\al,1-\nobreak \bt)$. The matrix element in position
$(0,0)$ is not changed by the conjugation; hence we have obtained the
formula for $D_{0,0}(\al,\bt)$ in~\eqref{D-def}, except for the
equality $C_{0,\infty}=C_{\infty,0}$, which we will prove
in~\S\ref{sect-ks}.

\subsection{Cusp forms and singularities of Poincar\'e
series}\label{sect-cpPs}Let $\al\in \RR$. A Poincar\'e series
$P^{\x,n}_\al(\bt)$ with $n\neq0$ has a singularity at
$\bt_0\in\frac12+i(0,\infty)$ if and only if there exists a cusp form
$f\in \mf^0(\al,\bt_0)$ with $F^\x_n f\neq 0$. See, {\sl e.g.},
Proposition~11.3.9 in~\cite{Br94}.

This implies \ref{fact-odd-eq}. Indeed, the third column of the matrix
in~\eqref{UDUi} describes a Fourier term of the family of Poincar\'e
series
$\frac1{\sqrt 2} \left( P^\infty_\al(\cdot) - P^{-1/2}_\al(\cdot)\right)$.
If for $\al\in (0,1)$, $\bt\in \frac12+i\RR$ we have
$X(\al,\bt)\, C_-(\al,\bt)=1$ and $C_-$ is holomorphic at
$(\al,\bt)$, then this Poincar\'e series has a singularity at $\bt$,
and hence there is a cusp form with this spectral parameter. This
cusp form arises as a residue of this odd Poincar\'e series, and
hence is in $\mf_-^0(\al,\bt)$.

\subsection{The extended scattering matrix and unperturbed cusp
forms}\label{sect-ecm-upcf}
We turn to the proof of fact~\ref{fact-us-escm}.

The restriction $\bt \mapsto \scm(0,\bt)$ exists and is equal to
$\scm_0$. Since $\scm_0(\bt)$ is a unitary matrix for $\bt$ on the
central line, this restriction has no singularities on
$\frac12+i\RR$. Nevertheless, this does not exclude the possibility
that it has singularities at points
$(0,\bt)\in \{0\}\times\bigl( \frac12+\nobreak i\RR\bigr)$ as a
meromorphic function of two variables. At such a point the zero set
and the set of singularities intersect.

\begin{prop}\label{prop-singEis}Let $\bt_0\in \frac12+i[0,\infty)$. If
$\scm$ has a singularity at $(0,\bt_0)$ then $Z(0,\bt_0)=0$; in
particular $\bt_0\neq\frac12$.
\end{prop}
\begin{proof}First let $\bt_0\neq\frac12$. Proposition~10.2.14
in~\cite{Br94} implies that $\scm $ has a singularity at $(0,\bt_0)$
if and only if $(\al,\bt) \mapsto \four\al \eis(\al,\bt)$ has a
singularity at $(0,\bt_0)$. Since $\mum(\al,\bt)$ and
$\mum(\al,1-\nobreak \bt)$ are holomorphic and linearly independent
on a neighborhood of $(0,\bt_0)$, $\four\al \eis(\al,\bt)$ being
singular at $(0,\bt_0)$ is also equivalent to $\scm$ having a
singularity at~$(0,\bt_0)$.

There is a holomorphic function $\ps$ on a neighborhood of $(0,\bt_0)$
such that $\ps\,\scm$ and $\ps \,\eis$ are holomorphic at
$(0,\bt_0)$. If we choose $\ps$ minimally with respect to
divisibility in the ring of germs of holomorphic functions at
$(0,\bt_0)$, then $\ps \,\eis$ is not identically zero on the zero
set of~$\ps$. We can write $\ps$ as the product of a holomorphic
function that is non-zero at $(0,\bt_0)$ and a Weierstrass
polynomial. (See, {\sl e.g.}, Corollary 6.1.2 in \cite{Ho}.) This
implies that there are $\ell\in \NN$ and $h$ holomorphic on a
neighborhood of~$0$ in~$\CC$ with $h(0)=0$ such that
$w\mapsto \ps\bigl( w^\ell,\bt_0+\nobreak h(w)\bigr)$ is the germ of
a curve through the zero set of~$\ps$. Then the holomorphic family
$\mathbf{f}(w) = (\ps\,\eis)\bigl(w^\ell, \bt_0+\nobreak h(w)\bigr)$
of elements of $\mf\bigl(w^\ell, \bt_0+\nobreak h(w) \bigr)$ is not
the zero family. But $\mathbf{f}(0)$ may be zero. However, since
$\mathbf{f}$ is a family of one variable, we can find $m\in \ZZ$ such
that $\mathbf{L}=\lim_{w\rightarrow 0} w^{-m}\mathbf{f}(w)$ is a
non-zero element of~$\mf(0,\bt_0)$.

The main Fourier term of $\mathbf{f}(w)$ is given by
\[ \four{w^\ell} \mathbf{f}(w) \= 0\; \mum\bigl(w,\bt_0+h(w)\bigr)+
\mum\bigl(w,\bt_0+h(w)\bigr)\, (\ps \scm)\bigl(w,\bt_0+h(w)\bigr)\,,
\]
since $\ps$ vanishes along
$w\mapsto \bigl(w^\ell,\bt_0+\nobreak h(w)\bigr)$. This implies that
in $\four 0\mathbf{L}$ occur only multiples of $y^{1-\bt_0}$, and not
of $y^{\bt_0}$.

We use the Maass-Selberg relation
(as discussed in, {\sl e.g.}, Theorem~4.6.5 in \cite{Br94}) to show
that $\mf(0,\bt_0)\bmod\mf^0(0,\bt_0)$ has dimension three. It has a
basis induced by the three Eisenstein series $E^0_0(\bt_0)$,
$E^\infty_0(\bt_0)$ and $E^{-1/2}_0(\bt_0)$, which all have a Fourier
term with $y^{\bt_0}$. This implies that $\four 0\mathbf{f}=0$, so
that $\mathbf{f}$ is a vector of cusp forms in $\mf^0(0,\bt_0)$. So
$\bt_0 \in \zspec(0)$.
\smallskip

Let $\bt_0=\frac12$. In this case we have to use another basis of the
Fourier terms, for instance the basis $\ld(\al,\bt)$,
$\mu(\al,1-\nobreak\bt)$ indicated in Lemma~7.6.14i) in~\cite{Br94}.
For $\bt\neq\frac12$ we have $(2\bt-\nobreak
1)\ld(\al,\bt) = \mu(\al,\bt)-\mu(\al,1-\nobreak\bt)$. With
\[ \lmm \=
\begin{pmatrix}\ld(0,\bt)&0&0\\
0&\ld(\al,\bt)&0\\
0&0&\ld(-\al,\bt)\end{pmatrix}\]
there is according to Proposition~10.2.4 in~\cite{Br94} a unique
meromorphic family $\R$ of vectors of automorphic forms such that
\[ \four\al \R(\al,\bt) \= \lmm(\al,\bt) + \mum(\al,1-\bt)\,
\escm(\al,\bt)\,,\]
with $\escm$ a matrix with meromorphic functions. All other Fourier
terms of the family are quickly decreasing, and the restriction
$\bt\mapsto \R(0,\bt)$ exists. The uniqueness implies that
$\R(\al,\bt) = \frac1{2\bt-1}\,\eis(\al,\bt)$, and hence
$\scm(\al,\bt) = (2\bt-\nobreak 1)\, \escm(\al,\bt) - 1$. If $\scm$
is singular at $(0,\frac12)$, then so are $\escm$ and~$\R$. Applying
the reasoning above, we get as a limit of $\R$ along its singular set
a non-zero automorphic form $\mathbf{L}$ with multiples of
$y^{1-\bt_0}=y^{\frac12}$ in its main Fourier terms. Now the
derivatives of $\bt\mapsto \eis(0,\bt)$ at $\bt=\frac12$ exist and
form a basis of $\mf(0,\frac12)/\mf^0(0,\frac12)$, with multiples of
$\ld(0,\frac12)$ in the main Fourier term. Hence $\mathbf{L}$ is a
non-zero cusp form. But $\mf^0(0,\frac12)=\{0\}$; \cite{Hu84}. So
$\escm$ and $\scm$ do not have a singularity at~$(0,\frac12)$.
\end{proof}

\subsection{Some Fourier coefficients of Poincar\'e
series}\label{sect-ks}
A comparison of the matrices in~\eqref{escm} and~\eqref{escm-2} shows
that there is still work to be done to obtain
fact~\ref{fact-sym-escm} completely. From~\eqref{feMS} we conclude
that it suffices to show that $C_{0,\infty}(\al,\bt)$ is even
in~$\al$. It takes us an amazing amount of work to show this.\medskip

The restriction of the function $C_{0,\infty}$ to $\{0\}\times\CC$ for
$\al\in [0,1)$ exists. It is given by the factor $C_0(0,\infty;\bt)$
in the Fourier term
\[ F^0_0 E^\infty_0(\bt) \= y^\bt+ C_0(0,\infty;\bt)\, y^{1-\bt}\,.\]
For $\al\in (0,1)$ it is related by \eqref{UDUi} to the coefficient
$D_{0,\infty}(\al,\bt)$ in
\[ F^0_0 P^{\infty,0}_\al(\bt) = \mu(\al,\bt) + D_{0,\infty}(\al,\bt)
\, \om(\al,\bt)\,.\]
In the region of absolute convergence these Fourier coefficients can
be described by explicit series. We use the formulas in \S5.2
of~\cite{Br94} to get for $\re \bt>1$:
\bad\label{CGf} C_0(0,\infty;\bt) &\= \frac{\pi\,
2^{2-2\bt}\Gf(2\bt-1)}{\Gf(\bt)^2}\, \Ph_{0,\infty}(0,\bt)\,,\\
D_{0,\infty}(\al,\bt)&\= \frac{\pi\,
2^{2-2\bt}\Gf(2\bt-1)}{\Gf(\bt)^2}\, \Ph_{0,\infty}(\al,\bt)\,, \ead
with an absolutely converging series
\be \label{Phseries}
\Ph_{0,\infty}(\al,\bt) \= \sum_{c>0 } c^{-2\bt}\sum_{d\bmod c}
\ch_\al(\gm_g)^{-1}\, e^{2\pi i \al a/c}\,. \ee
The variables run over $g=\matc ab cd\in \PSL_2(\RR)$ such that
$\gm_g = \matc{-c/2}{-d/2}{2a}{ 2b}= \matc ab c d g_0^{-1} \in \Gm$.
Right multiplication of $g $ by $\pi_\infty$ does not change
$\frac ac$, and multiplies $\gm_g$ by $\pi_0$ on the right; hence it
leaves $\ch_\al(\gm_g)$ invariant as well. So the terms are functions
of $d\bmod c$. Left multiplication of~$g$ with
$\pi_\infty=\matc 1101$ corresponds to left multiplication of $\gm_g$
by $\pi_\infty$, and gives a factor $e^{-2\pi i \al}$ in
$\ch_\al(\gm_g)^{-1}$. This is compensated by $e^{2\pi i \al a/c}$
since $a$ is changed to $a+c$. So the terms in the sum are well
defined.

For $\al=0$ we find
\begin{align*}
\Ph(0,\bt) &\= \Ph_{0,\infty}(0,\bt)
\=\sum_{p>0,\, p\equiv 1\bmod 2} (2p)^{-2\bt} \sum_{q\bmod p,\,
(q,p)=1} 1\\
&\= \frac1{(1-2^{-2\bt})\, \z(2\bt)}\sum_{n>0,\, n\equiv 1\bmod 2}
2^{-2\bt} \, n^{-2\bt} \,n
\\
&\= \frac1{2^{2\bt}-1} \frac{\z(2\bt-1)}{\z(2\bt)}\,(1-2^{1-2\bt})\,.
\end{align*}
Together with~\eqref{CGf} we get the value in~\eqref{sm0-expl}.

For most values of $\al$ in $(0,1)$ we cannot get explicit expressions
for $\Ph(\al,\bt)=\Ph_{0,\infty}(\al,\bt)$. However, we can conclude
that $\al\mapsto \Ph(\al,\bt)$ is continuous in $\al\in \RR$, and
with \eqref{Om-est} even twice differentiable in~$\al$ for
$\re \bt>\frac 52$.

Let us write $\gm_g = \matc r s {-4q}p$. The inner sum
in~\eqref{Phseries} for a given $p$ has the form
\[ \sum_{q\bmod p,\, (q,p)=1} e^{-i\al\,\Om(r,s)-\pi i \al r/2q}\,.\]
Conjugation with $\matr100{-1}$ has the effect
$\matc r s {-4q} p \mapsto \matc r{-s}{4q}p$. Under this
transformation the set of $q\bmod p$, $(q,p)=1$ is unchanged, and
$\Om(r,s)\mapsto \Om(r,-s) =
-\Om(r,s)$ and $\frac r{2q}\mapsto -\frac r{2q}$. This implies
\be \frac{\partial\Ph}{\partial\al}(0,\bt) \= 0 \qquad\text{ for
}\re\bt>\frac 32\,. \ee
Expanding one step further, we conclude that
$\frac{\partial^2\Ph}{\partial\al^2}(0,\bt)$ is given by a Dirichlet
series with negative coefficients that converges absolutely on the
region $\re \bt>\frac 52$. Going to smaller and smaller right
half-planes, we conclude that the odd derivatives of $\Ph_{0,\infty}$
with respect to~$\al$ vanish on some half-plane, and that the even
derivatives are given on a right half-plane by a Dirichlet series
with real coefficients.

For the other elements of the matrix $\dscm(\al,\bt)$ we can carry out
similar computations. (See \S4.2 in~\cite{Br94}.) If
$\x,\eta\in \{\infty,-\frac12\}$ these elements all have the form
\[ D_{\eta,\x}(\al,\bt) \=\frac{\pi^\bt \al^{\bt-1}}{\Gf(\bt) }\,
\Ph_{\eta,\x}(\al,\bt)\,,\]
with $\Ph_{\eta,\x}(\al,\bt)$ holomorphic in~$\bt$ on the region
$\re \bt>1$ and continuous in~$\al\in \RR$. It is given by a more
complicated series than that for $\Ph_{0,\infty}$, since the factor
$c^{-2\bt}$ is replaced by a more complicated expression, with Bessel
functions. The element $D_+(\al,\bt)$ in position $(\infty,\infty)$
in the matrix $\Um \, \dscm(\al,\bt)\, \Um^{-1}$ has the same
properties:
\[ D_+(\al,\bt) \= \frac{\pi^\bt \al^{\bt-1}}{\Gf(\bt) }\,
\Ph_+(\al,\bt)\,,\]
where $\Ph_+$ is given by a similar sum.

In \eqref{UDUi} the function $D_+$ is seen to be equal to $C_+\tilde
v^{-1}(1-\nobreak X C_+)^{-1}$. This implies
\begin{align*}
C_+(\al,\bt) &\= \frac{v(\al,1-\bt) \, D_+(\al,\bt)}{1+ v(\al,1-\bt)
\, X(\al,\bt)\, D_+(\al,\bt)}\\
&\= \frac{\pi^{1/2}\Gf(\bt-\frac12)\,
\Ph_+(\al,\bt)}{\Gf(\bt)+\pi^{1/2}(\pi\al)^{2\bt-1}\Gf(\frac12-\bt)\,
\Ph_+(\al,\bt)}\,.
\end{align*}
So on the region $\re\bt>\frac 32$, $\im\bt>0$, we have as
$\al\downarrow 0$
\[ C_+(\al,\bt) \= \frac{\pi^{1/2}\Gf(\bt-\frac12)\, \oh(1)}{\Gf(\bt)
+ \oh(\al^{2\bt-1})} \= \oh(1)\,. \]

This we use for $\re \bt>\frac {k+1}2$, with $k\in \NN$.
\begin{align*}
C_{0,\infty}(\al,\bt)&\=
(1-X(\al,\bt)C_+(\al,\bt)) \, D_{0,\infty}(\al,\bt)\\
&\= \bigl( 1+ \oh(\al^{2\bt-1})\bigr) \frac{\pi \,
2^{2-2\bt}\,\Gf(2\bt-1)}{\Gf(\bt)^2 }\, \Ph_{0,\infty}(\al,\bt)\\
&\= \frac{\pi^{1/2} \,\Gf(\bt-\frac12)}{\Gf(\bt)}\,\biggl( \sum_{n\geq
0,\; n<k}\frac{\al^n}{n!}\,
\frac{\partial^n\Ph_{0,\infty}}{\partial\al^n}(0,\bt) +
\oh\bigl(\al^k\bigr)
\biggr)\,.
\end{align*}
This implies that
$\frac{\partial^n C_{0,\infty}}{\partial\al^n}(0,\bt)=0$ on the
region $\re \bt>\frac{n+1}2$ for $n$ odd, and that the even
derivatives are given by a Dirichlet series with real coefficients.
Since $C_{0,\infty}$ is a meromorphic function with a non-trivial
restriction to the line $\{0\}\times\CC$ these results lead to the
following statement:
\begin{prop}\label{prop-deri0}
We have the following equalities of meromorphic functions on~$\CC$:
\bad C_{0,\infty}(0,\bt) &\= \sqrt{\pi} \,
\frac{\Gf(\bt-\frac12)}{\Gf(\bt)}\, \frac{1-2^{1-2\bt}}{2^{2\bt}-1}\,
\frac{\z(2\bt-1}{\z(2\bt)}\,,\\
\frac{\partial^n C_{0,\infty}}{\partial\al^n}(0,\bt) &\= 0\qquad\text{
for odd }n\,,\\
\frac{\partial^n C_{0,\infty}}{\partial\al^n} (0,\bt) &\= \sqrt{\pi}
\, \frac{\Gf(\bt-\frac12)}{\Gf(\bt)}\,
\frac{\partial^n\Ph_{0,\infty}}{\partial\al^n}(0,\bt)\qquad\text{ for
even }n\,. \ead
For even $n$ the function $
\frac{\partial^n\Ph_{0,\infty}}{\partial\al^n}(0,\bt)$ is represented
by a Dirichlet series with negative coefficients that converges on
the region $\re \bt>\frac{n+1}2$.

In particular, $C_{\infty,0}(\al,\bt)$ and $C_{0,\infty}(\al,\bt)$ are
equal and even functions of~$\al$.
\end{prop}

\subsection{Concluding remarks}\label{sect-concl}
In \S\ref{sect-spth} we have explained those parts of the theory of
automorphic forms that provide the facts that we needed to explain
some observations in the computational results in~\cite{Fr}.

\rmk The essential point enabling us to get some hold on the behavior
of the zeros of the Selberg zeta-function seems to be relation
\eqref{UDUi}, which relates the Eisenstein series and Poincar\'e
series for non-trivial $\ch_\al$ to those for the trivial character.

In all cases in \S\ref{sect-prfs} we have an equation where
$(\pi\al)^{2\bt-1}$ is equal to some meromorphic function on a
neighborhood of $\{0\}\times\CC$ in~$\CC^2$. This causes a
proportionality relation between $\al$ and $e^{-\pi k /t}$ in many
cases.

\rmk The full result in Proposition~\ref{prop-deri0} is not needed for
the proofs. We could have managed with the estimate $\oh(\al^2)$ for
the matrix elements of the extended scattering matrix.

\rmk All zeros of the Selberg zeta-function with the spectral
parameter on the central line that are visible in the computations
in~\cite{Fr} are related to properties of the extended scattering
matrix. The spectral theory of automorphic forms allows the existence
of cusp forms $f\in \mf^0(\al,\bt)$ for which the Fourier terms
$F^0_0f$, $F^\infty_\al f$ and $F^{-1/2}_{-\al} f$ vanish. The
presence of such cusp forms has not been detected in the
computations.

\rmk In Theorem~\ref{thm-ei-0} we have stated that the functions
$\tau_k$ are defined on an interval $(0,\z_k)\subset (0,1)$.
Actually, one can prove, that the families of cusp forms associated
with the eigenvalue $\al\mapsto \frac14+\tau_k(\al)^2$ are
real-analytic on $(0,1)$. They belong to a so-called Kato basis.
Compare \S2 of~\cite{PS94}.

\rmk \label{rmk-odd}As discussed in \S\ref{sect-mf} all automorphic
forms for $\Gm_0(4)$ with the family of characters
$\al\mapsto \ch_\al$ split completely in an even and an odd part. The
zeros of the Selberg zeta-function are related to eigenfunctions of a
transfer operator, to which also a parity is associated. In \cite{FM}
it is shown that this parity corresponds to the parity of automorphic
forms. It turns out that zeros of the Selberg zeta-function
in~\S\ref{sect-cei} are odd, and those in~\S\ref{sect-cres} even.

\rmk All odd cusp form observed in the computations occur in families
on
(an interval contained in) $(0,1)$ and have no real-analytic extension
across $\al=0$. Such an extension would be allowed by the theory, and
would give at $\al=0$ an unperturbed odd cusp form. All odd
unperturbed cusp forms inferred from the computations do not occur in
such families, but make their influence noticeable by the phenomenon
of avoided crossing.

\rmk All inferred even cusp forms for a nontrivial character,
$\al\in (0,1)$, occur discretely as cusp form. Their parameters
$(\al,\bt)$ occur on a curve of resonances, where it touches the
central line.

The limit point $(0,\bt_0)$ for a curve of resonances as
$\al\downarrow 0$ is equal to $(0,\frac12)$ for countably many
curves. All inferred unperturbed cusp forms with parameters
$(0,\bt_0)\neq(0,\frac12)$ are approached by a curve of resonances.
Such a curve describes infinitely many loops, giving rise to a
sequence $(\al_k,\bt_k) \rightarrow (0,\bt_0)$ of parameters of even
perturbed cusp forms.

\rmk The considerations in this paper concern a special situation,
namely the cofinite discrete subgroup $\Gm_0(4)$ of $\PSL_2(\RR)$ and
the $1$-parameter family of characters $\al\mapsto \ch_\al$. We have
tried to make use of all special properties of this specific
situation that we could obtain. It remains to be investigated how
much of the results of this paper are valid more generally.
Computations
in~\cite{Fr} indicate that for $\Gm_0(8)$ similar phenomena occur.

%%%%%%%%%%%%%%%%%%%%%%%%%%%%%%%%%%%%%%%%%%%%%%%%%%%%%%%%%%%
{\raggedright
\newcommand\bibit[4]{\bibitem %[#1]
  {#1}#2: {\em #3;\/ } #4}
% 1: label
% 2: author(s)
% 3: title
% 4: remaining part of reference
}

\end{document}